\documentclass[11pt,reqno]{amsart}    
\usepackage{mathrsfs}
\usepackage{setspace}
\usepackage{tikz}

\usetikzlibrary{matrix,arrows,shapes,calc}

\usepackage[margin=1in]{geometry}
\usepackage{amsfonts,amsmath,amssymb,enumerate}    
\usepackage{amsthm}    

\onehalfspace

\numberwithin{equation}{section}    

\theoremstyle{plain}
\newtheorem{thm}{Theorem}[section]
\newtheorem{lem}[thm]{Lemma}
\newtheorem{prop}[thm]{Proposition}
\newtheorem{cor}[thm]{Corollary}
\theoremstyle{definition}
\newtheorem{defn}[thm]{Definition}

\theoremstyle{remark}
\newtheorem{rem}[thm]{Remark}

\newtheorem*{rem*}{Remark}
\newtheorem*{ack}{Acknowledgements}

\newcommand{\bs}{\boldsymbol}

\newcommand{\T}{{M}}    
\newcommand{\Ti}{{m}}    
\newcommand{\Tii}{{n}}    
\newcommand{\be}{\begin{equation}}    
\newcommand{\ee}{\end{equation}}    
\newcommand{\beu}{\begin{equation*}}    
\newcommand{\eeu}{\end{equation*}}    
\newcommand{\bea}{\begin{eqnarray}}    
\newcommand{\eea}{\end{eqnarray}}    
\newcommand{\beaa}{\begin{eqnarray*}}    
\newcommand{\eeaa}{\end{eqnarray*}}    
\newcommand{\bmx}{\begin{pmatrix}}    
\newcommand{\emx}{\end{pmatrix}}

\newcommand{\g}{{\mathfrak g}}    
\newcommand{\h}{{\mathfrak h}}    
\newcommand{\m}{{\mathfrak m}}

\newcommand{\Pp}{\mathcal P^{+}}

\newcommand{\mf}{\mathfrak}
\newcommand{\mc}{\mathcal}    
\newcommand{\al}{{\alpha}}    
\newcommand{\gh}{{\widehat \g}}

\newcommand{\alf}{{\textstyle{\frac{1}{2}}}}

\newcommand{\nn}{\nonumber}

\newcommand{\8}{{\infty}}

\newcommand{\eps}{\epsilon}

\newcommand{\Z}{{\mathbb Z}}

\renewcommand{\P}{{\mathcal P}}
\newcommand{\Q}{{\mathcal Q}}    
\newcommand{\R}{{\mathbb R}}

\newcommand{\ket}[1]{{\,\left|#1\right>}\,}

\newcommand{\id}{{\mathrm{id}}}

\newcommand{\uq}{{U_q}}
\newcommand{\uqsl}[1]{{U_{q^{r_{#1}}}(\widehat{\mathfrak{sl}}_2{}^{(#1)})}}    
\newcommand{\uqslp}[1]{{U_{q^{r_{#1}}}(\widehat{\mathfrak{sl}}_2)}}    
    
\newcommand{\uqgh}{{\uq(\widehat\g)}}    

\newcommand{\uqg}{{\uq(\g)}}
\newcommand{\uqbt}{{\uq(\mf{b}_2)}}

\newcommand{\YY}[2]{Y^{}_{#1, #2}}    
\newcommand{\MM}[2]{Y^{-1}_{#1, #2}}

\newcommand{\goi}[2]{=}

\newcommand{\It}{\mathcal X}
\newcommand{\Iw}{\mathcal W}
\newcommand{\Iy}{\mathcal Y}

\newcommand{\on}{}    

\newcommand{\groth}[1]{{\mathrm{Rep}(#1)}}    
\newcommand{\Cx}{\mathbb C^*}
\newcommand{\qnum}[1]{\left[ #1\right]_q}
\renewcommand{\binom}[2]{\begin{bmatrix} #1 \\ #2 \end{bmatrix}}
\newcommand{\qbinom}[2]{\binom{#1}{#2}_q}
    
\newcommand{\btp}{\begin{tikzpicture}[baseline=0pt,scale=0.9,line width=0.25pt]}    
\newcommand{\etp}{\end{tikzpicture}}

\newcommand{\Zys}{\Z\!\left[ Y_{i,a}^{\pm 1} \right]_{i\in I, a \in \Cx}}
\newcommand{\ma}[1]{\mathbb D_{#1}}

\renewcommand{\L}{L}

\newcommand{\scr}{\mathscr}

\newcommand{\nbri}{\mathbb X}
\newcommand{\nbrii}{\mathbb Y}
\newcommand{\snk}{{(i_t,k_t)_{1\leq t\leq \T}}}
\newcommand{\atp}[1]{}

\newcommand{\mon}{\mathsf m}

\newcommand{\path}{\longrightarrow}
\newcommand{\confer}{c.f. } 
\newcommand{\mchiq}{\scr M}
\newcommand{\nops}{\overline{\scr P}_{(i_t,k_t)_{1\leq t\leq \T}}}
\DeclareMathOperator{\concat}{\#}
\newcommand{\ps}{{(p_1,\dots,p_\T)}}
\newcommand{\pps}{{(p'_1,\dots,p'_\T)}}

\newcommand{\ptop}{{\left(\phigh_{i_1,k_1},\dots,\phigh_{i_\T,k_\T}\right)}}

\newcommand{\phigh}{p^{+}}
\newcommand{\plow}{p^{-}}
\newcommand{\psnake}{p^{\mathrm{snake}}}

\DeclareMathOperator{\wt}{wt}

\DeclareMathOperator{\res}{res}
\DeclareMathOperator{\trunc}{trunc}
\DeclareMathOperator{\Span}{span}

\DeclareMathOperator{\hgt}{hgt}

\DeclareMathOperator{\bott}{bot}

\newcommand{\emptyseq}{\emptyset}
\newcommand{\WQ}[1]{\mathbf Q^{#1}}

\newcommand{\W}[3]{\mathbf{T}_{#1,#2}^{#3}}
\newcommand{\Wt}[3]{\widetilde{\mathbf{T}}_{#1,#2}^{#3}}

\newcommand{\Ww}[3]{\overline{\mathbf{T}}_{#1,#2}^{#3}}
\newcommand{\st}[3]{\pi_{#1,#2}^{#3}}

\author{E. Mukhin} 
\address{\vspace{-.15cm} Department of Mathematical Sciences, 402
  N. Blackford St, LD 270, IUPUI, Indianapolis, IN 46202, USA.  }
\email{mukhin@math.iupui.edu}

\author{C. A. S. Young}

\address{\vspace{-.15cm} Department of Mathematics, University of York, Heslington, York YO105DD, UK. (Present address)\\
and Yukawa Institute for Theoretical Physics, Kyoto University,
  Kyoto, 606-8502, Japan.}  \email{charlesyoung@cantab.net}
\date{June 2011}

\begin{document} 
\title{Extended T-systems}
\begin{abstract}
We use the theory of $q$-characters to establish a number of short exact sequences 
in the category of finite-dimensional representations of the quantum affine groups of types A and B. That allows us to introduce a set of 3-term recurrence relations which contains  the celebrated T-system as a special case.
\end{abstract}

\maketitle

\section{Introduction}
The T-systems are important sets of recurrence relations which have many applications in integrable systems. The literature on the subject is vast: we refer 
the reader to the survey \cite{KNSrev} and references therein.

Originally, the T-systems were introduced as a family of relations in
the Grothendieck ring of the category of the finite-dimensional
modules of the Yangians and quantum affine algebras \cite{KR, KNS,
  NakajimaKR, HernandezKR}. More precisely, the T-systems 
correspond to a family of short
exact sequences of tensor products of Kirillov-Reshetikhin (KR) modules
\cite{HernandezKR}. 
The knowledge of the T-system is one of the main reasons the KR modules are comparatively well understood.

In this paper we argue that other classes of finite-dimensional modules of quantum affine algebras can be studied using recursions similar to the T-systems. We call such recursions  {\it extended T-systems}. The sense in which they generalize the usual T-system is perhaps most rapidly understood by examining Figures \ref{figA} and \ref{figB}, in which simple examples of T-system relations (above) and relations in the extended system (below) can be compared. Let us begin by describing the construction of these recursion relations.

The irreducible finite-dimensional modules of affine quantum groups
are parameterized by their highest $l$-weights or, equivalently, their Drinfeld polynomials. 
Given an irreducible module $T$, in many cases there is a natural rightmost and leftmost zero of the set of Drinfeld polynomials of $T$. 
We call the irreducible module corresponding to the Drinfeld polynomials of $T$ with
the rightmost (resp. leftmost) zero removed the left (resp. right) module, $L$ (resp. $R$). 
We call the irreducible module corresponding to the Drinfeld polynomials of $T$ with
both rightmost and leftmost zeros removed the bottom module, $B$. We call $T$ the top module. Obviously, the modules $L\otimes R$ and $T\otimes B$ have the same highest
$l$-weights.

We find that in many cases $T\otimes B$ is in fact irreducible, and the difference $L\otimes R- T\otimes B$ in the Grothendieck ring is also a class of an irreducible module. In fact this difference is even special, meaning that it has a unique dominant $l$-weight. We then proceed to factor the difference into a product of prime irreducible modules $N_1$ and $N_2$, which we call neighbours. Therefore we obtain a short exact sequence
\be\label{TTT}
0\to T\otimes B \to L \otimes R \to  N_1\otimes N_2\to 0,
\ee
which can be used to express module $T$ in the Grothendieck ring via smaller (in a natural order) modules. We note that unfortunately, in general we do not know which of the two coproducts corresponds to the choice of the arrows in the short exact sequence we made here.

\medskip

Our starting point in types A and B is the class of minimal affinizations (MA). The MA \cite{CPminaffBCFG,CPminaffireg,CPminaffADE} form an important class of irreducible modules of affine quantum groups which are the closest possible analogs of the evaluation modules. The KR modules are simplest examples of MA, corresponding to highest weights which are multiples of a fundamental weight. 

We find that the extended T-system closes among MA in type A. That means that if the top module $T$ is a MA then all other modules in (\ref{TTT}) are also MA.

We turn next to type B. Here we find that 
our extended T-system does not close in the class of MA (except for the case of $B_2$). Namely if $T$ is a MA, the modules $L,R,B$ are also  MA, but $N_1$ and $N_2$ in general are not. Therefore we are forced to consider a slightly more general class of modules, which we call ``wrapping'' modules, where the extended T-system does close. The wrapping modules seem to have properties similar to MA, at least from the combinatorial point of view. The extended T-system is then a recursion relation which allows one to compute, in particular, the wrapping modules via fundamental representations.

We proceed to extend our T-system to yet a larger class of modules -- the snake modules, introduced in \cite{MY1}. 
In type A, the snake modules are just modules related to skew Young diagrams, \cite{NT}. In type B, the modules related to skew Young diagrams \cite{KOS} form a subset of snake modules. The term ``snake'' is meant to be suggestive of the pattern formed by the zeros of the Drinfeld polynomials of such modules, c.f. Figure \ref{snakefig}. 

\medskip
In this paper we work in types A and B only, where MA, and more generally snake modules, are thin and special in the terminology of $q$-characters. Thin means that the Cartan part of the quantum affine algebra acts in a semi-simple way. This allows one to compute their $q$-characters explicitly, which was done \cite{NT, FMgl8} in type A and recently \cite{MY1} in type B. Then the problem of the existence of short exact sequences can be reduced to combinatorics.


A natural question is whether extended T-systems exist in all types, as is the case with the usual T-system. Computations suggest that they do: we give some illustrative examples in Appendix \ref{typeCD}. However, it remains a challenge to identify a suitable class of  representations -- which one would like to include all minimal affinizations -- and to furnish the necessary proofs. Note that minimal affinizations in other types are not thin \cite{HernandezMinAff} in general, which makes the analysis more difficult.

\medskip

The explicit form of several instances of the extended T-system is written in Section \ref{Tsyssec}. For example, the extended T-system for evaluation modules in type $A_N$ is  an equation for the functions $T_{i,k}^{m_1,\dots,m_s}$, $s=1,\dots, N$, where $k\in\Z$, $m_j\in\Z_{\geq 0}$, $i=1,\dots,N+1-s$ :
\be
T_{i,k}^{m_1,\dots,m_s-1}T_{i,k+2}^{m_1-1,\dots,m_s}=
T_{i,k}^{m_1,\dots,m_s} T_{i,k+2}^{m_1-1,\dots, m_s-1}
+T_{i-1,k+1}^{m_1,\dots,m_s-1}T_{i+1,k+1}^{m_1-1,\dots,m_s},\nn
\ee
with some natural boundary conditions. In particular, for $s=1$, the functions 
$T_{i,k}^m$ satisfy the usual T-system.

The T-systems are functional equations which are widely studied and very important in many parts of mathematics and physics.  We hope that many 
properties and applications of the T-systems can be generalized to the
extended T-systems. 
For example, we expect that the extended T-system has a restricted version,  similarly to the usual T-system \cite{IIKNS}.

\medskip

As a first application of the extended T-systems we consider the corresponding extended Q-system of type $B_2$. We use it to compute the decomposition of $B_2$ wrapping modules after restriction to the finite quantum group, generalizing the results of \cite{Cminaffrank2}.

\medskip

In types ADE, there is a recent remarkable conjecture on the cluster algebra relations in the category of finite-dimensional representations of affine quantum groups \cite{HernandezLeclerc} which includes the T-system. We should like to think that similar to the T-systems, the extended T-systems would provide a large family of explicit cluster relations in the cluster algebra.

\medskip

The paper is organized as follows. Section 2 contains background material. In Section 3 we recall the definition of snake modules, and define the neighbours of a prime snake. Then in Section 4 we state our main result, Theorem 4.1. The remainder of Section 4 exhibits various special cases of this theorem. In Section 5 we apply our result to compute the dimensions and $\uqg$-decompositions of wrapping modules in type $B_2$. To prove Theorem 4.1, we first recall in Section 6 the machinery of 
paths and moves from \cite{MY1}; then the proof itself is in Section 7. 
Appendix \ref{typeCD} is devoted to examples of 3-term relations in types C and D. Finally, in Appendix \ref{proofA}, we prove a theorem on thin, special, truncated $q$-characters which we need in Section 7.

\medskip
While this paper was in preparation, D. Hernandez and B. Leclerc informed
us that motivated by the cluster algebra conjecture they have also proved
a number of 3 term relations in types A and B.

\medskip
 
\begin{ack} 
We would like to thank D. Hernandez, B. Leclerc, V. Tarasov for interesting discussions.
EM would like to thank D. Hernandez for his hospitality during a visit to Paris in Summer 2010.
CY would like to thank IUPUI Department of Mathematics for hospitality during his visit
in Fall 2010 when part of this work was carried out. The research of CASY was funded by a Postdoctoral Research Fellowship (grant number P09771) from the Japan Society for the Promotion of Science (until end Oct 2010) and by the EPSRC (from Nov 2010, grant number EP/H000054/1). 
The research of EM is supported by the NSF, grant number DMS-0900984.
Computer programs to calculate $q$-characters were written in FORM \cite{FORM}.
\end{ack}

\section{Background}\label{qcharsec}
\subsection{Cartan data}
Let $\g$ be a complex simple Lie algebra of rank $N$ and $\h$ a Cartan subalgebra of $\g$. We identify $\h$ and $\h^*$ by means of the invariant inner product $\left<\cdot,\cdot\right>$  on $\g$ normalized such that the square length of the maximal root equals 2.  Let $I=\{1,\dots,N\}$ and let $\{\alpha_i\}_{i\in I}$ be a set of simple roots, with $\{\alpha^\vee_i\}_{i\in I}$ and $\{\omega_i\}_{i\in   I}$,
the sets of, respectively, simple coroots and fundamental weights. 
Let $C=(C_{ij})_{i,j\in I}$ denote the Cartan matrix. We have 
\be\nn 2 \left< \alpha_i, \alpha_j\right> = C_{ij} \left<   \alpha_i,\alpha_i\right>,\quad 2 \left< \alpha_i, \omega_j\right> = \delta_{ij}\left<\alpha_i,\alpha_i\right> .\ee Let $r^\vee$ be the maximal number of edges connecting two vertices of the Dynkin diagram of $\g$. Thus $r^\vee=1$ if $\g$ is of types A, D or E, $r^\vee = 2$ for types B, C and F and $r^\vee=3$ for $\mathrm G_2$.  Let $r_i= \alf r^\vee \left<\alpha_i,\alpha_i\right>$. The numbers $(r_i)_{i\in   I}$ are relatively prime integers. We set \be\nn D:= \mathrm{diag}(r_1,\dots,r_N),\qquad B := DC;\ee the latter is the symmetrized Cartan matrix, $B_{ij} = r^\vee \left<\alpha_i,\alpha_j\right>$.
 
Let $Q$ (resp. $Q^+$) and $P$ (resp. $P^+$) denote the $\Z$-span (resp. $\Z_{\geq 0}$-span) of the simple roots and fundamental weights respectively. Let $\leq$ be the partial order on $P$ in which $\lambda\leq \lambda'$ if and only if $\lambda'-\lambda\in Q^+$. 
 
Let $\gh$ denote the untwisted affine algebra corresponding to $\g$. 

Fix a $q \in \Cx$, not a root of unity. 
Define 
the $q$-numbers, $q$-factorial and $q$-binomial: \be\nn \qnum n := \frac{q^n-q^{-n}}{q-q^{-1}},\quad \qnum n ! := \qnum n \qnum{n-1} \dots \qnum 1,\quad \qbinom n m := \frac{\qnum n !}{\qnum{n-m} ! \qnum m !}.\ee

\subsection{Quantum Affine Algebras}
The \emph{quantum affine algebra} $\uqgh$ in Drinfeld's new realization, \cite{Drinfeld} is generated by $x_{i,n}^{\pm}$ ($i\in I$, $n\in\Z$), $k_i^{\pm 1}$ ($i\in I$), $h_{i,n}$ ($i\in I$, $n\in \Z\setminus\{0\}$) and central elements $c^{\pm 1/2}$, subject to the following relations:
\begin{align}
  k_ik_j = k_jk_i,\quad & k_ih_{j,n} =h_{j,n}k_i,\nn\\
  k_ix^\pm_{j,n}k_i^{-1} &= q^{\pm B_{ij}}x_{j,n}^{\pm},\nn\\
 \label{hxpm} [h_{i,n} , x_{j,m}^{\pm}] &= \pm \frac{1}{n} [n B_{ij}]_q c^{\mp
    {|n|/2}}x_{j,n+m}^{\pm},\\ 
x_{i,n+1}^{\pm}x_{j,m}^{\pm} -q^{\pm B_{ij}}x_{j,m}^{\pm}x_{i,n+1}^{\pm} &=q^{\pm
    B_{ij}}x_{i,n}^{\pm}x_{j,m+1}^{\pm}
  -x_{j,m+1}^{\pm}x_{i,n}^{\pm},\nn\\ [h_{i,n},h_{j,m}]
  &=\delta_{n,-m} \frac{1}{n} [n B_{ij}]_q \frac{c^n -
    c^{-n}}{q-q^{-1}},\nn\\ [x_{i,n}^+ , x_{j,m}^-]=\delta_{ij} & \frac{
    c^{(n-m)/2}\phi_{i,n+m}^+ - c^{-(n-m)/2} \phi_{i,n+m}^-}{q^{r_i} -
    q^{-r_i}},\nn\\
  \sum_{\pi\in\Sigma_s}\sum_{k=0}^s(-1)^k\left[\begin{array}{cc} s \nn\\
      k \end{array} \right]_{q^{r_i}} x_{i, n_{\pi(1)}}^{\pm}\ldots
  x_{i,n_{\pi(k)}}^{\pm} & x_{j,m}^{\pm} x_{i,
    n_{\pi(k+1)}}^{\pm}\ldots x_{i,n_{\pi(s)}}^{\pm} =0,\ \
  s=1-C_{ij},\nn
\end{align}
for all sequences of integers $n_1,\ldots,n_s$, and $i\ne j$, where $\Sigma_s$ is the symmetric group on $s$ letters, and $\phi_{i,n}^{\pm}$'s are determined by the formula
\begin{equation} \label{phidef} \phi_i^\pm(u) :=
  \sum_{n=0}^{\infty}\phi_{i,\pm n}^{\pm}u^{\pm n} = k_i^{\pm 1}
  \exp\left(\pm(q-q^{-1})\sum_{m=1}^{\infty}h_{i,\pm m} u^{\pm
      m}\right).
\end{equation}
There exist a coproduct, counit and antipode making $\uqgh$ into a Hopf algebra. 

The subalgebra of $\uqgh$ generated by $(k_i)_{i\in I}$, $(x^\pm_{i,0})_{i\in I}$ is a Hopf subalgebra of $\uqgh$ and is isomorphic as a Hopf algebra to $\uqg$, the quantized enveloping algebra of $\g$.
In this way, $\uqgh$-modules restrict to $\uqg$-modules.



\subsection{Finite-dimensional representations and $q$-characters}\label{ssec:fdreps}
A representation $V$ of $\uqgh$ is \emph{of type $1$} if $c^{\pm 1/2}$ acts as the identity on $V$ and 
\be V = \bigoplus_{\lambda\in P} V_\lambda\,\,\,,\qquad\quad 
V_\lambda = \{ v \in V: k_i \on v = q^{\langle \alpha_i, \lambda \rangle} v\}.\label{gwts}\ee 
In what follows, all representations will be assumed to be of type 1 without further comment.
The decomposition (\ref{gwts}) of a finite-dimensional representation $V$ into its $\uqg$-weight spaces can be refined by decomposing it into the Jordan subspaces of the mutually commuting $\phi_{i,\pm r}^\pm$ defined in (\ref{phidef}),
\cite{FR}: 
\be V = \bigoplus_{\bs\gamma} V_{\bs\gamma}\,\,, \qquad \bs\gamma = (\gamma_{i,\pm r}^\pm)_{i\in I, r\in \Z_{\geq 0}}, \quad \gamma_{i,\pm r}^\pm \in \mathbb C\label{lwdecomp}\ee 
where \be V_{\bs\gamma} = \{ v \in V : \exists k \in \mathbb N, \,\, \forall i \in I, m\geq 0 \quad \left(
  \phi_{i,\pm m}^\pm - \gamma_{i,\pm m}^\pm\right)^k \on v = 0 \} \,. \nn\ee 
If $\dim (V_{\bs\gamma}) >0$, $\bs\gamma$ is called an \emph{$l$-weight} of $V$. For every finite-dimensional representation of $\uqgh$, the $l$-weights are known \cite{FR} to be of the form 
\be\nn \gamma_i^\pm(u) := \sum_{r =0}^\8 \gamma_{i,\pm r}^\pm u^{\pm r} 
 = q^{r_i\deg Q_i - r_i\deg R_i} \, \frac{Q_i(uq^{-r_i}) R_i(uq^{r_i})}{Q_i(uq^{r_i}) R_i(uq^{-r_i})} \,,\ee 
where the right hand side is to be treated as a formal series in positive (resp. negative) integer powers of $u$, and $Q_i$ and $R_i$ are polynomials of the form 
\be\nn Q_i(u) = \prod_{a\in \Cx} \left( 1- ua\right)^{w_{i,a}}, \quad\quad 
       R_i(u) = \prod_{a\in \Cx} \left( 1- ua\right)^{x_{i,a}}, \ee 
for some $w_{i,a}, x_{i,a}\geq 0$, $i\in I,a\in \Cx$.  Let $\P$ denote the free abelian multiplicative group of  monomials in infinitely many formal variables $(Y_{i,a})_{i\in I,a\in \Cx}$. $\P$ is in bijection with the set of $l$-weights $\bs\gamma$ of the form above according to 
\be\label{lwdef} \bs\gamma=\bs\gamma(m) \quad\text{with}\quad m = \prod_{i \in I, a\in \Cx} Y_{i,a}^{w_{i,a}-x_{i,a}}.\ee
We identify elements of $\P$ with $l$-weights of finite-dimensional representations in this way, and henceforth write $V_m$ for $V_{\bs\gamma(m)}$.  Let $\Z\P=\Zys$ be the ring of Laurent polynomials in $(Y_{i,a})_{i\in I,a\in \Cx}$ with integer coefficients.
The $q$-character map $\chi_q$ \cite{FR} is then defined by 
\be \chi_q (V) = \sum_{m\in\P} \dim\left(V_m\right) m.\nn\ee
Let $\groth\uqgh$ be the Grothendieck ring of finite-dimensional representations of $\uqgh$, and let us write $\left[ V \right]\in \groth\uqgh$ for the class of a finite-dimensional $\uqgh$-module $V$. The $q$-character map defines an injective ring homomorphism \cite{FR} 
\be \chi_q : \groth\uqgh \longrightarrow \Zys. \nn\ee

For any finite-dimensional representation $V$ of $\uqgh$, we let
\be \mchiq(V) := \left\{ m\in \P: m \text{ is a monomial of } \chi_q(V) \right\}.\nn\ee 

For each $j\in I$, a monomial $m = \prod_{i \in I, a\in \Cx} Y_{i,a}^{u_{i,a}}$ is said to be \emph{$j$-dominant} (resp. \emph{$j$-anti-dominant}) if and only if $u_{j,a}\geq 0$ (resp. $u_{j,a}\leq 0$) for all $a\in\Cx$. A monomial is (\emph{anti-})\emph{dominant} if and only if it is $i$-(anti-)dominant for all $i\in I$. Let $\Pp\subset \P$ denote the set of all dominant monomials.

If $V$ is a finite-dimensional representation of $\uqgh$ and $m\in\mchiq(V)$ is dominant, then a non-zero vector $\ket{m}\in V_m$ is called a \emph{highest $l$-weight vector}, with \emph{highest $l$-weight} $\bs \gamma(m)$, if and only if 
\be  \phi^\pm_{i,\pm t}\on\ket m= \ket m\gamma(m)^\pm_{i,\pm t} \quad \text{ and }\quad x^+_{i,r} \on \ket m = 0,\quad\text{for all } i\in I, r\in \Z, t\in \Z_{\geq 0} .\nn\ee 
A finite-dimensional representation $V$ of $\uqgh$ is said to be a \emph{highest $l$-weight representation} if $V= \uqgh \on \ket m$ for some highest $l$-weight vector $\ket m\in V$.

It is known \cite{CPbook,CP94} that for each $m\in\P^+$ there is a unique finite-dimensional irreducible representation, denoted $\L(m)$, of $\uqgh$ that is highest $l$-weight with highest $l$-weight $\bs\gamma(m)$, and moreover every finite-dimensional irreducible $\uqgh$-module is of this form for some $m\in \P^+$. 

For each $m\in \P^+$, there exists a highest $l$-weight representation $W(m)$, called the \emph{Weyl module}, with the property that every highest $l$-weight representation of $\uqgh$ with highest $l$-weight $\bs\gamma(m)$ is a quotient of $W(m)$ \cite{CPweyl}. 

A finite-dimensional $\uqgh$-module $V$ is said to be \emph{special} 
if and only if $\chi_q(V)$ has exactly one dominant monomial. It is \emph{anti-special} if and only if $\chi_q(V)$ has exactly one anti-dominant monomial. 
It is \emph{thin} 
if and only if no $l$-weight space of $V$ has dimension greater than 1. In other words, the module is thin if and only if the $(\phi_{i,\pm r}^\pm)_{i\in I,r\in \Z_{\geq 0}}$ are simultaneously diagonalizable with joint simple spectrum. A finite-dimensional $\uqgh$-module $V$ is said to be \emph{prime} if and only if it is not isomorphic to a tensor product of two non-trivial $\uqgh$-modules \cite{CPprime}.  

Let $\chi:\groth\uqg \to \Z [e^{\pm \omega_i}]_{i\in I}$ be the $\uqg$-character homomorphism. Let $\wt:\P \to P$ be the homomorphism of abelian groups defined by $\wt: \YY i a \mapsto \omega_i$. The map $\wt{}$ induces in an obvious way a map $\Z\P\to \Z[e^{\pm \omega_i}]_{i\in I}$ which we also call $\wt$. Then the following diagram commutes \cite{FR}:
\be\begin{tikzpicture}    
\matrix (m) [matrix of math nodes, row sep=3em,    
column sep=4em, text height=2ex, text depth=1ex]    
{ \groth\uqgh &  \mathbb Z \P     \\    
\groth\uqg & \mathbb Z [e^{\pm \omega_i}]_{i\in I}     \\   };    
\path[->,font=\scriptsize]    
(m-1-1) edge node [above] {$\chi_q$} (m-1-2)    
(m-2-1) edge node [above] {$\chi$} (m-2-2)    
(m-1-1) edge node [left] {$\res$} (m-2-1)    
(m-1-2) edge node [right] {$\wt$} (m-2-2);    
\end{tikzpicture}\nn\ee    
where $\res:\groth\uqgh\to\groth\uqg$ is the restriction homomorphism.

Define $A_{i,a}\in \P$, $i\in I,a\in\Cx$, by 
\be\label{adef} A_{i,a} = Y_{i,aq^{r_i}} Y_{i,aq^{-r_i}} \prod_{C_{ji}=-1} Y_{j,a}^{-1} \prod_{C_{ji}=-2} Y_{j,aq}^{-1} Y_{j,aq^{-1}}^{-1} \prod_{C_{ji}=-3} Y_{j,aq^2}^{-1} Y_{j,a}^{-1} Y_{j,aq^{-2}}^{-1}.\ee 
Let $\Q$ be the subgroup of $\P$ generated by $A_{i,a}$, $i\in I,a\in\Cx$. Let $\Q^\pm$ be the monoid generated by $A_{i,a}^{\pm 1}$, $i\in I,a\in\Cx$. Note that $\wt A_{i,a} = \al_i$. There is a partial order $\leq$ on $\P$ in which $m\leq m'$ if and only if $m' m^{-1} \in \Q^+$. It is compatible with the partial order on $P$ in the sense that $m\leq m'$ implies $\wt m\leq \wt m'$.

We have \cite{FM} that for all $m_+\in \P^+$, 
\be \mchiq(\L(m_+)) \subset m_+ \Q^- \label{imchiq}.\ee

For all $i\in I, a\in \Cx$ let $u_{i,a}$
be the homomorphism of abelian groups $\P \to \mathbb Z$ such that
\be u_{i,a}(Y_{j,b})=\begin{cases} 1 & i=j \text{ and } a=b \\ 0 & \text{otherwise.}\end{cases}
\label{udef}\ee
Let 
$v$ be the homomorphisms of abelian groups $\Q \to \mathbb Z$ such that
\be 
  v(A_{j,b})=- 1 .\nn\ee
Note that the $(A_{i,a})_{i\in I,a\in \Cx}$ are algebraically independent, so 
$v$ is well-defined.

For each $j\in I$ we denote by $\uqsl j$ the copy of $\uqslp j$ generated by $c^{\pm1/2}, (x_{j,r}^\pm)_{r\in \Z}$, $(\phi_{j,\pm r}^\pm)_{r\in \Z_{\geq 0}}$. Let 
\be \beta_j : \mathbb Z \left[ Y^{\pm 1}_{i,a}\right]_{i\in I; a \in \Cx} \to \mathbb Z \left[ Y^{\pm 1}_{j,a}\right]_{a \in \Cx} \nn\ee     
be the ring homomorphism which  sends, for all $a\in\Cx$, $Y_{k,a}\mapsto 1$ for all $k \neq j$ and $Y_{j,a}\mapsto Y_{j,a}$.
For each $j\in I$, there exists \cite{FM} a ring homomorphism    
\be\tau_j : \Z\left[ Y^{\pm 1}_{i,a}\right]_{i\in I; a \in \Cx} \to \mathbb Z \left[ Y^{\pm 1}_{j,a}\right]_{a \in \Cx} \otimes \mathbb Z \left [ Z_{k, b}^{\pm 1}\right ]_{k \neq j; b \in \Cx},\nn\ee  
where $(Z_{k,b}^{\pm 1})_{k\neq j,b\in\Cx}$ are certain new formal variables, with the following properties:
\begin{enumerate}[i)] \item $\tau_j$ is injective. \item $\tau_j$ refines $\beta_j$ in the sense that $\beta_j$ is the composition of $\tau_j$ with the homomorphism $\mathbb Z [ Y^{\pm 1}_{j,a}]_{a \in \Cx} \otimes \mathbb Z [ Z_{k, b}^{\pm 1}]_{k \neq j; b \in \Cx}\to \mathbb Z [Y^{\pm 1}_{j,a}]_{a \in \Cx}$ which sends $Z_{k,b}\mapsto 1$ for all $k\neq j$, $b\in \Cx$.
\item In the diagram    
\be\label{tauj}\begin{tikzpicture}    
\matrix (m) [matrix of math nodes, row sep=3em,    
column sep=4em, text height=2ex, text depth=1ex]    
{     
\mathbb Z \left[ Y^{\pm 1}_{i,a}\right]_{i\in I; a \in \Cx}    
 & \mathbb Z \left[ Y^{\pm 1}_{j,a}\right]_{a \in \Cx} \otimes \mathbb Z \left [ Z_{k, b}^{\pm 1}\right ]_{k \neq j; b \in \Cx} \\    
\mathbb Z \left[ Y^{\pm 1}_{i,a}\right]_{i\in I; a \in \Cx}      
 & \mathbb Z \left[ Y^{\pm 1}_{j,a}\right]_{a \in \Cx} \otimes \mathbb Z \left [ Z_{k, b}^{\pm 1}\right ]_{k \neq j; b \in \Cx}\\    
};    
\path[->,font=\scriptsize,shorten <= 2mm,shorten >= 2mm]    
(m-1-1) edge node [above] {$\tau_j$} (m-1-2)    
(m-2-1) edge node [above] {$\tau_j$} (m-2-2)    
(m-1-1) edge (m-2-1)    
(m-1-2) edge (m-2-2);    
\end{tikzpicture}\ee    
let the right vertical arrow be multiplication by $\beta_j(A_{j,c}^{-1}) \otimes 1$; then the diagram commutes if and only if the left vertical arrow is multiplication by $A^{-1}_{j,c}$.
\end{enumerate}

\subsection{Truncated $q$-characters}
Given a set of monomials $\mc R\subset \P$, let $\Z\mc R$ denote the $\Z$-module of formal linear combinations of elements of $\mc R$ with integer coefficients.  Define
\be\trunc_{\mc R}: \P\to\mc R;\quad m \mapsto \begin{cases} m & m\in \mc R \\ 0 & m\notin \mc R\end{cases} \ee 
and extend $\trunc_{\mc R}$ as a $\Z$-module map $\Z \P\to \Z\mc R$. 

Given a subset $U\subset I\times \Cx$, let $\Q_U$ be the subgroup of $\Q$ generated by $A_{i,a}$, $(i,a)\in U$. Let $\Q^\pm_U$ be the monoid generated by $A^{\pm 1}_{i,a}$, $(i,a)\in U$. 
For each $n\in \Z_{\geq 0}$, let $\Q^\pm_{U,(=n)}$  be the set of monomials in the variables $A^{\pm 1}_{i,a}$, $(i,a)\in U$ of degree exactly $n$. Similarly, let $\Q^{\pm}_{U,(\leq n)}$ (resp. $\Q^\pm_{U,(>n)}$) be the set of monomials in the variables $A^{\pm 1}_{i,a}$, $(i,a)\in U$ of degree $\leq n$ (resp. $>n$).

Let us call $\trunc_{m_+ Q_U^{-1}} \chi_q(L(m_+))$ the $q$-character  of $L(m_+)$ \emph{truncated} to $U$. In this we follow D. Hernandez and B. Leclerc, who made use of $q$-characters truncated to sets of the form $I\times \{a,aq,\dots,aq^{r-1}\}$ in  \cite{HernandezLeclerc}.

Given $V$, a highest $l$-weight $\uqgh$-module with highest monomial $m_+\in\P^+$, and $U\subset I\times \Cx$, we say \emph{$V$ is thin in $U$} if and only if $\trunc_{m_+\Q^-_U}(\chi_q(V))$ has no monomial with multiplicity greater than 1. We say \emph{$V$ is special in $U$} if and only if $m_+$ is the unique dominant monomial of $\trunc_{m_+\Q^-_U}(\chi_q(V))$. 

In the proof of Proposition \ref{TBsimple} below we need the following result which gives sufficient conditions for a given set of monomials to be the truncation of a $q$-character. It is a generalized version of a similar result in \cite{MY1}.
\begin{thm}\label{thmA}
Let $U\subset I\times \Cx$. Let $m_+\in \P^+$. Suppose that $\mc M\subset \P$ is a finite set of distinct monomials such that:
\begin{enumerate}[(i)]
\item $\mc M \subset m_+  \Q^-_U$; \label{incone}
\item $\{m_+\} = \P^+ \cap \mc M$; \label{monlydom} 
\item For all $m\in \mc M$ and all $(i,a)\in U$, if $mA_{i,a}^{-1}\notin\mc M$ then $mA_{i,a}^{-1} A_{j,b}\notin \mc M$ unless $(j,b)=(i,a)$;\label{onewayback}
\item For all $m\in \mc M$ and all $i\in I$ there exists a unique $i$-dominant $M\in \mc M$ such that
\be\nn \trunc_{\beta_i(M \Q^-_U)}\chi_q(\L(\beta_i(M))) = \sum_{m'\in m\Q_{\{i\}\times \Cx}\cap \mc M} \beta_i(m') .\ee\label{inthinsimple}
\end{enumerate}
Then
\be \label{headch} \trunc_{m_+ \Q^-_U}\chi_q(\L(m_+)) = \sum_{m\in \mc M} m \ee 
and the module $\L(m_+)$ is thin in $U$ and special in $U$.
\end{thm}
\begin{proof} The proof is given in Appendix \ref{proofA}. \end{proof}

\subsection{Affinizations of $\uqg$-modules} \label{sec:minaff} 
For $\mu\in P^+$, let $V(\mu)$ be the (unique up to isomorphism) simple $\uqg$-module with highest weight $\mu$.
$\L(m)$, $m\in \P^+$, is said to be an \emph{affinization} of $V(\mu)$ if $\wt m=\mu$ \cite{Cminaffrank2}.
Two affinizations are said to be equivalent if they are isomorphic as $\uqg$-modules. Let $[[\L(m)]]$ denote the equivalence class of $\L(m)$, $m\in \P^+$. For each $\lambda\in P^+$ define $\ma \lambda := \left\{[[\L(m)]] : \wt m = \lambda \right\}$, the set of equivalence classes of affinizations of $V(\lambda)$.
Any finite-dimensional $\uqg$-module $V$ is isomorphic to a direct sum of finite-dimensional simple $\uqg$-modules; for each $\lambda\in P^+$ let $[V:V(\lambda)]$ denote the multiplicity of $V(\lambda)$ in $V$. There is a partial order $\leq$ on equivalence classes of affinizations in which $[[\L(m)]] \leq [[\L(m')]]$ if and only if for all $\nu\in P^+$ either
\begin{enumerate}[(i)]
\item $[\L(m):V(\nu)] \leq [\L(m'):V(\nu)]$, or
\item there exists a $\mu\in P^+$, $\mu\geq \nu$, such that $[\L(m):V(\mu)]<[\L(m'):V(\mu)]$.
\end{enumerate}
For all $\lambda\in P^+$, $\ma\lambda$ is a finite poset \cite{Cminaffrank2}.
A \emph{minimal affinization} of $V(\lambda)$, $\lambda\in P^+$, is a minimal element of $\ma \lambda$ with respect to the partial ordering \cite{Cminaffrank2}.
Thus a minimial affinization is by definition an equivalence class of $\uqgh$-modules; but we shall refer also to the elements of such a class -- the modules themselves -- as minimal affinizations.
Note that in type A all minimal affinizations are in fact \emph{evaluation representations}.

\section{Snake modules in types A and B}\label{sec:snakes}
In this section we introduce the class of representations to which our results apply, namely the snake modules, defined in \cite{MY1}. 
We specialize to types A and B: henceforth, $\g$ is either $\mf a_N$ or $\mf b_N$. 

\subsection{Notation, and the subring $\Z[Y^{\pm 1}_{i,k}]_{(i,k)\in\It}$} 
We define a subset $\It\subset I\times \Z$ as follows. 
\begin{enumerate}[Type A:]
\item Let $\It := \{ (i,k) \in I \times \Z : i-k\equiv 1 \mod 2 \}$.
\item Let $\It := \{(N,2k+1):k\in \Z\} \sqcup \{(i,k) \in I\times \Z : i<N \text{ and } k\equiv 0 \mod 2\}$.
\end{enumerate}
For the remainder of this paper, we pick and fix once and for all an $a\in \Cx$, and work solely with representations whose $q$-characters lie in the subring $\Z[Y_{i,aq^k}^{\pm 1}]_{(i,k)\in \It}$. These form a subcategory of the category of all finite-dimensional $\uqgh$-modules closed under taking tensor products.  
It is helpful to define also \be \Iw := \{(i,k) : (i,k-r_i)\in \It \}\ee
for we have, as a refinement of (\ref{imchiq})
\be\nn\forall\, m_+\in \Z[Y_{i,aq^k}]_{(i,k)\in \It},\,\,\forall\, m\in \mchiq(\L(m_+)), \qquad m m_+^{-1} \in \Z[ A_{i,aq^k}^{-1}]_{(i,k)\in \Iw}.\ee


From now on it is convenient to write, by an abuse of notation,
\be\nn Y_{i,k} := Y_{i,aq^k}, \quad A_{i,k} := A_{i,aq^k}\label{Yikdef},\quad u_{i,k} := u_{i,aq^k}\ee 
for all $(i,k)\in I \times \Z$ (\confer (\ref{udef}) for the definition of $u_{i,aq^k}$).  

\subsection{Snake position and minimal snake position}\label{snakepos} Let $(i,k)\in \It$. A point $(i',k') \in \It$ is said to be \emph{in snake position} with respect to $(i,k)$ if and only if
\begin{enumerate}[Type A:]
\item $k'-k\geq |i'-i|+2$. 
\item  $ $
  \begin{align*}
   &i=i'=N :  &k'-k  &\geq 2 & \text{ and }\quad k'-k&\equiv 2 \mod 4\\
  &i\neq i'=N \text{ or } i'\neq i=N :  &k'-k&\geq 2|i'-i| +3 & \text{ and }\quad k'-k &\equiv 2|i'-i| - 1 \mod 4\\
  &i<N \text{ and } i'<N:  &k'-k&\geq 2|i'-i|+4 & \text{ and }\quad k'-k &\equiv 2|i'-i| \mod 4. \end{align*}
\end{enumerate} 
The point $(i',k')$ is in \emph{minimal} snake position to $(i,k)$ if and only if $k'-k$ is equal to the given lower bound. 
\subsection{Prime snake position} 
Let $(i,k)\in \It$. We say that $(i',k')\in \It$ is in \emph{prime snake position} with respect to $(i,k)$ if and only if 
\begin{enumerate}[Type A:]
\item $i'+i\geq k'-k\geq |i'-i|+2$. 
\item  $ $
  \begin{align*}
   &i=i'=N :  &4N-2 &\geq &k'-k  &\geq 2 & \text{ and }\quad k'-k&\equiv 2 \mod 4\\
  &i\neq i'=N \text{ or } i'\neq i=N :&2i'+2i-1 &\geq &k'-k&\geq 2|i'-i| +3 & \text{ and }\quad k'-k &\equiv 2|i'-i| - 1 \mod 4\\
  &i<N \text{ and } i'<N: &2i'+2i &\geq &k'-k&\geq 2|i'-i|+4 & \text{ and }\quad k'-k &\equiv 2|i'-i| \mod 4. \end{align*}
\end{enumerate}

\subsection{Snakes and snake modules} A finite sequence $(i_t,k_t)$, $1\leq t \leq \T$, $\T\in \Z_{\geq 0}$, of points in $\It$ is a \emph{snake} if and only if for all $2\leq t \leq \T$, $(i_t,p_t)$ is in snake position with respect to $(i_{t-1},k_{t-1})$. It is a \emph{minimal} (resp. \emph{prime}) snake if and only if successive points are in minimal (resp. prime) snake position. Minimal snakes are prime. 

The simple module $\L(m)$ is a \emph{snake module} (resp. a \emph{minimal snake module}) if and only if $m=\prod_{t=1}^\T \YY {i_t}{k_t}$ for some snake $(i_t,k_t)_{1\leq t\leq \T}$ (resp. for some minimal snake $(i_t,k_t)_{1\leq t\leq \T}$).

The meaning of snake position is illustrated in Figure \ref{snakeposfig}. Here and subsequently, in type B we draw the images of points in $\It$ under the  injective map $\iota:\It \to \Z \times \Z$ defined as follows:
\begin{align} \text{Type B:}&\qquad \iota:  (i,k)  \mapsto 
\begin{cases} (2i,k) & i<N\text{ and } 2N+k-2i \equiv 2\mod 4\\
             (4N-2-2i,k) & i<N\text{ and } 2N+k-2i \equiv 0\mod 4\\
                                 (2N-1,k)   & i = N .\end{cases}\label{iotadef}\\
  \text{Type A:}& \qquad \iota: (i,k)\mapsto (i,k).\nn\end{align} 
(We define $\iota$ in type A purely in order to make certain statements more uniform in what follows.)

\begin{figure} \caption{\label{snakeposfig} The points $\scriptscriptstyle\square$ (resp. $\scriptscriptstyle\blacksquare$) are in snake position (resp. minimal snake position) to the point marked $\circ$. Of these, the points on or inside the dashed polygon shown are in prime snake position to $\circ$. In type $B$, we have plotted the images of points under the map $\iota$, \confer  (\ref{iotadef}).}
\be\nn\begin{tikzpicture}[scale=.5,yscale=-1]
\draw[help lines] (0,0) grid (5,6);
\foreach \y in {1,2,3,4,5,6,7} {\node at (-1,\y-1) {$\scriptstyle\y$};}
\draw (1,-1) -- (4,-1);
\foreach \x in {1,2,3,4} {\filldraw[fill=white] (\x,-1) circle (2mm) node[above=1mm] {$\scriptstyle\x$}; }
\draw[thick] (2,0) circle (1.5mm);
\draw[gray,dashed] (0,2) -- (2,0) -- (5,3)--(3,5)--cycle;
\foreach \x/\y in {1/3,2/2,3/3,4/4} 
{\node[regular polygon, regular polygon sides=4,draw,fill=black,inner sep=.4mm] at (\x,\y) {};}
\foreach \x/\y in {1/5,2/4,2/6,3/5,4/6} 
{\node[regular polygon, regular polygon sides=4,draw,fill=white,inner sep=.4mm] at (\x,\y) {};}
\end{tikzpicture}
\nn\ee
\be
\begin{tikzpicture}[scale=.35,yscale=-1]
\draw[help lines] (0,0) grid (14,18);
\draw (2,-1) -- (6,-1); \draw (8,-1) -- (12,-1);
\draw[double,->] (6,-1) -- (6.8,-1); \draw[double,->] (8,-1) -- (7.2,-1);
\filldraw[fill=white] (7,-1) circle (2mm) node[above=1mm] {$\scriptstyle 4$};
\foreach \x in {1,2,3} {
\filldraw[fill=white] (2*\x,-1) circle (2mm) node[above=1mm] {$\scriptstyle\x$}; 
\filldraw[fill=white] (2*7-2*\x,-1) circle (2mm) node[above=1mm] {$\scriptstyle\x$}; }
\foreach \y in {0,2,4,6,8,10,12,14,16,18} {\node at (-1,\y) {$\scriptstyle\y$};}
\draw[thick] (4,2) circle (2mm);
\draw[gray,dashed] (0,6) -- (4,2) -- (14,12)--(10,16)--cycle;
\foreach \x/\y in {2/8,4/6,4/10,6/8,6/12,7/9,7/13,2/12,2/16,4/14,4/18,6/16,7/17} 
{\node[regular polygon, regular polygon sides=4,draw,fill=white,inner sep=.4mm] at (\x,\y) {};}
\foreach \x/\y in {2/8,4/6,6/8,7/9} 
{\node[regular polygon, regular polygon sides=4,draw,fill=black,inner sep=.4mm] at (\x,\y) {};}
\end{tikzpicture}
\begin{tikzpicture}[scale=.35,yscale=-1]
\draw[help lines] (0,0) grid (14,18);
\draw (2,-1) -- (6,-1); \draw (8,-1) -- (12,-1);
\draw[double,->] (6,-1) -- (6.8,-1); \draw[double,->] (8,-1) -- (7.2,-1);
\filldraw[fill=white] (7,-1) circle (2mm) node[above=1mm] {$\scriptstyle 4$};
\foreach \x in {1,2,3} {
\filldraw[fill=white] (2*\x,-1) circle (2mm) node[above=1mm] {$\scriptstyle\x$}; 
\filldraw[fill=white] (2*7-2*\x,-1) circle (2mm) node[above=1mm] {$\scriptstyle\x$}; }
\foreach \y in {0,2,4,6,8,10,12,14,16,18} {\node at (-1,\y) {$\scriptstyle\y$};}
\begin{scope}[xscale=-1,xshift=-14cm,yshift=-2cm]
\draw[thick] (4,2) circle (2mm);
\draw[gray,dashed] (0,6) -- (4,2) -- (14,12)--(10,16)--cycle;
\foreach \x/\y in {2/8,4/6,4/10,6/8,6/12,7/9,7/13,2/12,2/16,4/14,4/18,6/16,7/17,6/20,2/20} 
{\node[regular polygon, regular polygon sides=4,draw,fill=white,inner sep=.4mm] at (\x,\y) {};}
\foreach \x/\y in {2/8,4/6,6/8,7/9} 
{\node[regular polygon, regular polygon sides=4,draw,fill=black,inner sep=.4mm] at (\x,\y) {};}
\end{scope}
\end{tikzpicture}
\nn\ee\be\begin{tikzpicture}[scale=.35,yscale=-1]
\draw[help lines] (0,0) grid (14,18);
\draw (2,-1) -- (6,-1); \draw (8,-1) -- (12,-1);
\draw[double,->] (6,-1) -- (6.8,-1); \draw[double,->] (8,-1) -- (7.2,-1);
\filldraw[fill=white] (7,-1) circle (2mm) node[above=1mm] {$\scriptstyle 4$};
\foreach \x in {1,2,3} {
\filldraw[fill=white] (2*\x,-1) circle (2mm) node[above=1mm] {$\scriptstyle\x$}; 
\filldraw[fill=white] (2*7-2*\x,-1) circle (2mm) node[above=1mm] {$\scriptstyle\x$}; }
\foreach \y in {0,2,4,6,8,10,12,14,16,18} {\node at (-1,\y) {$\scriptstyle\y$};}
\begin{scope}[yshift=1cm]
\draw[thick] (7,0) circle (2mm);
\draw[gray,dashed] (7,0) -- (14,7)--(7,14) -- cycle;
\foreach \x/\y in {7/2,7/6,7/10,7/14,8/5,8/9,8/13,10/7,10/11,12/9,8/17,10/15,12/13,12/17} 
{\node[regular polygon, regular polygon sides=4,draw,fill=white,inner sep=.4mm] at (\x,\y) {};}
\foreach \x/\y in {7/2,8/5,10/7,12/9} 
{\node[regular polygon, regular polygon sides=4,draw,fill=black,inner sep=.4mm] at (\x,\y) {};}
\end{scope}
\end{tikzpicture}
\begin{tikzpicture}[scale=.35,yscale=-1]
\draw[help lines] (0,0) grid (14,18);
\draw (2,-1) -- (6,-1); \draw (8,-1) -- (12,-1);
\draw[double,->] (6,-1) -- (6.8,-1); \draw[double,->] (8,-1) -- (7.2,-1);
\filldraw[fill=white] (7,-1) circle (2mm) node[above=1mm] {$\scriptstyle 4$};
\foreach \x in {1,2,3} {
\filldraw[fill=white] (2*\x,-1) circle (2mm) node[above=1mm] {$\scriptstyle\x$}; 
\filldraw[fill=white] (2*7-2*\x,-1) circle (2mm) node[above=1mm] {$\scriptstyle\x$}; }
\foreach \y in {0,2,4,6,8,10,12,14,16,18} {\node at (-1,\y) {$\scriptstyle\y$};}
\begin{scope}[yshift=3cm,xscale=-1,xshift=-14cm]
\draw[thick] (7,0) circle (2mm);
\draw[gray,dashed] (7,0) -- (14,7)--(7,14)--cycle;
\foreach \x/\y in {7/2,7/6,7/10,7/14,8/5,8/9,8/13,10/7,10/11,12/9,10/15,12/13} 
{\node[regular polygon, regular polygon sides=4,draw,fill=white,inner sep=.4mm] at (\x,\y) {};}
\foreach \x/\y in {7/2,8/5,10/7,12/9} 
{\node[regular polygon, regular polygon sides=4,draw,fill=black,inner sep=.4mm] at (\x,\y) {};}
\end{scope}
\end{tikzpicture}
\nn\ee
\end{figure}


Clearly every snake is a concatenation of prime snakes. Moreover, we have
\begin{prop}\label{primesnakes}
A snake module is prime if and only if its snake is prime. If a snake module is not prime then it is isomorphic to a tensor product of prime snake modules defined uniquely up to permutation.   
\end{prop}
The proof will be given in \S\ref{sec:snakechar}.

We also recall \cite{MY1} that for any snake \label{minaffprop} $(i_t,k_t)\in \It$, $1\leq t \leq \T$, of length $\T\in \Z_{\geq 1}$, the following are equivalent:
\begin{enumerate}
\item $\L(\prod_{t=1}^\T Y_{i_t,k_t})$ is a minimal affinization;
\item $(i_t,k_t)_{1\leq t \leq \T}$ is a minimal snake and the sequence $(i_t)_{1\leq t\leq \T}$ is monotonic. 
\end{enumerate}

\begin{figure}
\caption{In type $A_{13}$ (left) and in type $B_5$ (right), the black dots form a prime snake. The two neighbouring snakes are shown as triangles and diamonds. The snake-lowered paths (\S\ref{snakelowereddef}) are sketched as dotted lines.\label{snakefig} In type B, we have plotted the images of points under the map $\iota$, \confer  (\ref{iotadef}).}
\be\nn\begin{tikzpicture}[scale=.35,yscale=-1]
\draw[help lines] (0,0) grid (18,34);
\draw (2,-1) -- (8,-1); \draw (10,-1) -- (16,-1);
\draw[double,->] (8,-1) -- (8.8,-1); \draw[double,->] (10,-1) -- (9.2,-1);
\filldraw[fill=white] (9,-1) circle (2mm) node[above=1mm] {$\scriptstyle 5$};
\foreach \x in {1,2,3,4} {
\filldraw[fill=white] (2*\x,-1) circle (2mm) node[above=1mm] {$\scriptstyle\x$}; 
\filldraw[fill=white] (2*9-2*\x,-1) circle (2mm) node[above=1mm] {$\scriptstyle\x$}; }
\foreach \y in {0,2,4,6,8,10,12,14,16,18,20,22,24,26,28,30,32} {\node at (-1,\y) {$\scriptstyle\y$};}
\begin{scope}[yshift=0cm,xscale=-1,xshift=-18cm]
\draw[dotted] (0,6) -- (2,4) -- (4,6) -- (8,2) -- (18,12);
\draw[dotted] (0,10) --(2,8) -- (8,14) -- (9,13) -- (9,11) -- (18,20);
\draw[dotted] (0,22) --(6,16) -- (9,19) -- (9,15) -- (18,24);
\draw[dotted] (9,21) -- (10,20) -- (18,28);
\draw[dotted] (9,23) -- (8,22) -- (0,30);
\draw[dotted] (9,27) -- (10,28)--(12,26) -- (18,32);
\draw[dotted] (4,34) -- (9,29) -- (9,31) --(10,32) -- (12,30) -- (16,34);
\foreach \x/\y in {6/0,4/6,8/14,9/19,9/21,9/23,10/28,10/32} 
{\node[shape=circle,draw,fill=black,inner sep=.5mm] at (\x,\y) {};}
\foreach \x/\y in {8/2,9/11,9/13,9/15,10/20,12/26,12/30} 
{\node[regular polygon, regular polygon sides=3,draw,fill=white,inner sep=.3mm] at (\x,\y) {};}
\foreach \x/\y in {2/4,2/8,6/16,8/22,9/27,9/31,9/29} 
{\node[shape=diamond,draw,fill=white,inner sep=.5mm] at (\x,\y) {};}
\end{scope}

\begin{scope}[xshift=-18cm]
\draw[help lines] (0,0) grid (14,34);
\draw (1,-1) -- (13,-1);
\foreach \x in {1,2,3,4,5,6,7,8,9,10,11,12,13} {
\filldraw[fill=white] (\x,-1) circle (2mm) node[above=1mm] {$\scriptstyle\x$}; }
\foreach \y in {0,2,4,6,8,10,12,14,16,18,20,22,24,26,28,30,32} {\node at (-1,\y) {$\scriptstyle{\y}$};}
\begin{scope}[yshift=1cm]
\draw[dotted] (0,8) -- (7,1) -- (8,2) -- (9,1) -- (14,6);
\draw[dotted] (0,10) -- (6,4) -- (7,5) -- (9,3) -- (14,8);
\draw[dotted] (0,12) -- (6,6) -- (7,7) -- (8,6) -- (14,12);
\draw[dotted] (0,14) -- (2,12) -- (3,13) -- (8,8) -- (14,14);
\draw[dotted] (0,16)  -- (3,19) -- (6,16) -- (14,24);
\draw[dotted] (0,22) -- (2,20) -- (4,22) -- (5,21) -- (14,30);
\draw[dotted] (0,26) -- (3,23) -- (5,25) -- (6,24) -- (14,32);
\draw[dotted] (0,30) -- (3,27) -- (8,32) -- (10,30) -- (13,33);
\foreach \x/\y in {8/0,8/2,7/5,7/7,3/13,3/19,4/22,5/25,8/32} 
{\node[shape=circle,draw,fill=black,inner sep=.5mm] at (\x,\y) {};}
\foreach \x/\y in {7/1,6/4,6/6,2/12,2/20,3/23,3/27} 
{\node[regular polygon, regular polygon sides=3,draw,fill=white,inner sep=.3mm] at (\x,\y) {};}
\foreach \x/\y in {9/1,9/3,8/6,8/8,6/16,5/21,6/24,10/30} 
{\node[shape=diamond,draw,fill=white,inner sep=.5mm] at (\x,\y) {};}
\end{scope}
\end{scope}
\end{tikzpicture}\ee
\end{figure}
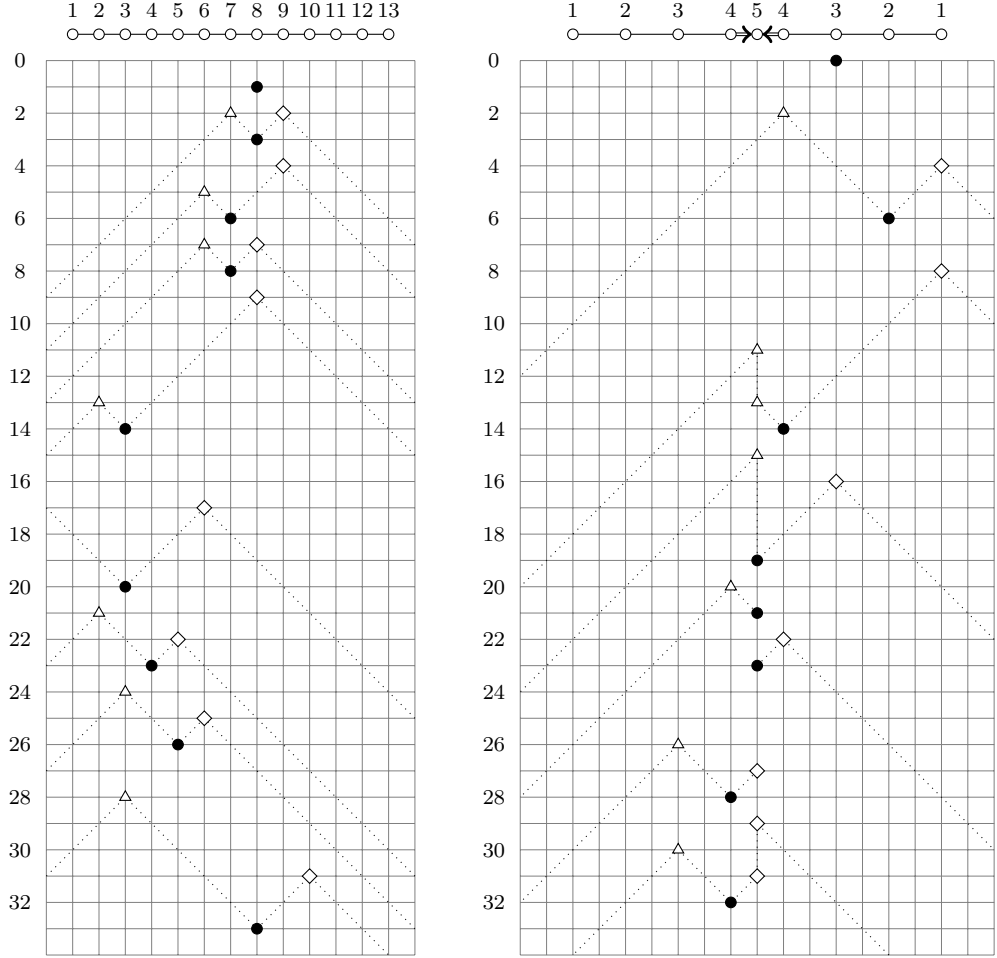
\subsection{Neighbouring points}\label{sec:nbrdef}
Suppose $(i,k)\in \It$ and $(i',k')\in \It$ are such that $(i',k')$ is in prime snake position with respect to $(i,k)$. 
In this subsection we shall define two finite sequences $\nbri_{i,k}^{i',k'}$ and $\nbrii_{i,k}^{i',k'}$ of points in $\It$, called the \emph{neighbouring points} to the pair $\left( (i,k), (i',k') \right)$. These sequences each consist of a single point, with an exception in type B where one of them consists of two points. 

The motivation for our definition is that if $x=\L(\prod_{(j,\ell)\in \nbri_{i,k}^{i',k'}} Y_{j,\ell})$ and $y=\L(\prod_{(j,\ell)\in \nbrii_{i,k}^{i',k'}} Y_{j,\ell})$ then 
\be \left[\L(Y_{i,k})\right]\left[\L(Y_{i',k'})\right] = \left[\L(Y_{i,k}Y_{i',k'})\right]+ [x][y]; \nn\ee
this will be a special case of Theorem \ref{Tsys} below. Neighbouring points are best understood from the pictures in Figure \ref{snakefig}. For example in the generic situation, far from any end of the Dynkin diagram, the points in $\nbri_{i,k}^{i',k'}\sqcup\nbrii_{i,k}^{i',k'}\sqcup\{(i,k),(i',k')\}$ are the vertices of a rectangle as in the following sketch.
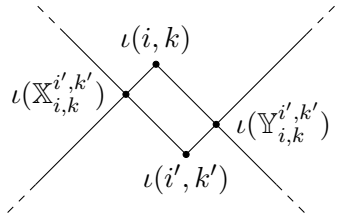
\be\nn\begin{tikzpicture}[scale=.4,yscale=-1]
\draw (2,0) -- ++(4,4);
\draw[dashed] (2,0)++(4,4) -- ++(1,1);
\draw (2,0) -- ++(-4,4);
\draw[dashed] (2,0)++(-4,4) -- ++(-1,1);
\draw (3,3) -- ++(-4,-4);
\draw[dashed] (3,3)++(-4,-4) -- ++(-1,-1);
\draw (3,3) -- ++(4,-4);
\draw[dashed] (3,3)++(4,-4) -- ++(1,-1);
\filldraw (2,0) circle (1mm) node[above] {$\iota(i,k)$};
\filldraw (3,3) circle (1mm) node[below] {$\iota(i',k')$};
\filldraw (1,1) circle (1mm) node[left]  {$\iota(\nbri_{i,k}^{i',k'})\,\,$};
\filldraw (4,2) circle (1mm) node[right] {$\,\,\iota(\nbrii_{i,k}^{i',k'})$};
\end{tikzpicture}\ee
The formal definition is as follows.
\subsubsection*{Neighbouring points in type A} 
\bea \nbri_{i,k}^{i',k'} &:=& 
\begin{cases} \left( \left(\alf(i+k +i'-k'), \alf(i+k - i'+k')\right)\right) 
                                               & k+i > k'-i' \\
               \emptyseq & k+i=k'-i',\end{cases}\nn\\
 \nbrii_{i,k}^{i',k'} &:=& 
\begin{cases} \left( \left(\alf(i'+k' + i-k), \alf(i'+k' - i+k)\right)\right) 
                                               & k+N+1-i > k'-(N+1-i') \\
               \emptyseq & k+N+1-i=k'-N-1+i'. \end{cases}\nn\eea
\subsubsection*{Neighbouring points in type B}
We define
\be (\nbri_{i,k}^{i',k'},\nbrii_{i,k}^{i',k'}) := \begin{cases} 
(B_{i,k}^{i',k'},F_{i,k}^{i',k'}) 
& (i<N,\, 2N-2i-k\equiv 2 \!\!\mod 4) \text{ or } (i=N,\, k\equiv 1 \!\!\mod 4) \\
(F_{i,k}^{i',k'},B_{i,k}^{i',k'}) 
& (i<N,\, 2N-2i-k\equiv 0 \!\!\mod 4) \text{ or } (i=N,\, k\equiv 3 \!\!\mod 4),
\end{cases} \nn\ee
where
\be B_{i,k}^{i',k'} := \begin{cases} 
\emptyseq & i<N, i'<N, k'-k = 2i+2i'\\
\left( \left( \tfrac{1}{4}(2i+k+2i'-k'),\alf(2i+k-2i'+k') \right) \right) & i<N, i'<N, k'-k < 2i+2i'\\
\emptyseq & i<N,i'=N, k'-k = 2i+2N-1\\
\left( \left( \tfrac{1}{4}(2i+k+2N-1-k'),\alf(2i+k-2N+1+k') \right) \right) & i<N,i'=N, k'-k < 2i+2N-1 \\
\left( \left( N, k-2N+1 + 2i'\right) \right) & i=N, i'<N \\
\emptyseq & i=N, i' = N, \end{cases} \nn\ee
\be F_{i,k}^{i',k'} := \begin{cases}
\left( \left( \tfrac{1}{4}(2i'+k'+2i-k),\alf(2i'+k'-2i+k) \right) \right) 
 & i<N, i'<N, k'-k \leq 4N-4-2i-2i'\\
\left( \left( N, k+2N-1-2i\right) , \left( N, k'-2N+1+2i'\right) \right) 
 & i<N, i'<N, k'-k \geq 4N-2i-2i'\\
\left( \left( N, k+2N-1-2i \right) \right) 
 & i<N, i'=N \\
\left( \left(  \tfrac{1}{4}(2N-1+k+2i'-k'),\alf(2N-1+k-2i'+k') \right) \right) 
 & i=N, i'<N, k'-k < 2N-1-2i'\\
\emptyseq & i=N, i'<N, k'-k = 2N-1-2i' \\
\left( \left( \tfrac{1}{4}(4N-2+k-k'), \alf(k+k') \right) \right)
 & i=N, i'= N. \end{cases}\nn\ee
\subsection{Neighbouring snakes}
\begin{prop}\label{nbrsprop}
Let $(i_t,k_t)\in \It$, $1\leq t\leq \T$, be a prime snake of length $\T\in \Z_{\geq 2}$. Then the sequences formed by concatenating the sequences of neighbouring points,
\begin{align*} \nbri_\snk &:= \nbri_{i_1,k_1}^{i_2,k_2} \concat \nbri_{i_2,k_2}^{i_3,k_3} \concat\dots \concat \nbri_{i_{\T-1},k_{\T-1}}^{i_\T,k_\T}, \\
 \nbrii_\snk &:= \nbrii_{i_1,k_1}^{i_2,k_2} \concat \nbrii_{i_2,k_2}^{i_3,k_3} \concat \dots \concat \nbrii_{i_{\T-1},k_{\T-1}}^{i_\T,k_\T}, \end{align*}
are snakes in $\It$ with no elements in common. 
\end{prop}
\begin{proof} By inspection, case by case. \end{proof}
We refer to the snakes $\nbri_\snk$ and $\nbrii_\snk$ as the \emph{neighbours} of the prime snake $\snk$. 
The neighbours are generically also prime, but there are exceptions whenever an A-type end of the Dynkin diagram is sufficiently close that some neighbouring points are ``missing''. For example, in the snake on the left of Figure \ref{snakefig}, $\nbri_{(3,14)}^{(3,20)} = \emptyseq$. The would-be neighbour is $(0,17)\notin I\times \Z$. 

In type A it is clear that every prime snake is strictly longer than its neighbours. In type B this is not always so, but we do have the following. 
Let $\hgt : \P \to \mathbb Z$, the \emph{height} map, be the homomorphism of abelian groups such that $\hgt Y_{i,k} = r_i$. Then
\begin{prop}\label{hgtprop}
For any prime snake $(i_t,k_t)\in \It$, $1\leq t\leq \T$, of length $\T\in \Z_{\geq 2}$,
\be\hgt\!\!\!\prod_{\substack{(i,k)\in\\\nbri_\snk}} \!\!\!\!\!\!\!\!Y_{i,k} \leq \hgt \prod_{t=1}^\T Y_{i_t,k_t}\quad\text{and}\quad\hgt\!\!\! \prod_{\substack{(i,k)\in\\\nbrii_\snk}} \!\!\!\!\!\!\!\!Y_{i,k} \leq \hgt \prod_{t=1}^\T Y_{i_t,k_t}\nn\ee
with equality only if the type is $B$ and $i_t=N$ for all $1\leq t\leq \T$.
\end{prop}
\begin{proof} By inspection.
\end{proof}

\section{The extended T-system}\label{Tsyssec}
\subsection{Main results}
We now state the main result of the paper. 
\begin{thm}\label{Tsys}  
Let $(i_t,k_t)\in \It$, $1\leq t\leq \T$, be a prime snake of length $\T\in \Z_{\geq 2}$. 
Let $\nbri:=\nbri_\snk$ and $\nbrii:=\nbrii_\snk$ be its neighbouring snakes.
Then we have the following relation in $\groth\uqgh$.
\bea \left[\L \left(\prod_{t=1}^{\T-1} Y_{i_t,k_t}\right)\right] \left[ \L\left(\prod_{t=2}^\T Y_{i_t,k_t}\right) \right] &=&  \left[\L\left(\prod_{t=2}^{\T-1} Y_{i_t,k_t}\right) \right]\left[ \L\left(\prod_{t=1}^\T Y_{i_t,k_t}\right) \right]\nn\\ &&+ \left[\L\left(\prod_{(i,k)\in\nbri} Y_{i,k}\right) \right]\left[ 
         \L\left(\prod_{(i,k)\in\nbrii}Y_{i,k}\right) \right]. 
              \label{grl}\eea
Moreover, the summands on the right-hand side are classes of irreducible modules, i.e. 
\begin{align} 
\L \left(\prod_{t=2}^{\T-1} Y_{i_t,k_t} \prod_{t=1}^\T Y_{i_t,k_t}\right)
   &\cong \L\left(\prod_{t=2}^{\T-1} Y_{i_t,k_t}\right)\otimes \L\left(\prod_{t=1}^\T Y_{i_t,k_t}\right),\label{tbsimp}\\
  \L\left(\prod_{(i,k)\in\nbri}Y_{i,k} \prod_{(i,k)\in\nbrii}Y_{i,k}\right) 
 &\cong  \L\left(\prod_{(i,k)\in\nbri}Y_{i,k}\right)  \otimes\L\left(\prod_{(i,k)\in\nbrii}Y_{i,k}\right).\label{nnsimp} 
\end{align}
\end{thm}

The proof is given in \S\ref{sec:proof}. In the case of Kirillov-Reshetikhin modules, the theorem is the standard T-system, and is proved \cite{NakajimaKR,HernandezKR}. 

\begin{rem} The statement in the theorem is equivalent to the existence of a short exact sequence of $\uqgh$-modules, which is either
\bea 0 \to \L\left(\prod_{(i,k)\in\nbri}Y_{i,k}\right) \otimes 
         \L\left(\prod_{(i,k)\in\nbrii}Y_{i,k}\right)  &\to& \L \left(\prod_{t=1}^{\T-1} Y_{i_t,k_t}\right) \otimes \L\left(\prod_{t=2}^\T Y_{i_t,k_t}\right) \nn\\ &\to& \L \left(\prod_{t=2}^{\T-1} Y_{i_t,k_t}\right) \otimes \L\left(\prod_{t=1}^\T Y_{i_t,k_t}\right)\to  0,\nn\eea
or the same sequence with arrows reversed. In the special case of Kirillov-Reshetikhin modules, it is known \cite{HernandezKR} that with the choice of coproduct in which $\Delta (x_{i,0}^+)= x_{i,0}^+ \otimes 1 + k_i\otimes x_{i,0}^+$ the arrows are as shown.
\end{rem}

We also have the following, whose proof is similar to that of Theorem \ref{Tsys}.
\begin{thm}\label{Tsys2}  
Let $(i_t,k_t)\in \It$, $1\leq t\leq \T$, be a non-prime snake of length $\T\in \Z_{\geq 2}$. 
Then we have the following isomorphisms of $\uqgh$-modules:
\be \L \left(\prod_{t=1}^{\T-1} Y_{i_t,k_t}\right)\otimes \L\left(\prod_{t=2}^\T Y_{i_t,k_t}\right) \cong
\L\left(\prod_{t=2}^{\T-1} Y_{i_t,k_t}\right) \otimes \L\left(\prod_{t=1}^\T Y_{i_t,k_t}\right)\cong
\L\left(\prod_{t=2}^{\T-1} Y_{i_t,k_t} \prod_{t=1}^\T Y_{i_t,k_t}\right).
\nn\ee
\end{thm}

We call the module $\L(\prod_{t=1}^\T Y_{i_t,k_t})$  in (\ref{grl}) the \emph{top module}. 
In almost all cases the height of the top module is strictly greater than the height of all other participating modules in (\ref{grl}), by Proposition \ref{hgtprop}. In the exceptional case, one neighbour may have the same height at the top module, but when this neighbour is treated as the top module in turn, its neighbours have strictly lower height. Therefore the relations (\ref{grl}) allow the classes of snake modules to be expressed in terms of classes of snake modules of lower height. 

Theorem \ref{Tsys} constitutes a system of relations among snake modules. We say a class of modules is \emph{closed} under the relations (\ref{grl}) if and only if whenever the top module is an element of that class, so too are all the other participating modules. In the remainder of this section, we exhibit various closed classes of modules and the corresponding sub-systems of relations.

The sub-systems we exhibit all involve only minimal snakes. Note that the relations of Theorem \ref{Tsys} are more general than the ones in \S\ref{KR}--\S\ref{wrapmods}. Note also that the class of all minimal snake modules is not closed under relations (\ref{grl}).

\subsection{Kirillov-Reshetikhin Modules}\label{KR}
For any $(i,k)\in I\times\Z$ and $\Ti\in \Z_{\geq 1}$, the sequence $(i,k+2tr_i)_{0\leq t\leq \Ti-1}$
is a minimal snake. 
It is usually called a \emph{$q$-string}. 
Define 
\be \st i k \Ti := \prod_{t=0}^{\Ti-1}  Y_{i,k+2tr_i},\quad\quad  \st i k 0 :=1.\nn\ee 
It is convenient also to define \be\nn\st i k \Ti:=1\quad\text{ for all } i\notin I.\ee
The Kirillov-Reshetikin (KR) modules are $\L(\st i k \Ti)$. Let us write
\be \W i k \Ti  := \Big[\L(\st i k \Ti)\Big] \nn.\ee 
The relations of Theorem \ref{Tsys} close on KR modules, and we recover the usual T-system:
\subsubsection*{Type A} For all $i\in I$, $k\in \Z$ and $m\geq 2$,
  \be\nn  \W i k {\Ti-1}   \W i {k+2} {\Ti-1}  
   = \W i k {\Ti}   \W i {k+2} {\Ti-2}  
   +  \W {i-1} {k+1} {\Ti-1}   \W {i+1} {k+1} {\Ti-1} . \ee

\subsubsection*{Type B} For all $1\leq i<N-1$, $k\in \Z$ and $m\geq 2$,
 \begin{align*}      \W i k {\Ti-1}   \W i {k+4} {\Ti-1}  
   &= \W i k {\Ti}   \W i {k+4} {\Ti-2}  
   +  \W {i-1} {k+2} {\Ti-1}   \W {i+1} {k+2} {\Ti-1}  .\end{align*}
The special cases are $i=N-1$ and $i=N$, as follows: for all $k\in \Z$ and $m\geq 2$
\begin{align}   \W {N-1} k {\Ti-1}   \W {N-1} {k+4} {\Ti-1}  
   &= \W {N-1} k {\Ti}   \W {N-1} {k+4} {\Ti-2}  
   +  \W {N-2} {k+2} {\Ti-1}   \W {N} {k+1} {2\Ti-2} ,\nn \\
  \W {N} k {\Ti-1}   \W {N} {k+2} {\Ti-1}  
   &= \W {N} k {\Ti}   \W {N} {k+2} {\Ti-2}  
   +  \W {N-1} {k+1} {\lfloor \frac{\Ti}{2}\rfloor}   \W {N-1} {k+3} {\lfloor \frac{\Ti-1}{2}\rfloor} \label{BendTsys}.
\end{align}
These relations uniquely determine every $\W i k m$ in terms of the classes of fundamental representations, $\W j \ell 1$.

\begin{figure}
\caption{Examples, in type $A_4$, of the 3 term relations in $\groth\uqgh$ from \S\ref{KR} and \S\ref{twonodes}.\label{figA}}
\be\nn\begin{tikzpicture}[scale=.35,yscale=-1]
\useasboundingbox (-.1,-1) -- (5.1,4.5);
\draw[help lines] (0,0) grid (5,9);
\draw (1,-1) -- (4,-1);
\foreach \x in {1,2,3,4} {
\filldraw[fill=white] (\x,-1) circle (2mm) node[above=1mm] {$\scriptstyle\x$}; }
\foreach \y in {0,2,4,6,8} {\node at (-1,\y) {$\scriptstyle{\y}$};}
\foreach \x/\y in {2/0,2/8} 
{\node[shape=circle,draw,fill=black,inner sep=.5mm] at (\x,\y) {};}
\foreach \x/\y in {2/2,2/4,2/6} 
{\node[shape=circle,draw,thick,double,fill=black,inner sep=.5mm] at (\x,\y) {};}
\foreach \x/\y in {1/1,1/3,1/5,1/7} 
{\node[regular polygon, regular polygon sides=3,draw,fill=white,inner sep=.3mm] at (\x,\y) {};}
\foreach \x/\y in {3/1,3/3,3/5,3/7}  
{\node[shape=diamond,draw,fill=white,inner sep=.5mm] at (\x,\y) {};}
\end{tikzpicture}\,\quad\nn
\left[\begin{tikzpicture}[scale=.35,yscale=-1]
\useasboundingbox (-.1,-1) -- (5.1,4.5);
\draw[help lines] (0,0) grid (5,9);
\draw (1,-1) -- (4,-1);
\foreach \x in {1,2,3,4} {
\filldraw[fill=white] (\x,-1) circle (2mm) node[above=1mm] {$\scriptstyle\x$}; }
\foreach \x/\y in {2/0,2/2,2/4,2/6} 
{\node[shape=circle,draw,fill=black,inner sep=.5mm] at (\x,\y) {};}
\end{tikzpicture}\right]
\left[\begin{tikzpicture}[scale=.35,yscale=-1]
\useasboundingbox (-.1,-1) -- (5.1,4.5);
\draw[help lines] (0,0) grid (5,9);
\draw (1,-1) -- (4,-1);
\foreach \x in {1,2,3,4} {
\filldraw[fill=white] (\x,-1) circle (2mm) node[above=1mm] {$\scriptstyle\x$}; }
\foreach \x/\y in {2/2,2/4,2/6,2/8} 
{\node[shape=circle,draw,fill=black,inner sep=.5mm] at (\x,\y) {};}
\end{tikzpicture}\right]
=
\left[\begin{tikzpicture}[scale=.35,yscale=-1]
\useasboundingbox (-.1,-1) -- (5.1,4.5);
\draw[help lines] (0,0) grid (5,9);
\draw (1,-1) -- (4,-1);
\foreach \x in {1,2,3,4} {
\filldraw[fill=white] (\x,-1) circle (2mm) node[above=1mm] {$\scriptstyle\x$}; }
\foreach \x/\y in {2/0,2/2,2/4,2/6,2/8} 
{\node[shape=circle,draw,fill=black,inner sep=.5mm] at (\x,\y) {};}
\end{tikzpicture}\right]
\left[\begin{tikzpicture}[scale=.35,yscale=-1]
\useasboundingbox (-.1,-1) -- (5.1,4.5);
\draw[help lines] (0,0) grid (5,9);
\draw (1,-1) -- (4,-1);
\foreach \x in {1,2,3,4} {
\filldraw[fill=white] (\x,-1) circle (2mm) node[above=1mm] {$\scriptstyle\x$}; }
\foreach \x/\y in {2/2,2/4,2/6} 
{\node[shape=circle,draw,fill=black,inner sep=.5mm] at (\x,\y) {};}
\end{tikzpicture}\right]
+
\left[\begin{tikzpicture}[scale=.35,yscale=-1]
\useasboundingbox (-.1,-1) -- (5.1,4.5);
\draw[help lines] (0,0) grid (5,9);
\draw (1,-1) -- (4,-1);
\foreach \x in {1,2,3,4} {
\filldraw[fill=white] (\x,-1) circle (2mm) node[above=1mm] {$\scriptstyle\x$}; }
\foreach \x/\y in {1/1,1/3,1/5,1/7} 
{\node[regular polygon, regular polygon sides=3,draw,fill=white,inner sep=.3mm] at (\x,\y) {};}
\end{tikzpicture}\right]
\left[\begin{tikzpicture}[scale=.35,yscale=-1]
\useasboundingbox (-.1,-1) -- (5.1,4.5);
\draw[help lines] (0,0) grid (5,9);
\draw (1,-1) -- (4,-1);
\foreach \x in {1,2,3,4} {
\filldraw[fill=white] (\x,-1) circle (2mm) node[above=1mm] {$\scriptstyle\x$}; }
\foreach \x/\y in {3/1,3/3,3/5,3/7}  
{\node[shape=diamond,draw,fill=white,inner sep=.5mm] at (\x,\y) {};}
\end{tikzpicture}\right]
\ee
%

$ $


\be\nn\begin{tikzpicture}[scale=.35,yscale=-1]
\useasboundingbox (-.1,-1) -- (5.1,4.5);
\draw[help lines] (0,0) grid (5,9);
\draw (1,-1) -- (4,-1);
\foreach \x in {1,2,3,4} {
\filldraw[fill=white] (\x,-1) circle (2mm) node[above=1mm] {$\scriptstyle\x$}; }
\foreach \y in {0,2,4,6,8} {\node at (-1,\y) {$\scriptstyle{\y}$};}
\foreach \x/\y in {2/0,3/9} 
{\node[shape=circle,draw,fill=black,inner sep=.5mm] at (\x,\y) {};}
\foreach \x/\y in {2/2,2/4,3/7} 
{\node[shape=circle,draw,thick,double,fill=black,inner sep=.5mm] at (\x,\y) {};}
\foreach \x/\y in {1/1,1/3,1/5,2/8} 
{\node[regular polygon, regular polygon sides=3,draw,fill=white,inner sep=.3mm] at (\x,\y) {};}
\foreach \x/\y in {3/1,3/3,4/6,4/8}  
{\node[shape=diamond,draw,fill=white,inner sep=.5mm] at (\x,\y) {};}
\end{tikzpicture}\,\quad\nn
\left[\begin{tikzpicture}[scale=.35,yscale=-1]
\useasboundingbox (-.1,-1) -- (5.1,4.5);
\draw[help lines] (0,0) grid (5,9);
\draw (1,-1) -- (4,-1);
\foreach \x in {1,2,3,4} {
\filldraw[fill=white] (\x,-1) circle (2mm) node[above=1mm] {$\scriptstyle\x$}; }
\foreach \x/\y in {2/0,2/2,2/4,3/7} 
{\node[shape=circle,draw,fill=black,inner sep=.5mm] at (\x,\y) {};}
\end{tikzpicture}\right]
\left[\begin{tikzpicture}[scale=.35,yscale=-1]
\useasboundingbox (-.1,-1) -- (5.1,4.5);
\draw[help lines] (0,0) grid (5,9);
\draw (1,-1) -- (4,-1);
\foreach \x in {1,2,3,4} {
\filldraw[fill=white] (\x,-1) circle (2mm) node[above=1mm] {$\scriptstyle\x$}; }
\foreach \x/\y in {2/2,2/4,3/7,3/9} 
{\node[shape=circle,draw,fill=black,inner sep=.5mm] at (\x,\y) {};}
\end{tikzpicture}\right]
=
\left[\begin{tikzpicture}[scale=.35,yscale=-1]
\useasboundingbox (-.1,-1) -- (5.1,4.5);
\draw[help lines] (0,0) grid (5,9);
\draw (1,-1) -- (4,-1);
\foreach \x in {1,2,3,4} {
\filldraw[fill=white] (\x,-1) circle (2mm) node[above=1mm] {$\scriptstyle\x$}; }
\foreach \x/\y in {2/0,2/2,2/4,3/7,3/9} 
{\node[shape=circle,draw,fill=black,inner sep=.5mm] at (\x,\y) {};}
\end{tikzpicture}\right]
\left[\begin{tikzpicture}[scale=.35,yscale=-1]
\useasboundingbox (-.1,-1) -- (5.1,4.5);
\draw[help lines] (0,0) grid (5,9);
\draw (1,-1) -- (4,-1);
\foreach \x in {1,2,3,4} {
\filldraw[fill=white] (\x,-1) circle (2mm) node[above=1mm] {$\scriptstyle\x$}; }
\foreach \x/\y in {2/2,2/4,3/7} 
{\node[shape=circle,draw,fill=black,inner sep=.5mm] at (\x,\y) {};}
\end{tikzpicture}\right]
+
\left[\begin{tikzpicture}[scale=.35,yscale=-1]
\useasboundingbox (-.1,-1) -- (5.1,4.5);
\draw[help lines] (0,0) grid (5,9);
\draw (1,-1) -- (4,-1);
\foreach \x in {1,2,3,4} {
\filldraw[fill=white] (\x,-1) circle (2mm) node[above=1mm] {$\scriptstyle\x$}; }
\foreach \x/\y in {1/1,1/3,1/5,2/8} 
{\node[regular polygon, regular polygon sides=3,draw,fill=white,inner sep=.3mm] at (\x,\y) {};}
\end{tikzpicture}\right]
\left[\begin{tikzpicture}[scale=.35,yscale=-1]
\useasboundingbox (-.1,-1) -- (5.1,4.5);
\draw[help lines] (0,0) grid (5,9);
\draw (1,-1) -- (4,-1);
\foreach \x in {1,2,3,4} {
\filldraw[fill=white] (\x,-1) circle (2mm) node[above=1mm] {$\scriptstyle\x$}; }
\foreach \x/\y in {3/1,3/3,4/6,4/8}  
{\node[shape=diamond,draw,fill=white,inner sep=.5mm] at (\x,\y) {};}
\end{tikzpicture}\right]
\ee

$ $

$ $
\end{figure}
\begin{figure}
\caption{Examples, in type $B_3$, of the 3 term relations in $\groth\uqgh$ of \S\ref{KR} and \S\ref{twonodes}.\label{figB}}
\be\nn\begin{tikzpicture}[scale=.25,yscale=-1]
  \useasboundingbox (1,-1.1) -- (10.1,7.1);
\draw[help lines] (0,0) grid (10,14);
\draw (2,-1) -- (4,-1); \draw (6,-1) -- (8,-1);
\draw[double,->] (4,-1) -- (4.8,-1); \draw[double,->] (6,-1) -- (5.2,-1);
\filldraw[fill=white] (5,-1) circle (3mm) node[above=1mm] {$\scriptstyle 3$};
\foreach \x in {1,2} {
\filldraw[fill=white] (2*\x,-1) circle (3mm) node[above=1mm] {$\scriptstyle\x$}; 
\filldraw[fill=white] (2*5-2*\x,-1) circle (3mm) node[above=1mm] {$\scriptstyle\x$}; }
\foreach \y in {0,4,8,12} {\node at (-1,\y) {$\scriptstyle\y$};}
\foreach \x/\y in {4/0,4/4,4/8,4/12} 
{\node[shape=circle,draw,fill=black,inner sep=.5mm] at (\x,\y) {};}
\foreach \x/\y in {4/4,4/8} 
{\node[shape=circle,draw,double,thick,fill=black,inner sep=.5mm] at (\x,\y) {};}
\foreach \x/\y in {2/2,2/6,2/10} 
{\node[regular polygon, regular polygon sides=3,draw,fill=white,inner sep=.3mm] at (\x,\y) {};}
\foreach \x/\y in {5/1,5/3,5/5,5/7,5/9,5/11} 
{\node[shape=diamond,draw,fill=white,inner sep=.5mm] at (\x,\y) {};}
\end{tikzpicture}\,\quad
\left[\begin{tikzpicture}[scale=.25,yscale=-1]
\useasboundingbox (1.5,-2) -- (8.5,7);
\draw[help lines] (2,0) grid (8,14);
\draw (2,-1) -- (4,-1); \draw (6,-1) -- (8,-1);
\draw[double,->] (4,-1) -- (4.8,-1); \draw[double,->] (6,-1) -- (5.2,-1);
\filldraw[fill=white] (5,-1) circle (3mm) node[above=1mm] {$\scriptstyle 3$};
\foreach \x in {1,2} {
\filldraw[fill=white] (2*\x,-1) circle (3mm) node[above=1mm] {$\scriptstyle\x$}; 
\filldraw[fill=white] (2*5-2*\x,-1) circle (3mm) node[above=1mm] {$\scriptstyle\x$}; }
\foreach \x/\y in {4/0,4/4,4/8} 
{\node[shape=circle,draw,fill=black,inner sep=.5mm] at (\x,\y) {};}
\end{tikzpicture}\right]
\left[\begin{tikzpicture}[scale=.25,yscale=-1]
\useasboundingbox (1.5,-2) -- (8.5,7);
\draw[help lines] (2,0) grid (8,14);
\draw (2,-1) -- (4,-1); \draw (6,-1) -- (8,-1);
\draw[double,->] (4,-1) -- (4.8,-1); \draw[double,->] (6,-1) -- (5.2,-1);
\filldraw[fill=white] (5,-1) circle (3mm) node[above=1mm] {$\scriptstyle 3$};
\foreach \x in {1,2} {
\filldraw[fill=white] (2*\x,-1) circle (3mm) node[above=1mm] {$\scriptstyle\x$}; 
\filldraw[fill=white] (2*5-2*\x,-1) circle (3mm) node[above=1mm] {$\scriptstyle\x$}; }
\foreach \x/\y in {4/4,4/8,4/12} 
{\node[shape=circle,draw,fill=black,inner sep=.5mm] at (\x,\y) {};}
\end{tikzpicture}\right]
=
\left[\begin{tikzpicture}[scale=.25,yscale=-1]
\useasboundingbox (1.5,-2) -- (8.5,7);
\draw[help lines] (2,0) grid (8,14);
\draw (2,-1) -- (4,-1); \draw (6,-1) -- (8,-1);
\draw[double,->] (4,-1) -- (4.8,-1); \draw[double,->] (6,-1) -- (5.2,-1);
\filldraw[fill=white] (5,-1) circle (3mm) node[above=1mm] {$\scriptstyle 3$};
\foreach \x in {1,2} {
\filldraw[fill=white] (2*\x,-1) circle (3mm) node[above=1mm] {$\scriptstyle\x$}; 
\filldraw[fill=white] (2*5-2*\x,-1) circle (3mm) node[above=1mm] {$\scriptstyle\x$}; }
\foreach \x/\y in {4/0,4/4,4/8,4/12} 
{\node[shape=circle,draw,fill=black,inner sep=.5mm] at (\x,\y) {};}
\end{tikzpicture}\right]
\left[\begin{tikzpicture}[scale=.25,yscale=-1]
\useasboundingbox (1.5,-2) -- (8.5,7);
\draw[help lines] (2,0) grid (8,14);
\draw (2,-1) -- (4,-1); \draw (6,-1) -- (8,-1);
\draw[double,->] (4,-1) -- (4.8,-1); \draw[double,->] (6,-1) -- (5.2,-1);
\filldraw[fill=white] (5,-1) circle (3mm) node[above=1mm] {$\scriptstyle 3$};
\foreach \x in {1,2} {
\filldraw[fill=white] (2*\x,-1) circle (3mm) node[above=1mm] {$\scriptstyle\x$}; 
\filldraw[fill=white] (2*5-2*\x,-1) circle (3mm) node[above=1mm] {$\scriptstyle\x$}; }
\foreach \x/\y in {4/4,4/8} 
{\node[shape=circle,draw,fill=black,inner sep=.5mm] at (\x,\y) {};}
\end{tikzpicture}\right]
+
\left[\begin{tikzpicture}[scale=.25,yscale=-1]
\useasboundingbox (1.5,-2) -- (8.5,7);
\draw[help lines] (2,0) grid (8,14);
\draw (2,-1) -- (4,-1); \draw (6,-1) -- (8,-1);
\draw[double,->] (4,-1) -- (4.8,-1); \draw[double,->] (6,-1) -- (5.2,-1);
\filldraw[fill=white] (5,-1) circle (3mm) node[above=1mm] {$\scriptstyle 3$};
\foreach \x in {1,2} {
\filldraw[fill=white] (2*\x,-1) circle (3mm) node[above=1mm] {$\scriptstyle\x$}; 
\filldraw[fill=white] (2*5-2*\x,-1) circle (3mm) node[above=1mm] {$\scriptstyle\x$}; }
\foreach \x/\y in {2/2,2/6,2/10} 
{\node[regular polygon, regular polygon sides=3,draw,fill=white,inner sep=.3mm] at (\x,\y) {};}
\end{tikzpicture}\right]
\left[\begin{tikzpicture}[scale=.25,yscale=-1]
\useasboundingbox (1.5,-2) -- (8.5,7);
\draw[help lines] (2,0) grid (8,14);
\draw (2,-1) -- (4,-1); \draw (6,-1) -- (8,-1);
\draw[double,->] (4,-1) -- (4.8,-1); \draw[double,->] (6,-1) -- (5.2,-1);
\filldraw[fill=white] (5,-1) circle (3mm) node[above=1mm] {$\scriptstyle 3$};
\foreach \x in {1,2} {
\filldraw[fill=white] (2*\x,-1) circle (3mm) node[above=1mm] {$\scriptstyle\x$}; 
\filldraw[fill=white] (2*5-2*\x,-1) circle (3mm) node[above=1mm] {$\scriptstyle\x$}; }
\foreach \x/\y in {5/1,5/3,5/5,5/7,5/9,5/11} 
{\node[shape=diamond,draw,fill=white,inner sep=.5mm] at (\x,\y) {};}
\end{tikzpicture}\right]
\ee

$ $

\be\nn\begin{tikzpicture}[scale=.25,yscale=-1]
  \useasboundingbox (1,-1.1) -- (10.1,7.1);
\draw[help lines] (0,0) grid (10,14);
\draw (2,-1) -- (4,-1); \draw (6,-1) -- (8,-1);
\draw[double,->] (4,-1) -- (4.8,-1); \draw[double,->] (6,-1) -- (5.2,-1);
\filldraw[fill=white] (5,-1) circle (3mm) node[above=1mm] {$\scriptstyle 3$};
\foreach \x in {1,2} {
\filldraw[fill=white] (2*\x,-1) circle (3mm) node[above=1mm] {$\scriptstyle\x$}; 
\filldraw[fill=white] (2*5-2*\x,-1) circle (3mm) node[above=1mm] {$\scriptstyle\x$}; }
\foreach \y in {0,4,8,12} {\node at (-1,\y) {$\scriptstyle\y$};}
\foreach \x/\y in {4/0,4/4,5/9,5/11,5/13} 
{\node[shape=circle,draw,fill=black,inner sep=.5mm] at (\x,\y) {};}
\foreach \x/\y in {4/4,5/9,5/11} 
{\node[shape=circle,draw,double,thick,fill=black,inner sep=.5mm] at (\x,\y) {};}
\foreach \x/\y in {2/2,2/6,4/12} 
{\node[regular polygon, regular polygon sides=3,draw,fill=white,inner sep=.3mm] at (\x,\y) {};}
\foreach \x/\y in {5/1,5/3,5/5,6/10} 
{\node[shape=diamond,draw,fill=white,inner sep=.5mm] at (\x,\y) {};}
\end{tikzpicture}\,\quad
\left[\begin{tikzpicture}[scale=.25,yscale=-1]
\useasboundingbox (1.5,-2) -- (8.5,7);
\draw[help lines] (2,0) grid (8,14);
\draw (2,-1) -- (4,-1); \draw (6,-1) -- (8,-1);
\draw[double,->] (4,-1) -- (4.8,-1); \draw[double,->] (6,-1) -- (5.2,-1);
\filldraw[fill=white] (5,-1) circle (3mm) node[above=1mm] {$\scriptstyle 3$};
\foreach \x in {1,2} {
\filldraw[fill=white] (2*\x,-1) circle (3mm) node[above=1mm] {$\scriptstyle\x$}; 
\filldraw[fill=white] (2*5-2*\x,-1) circle (3mm) node[above=1mm] {$\scriptstyle\x$}; }
\foreach \x/\y in {4/0,4/4,5/9,5/11} 
{\node[shape=circle,draw,fill=black,inner sep=.5mm] at (\x,\y) {};}
\end{tikzpicture}\right]
\left[\begin{tikzpicture}[scale=.25,yscale=-1]
\useasboundingbox (1.5,-2) -- (8.5,7);
\draw[help lines] (2,0) grid (8,14);
\draw (2,-1) -- (4,-1); \draw (6,-1) -- (8,-1);
\draw[double,->] (4,-1) -- (4.8,-1); \draw[double,->] (6,-1) -- (5.2,-1);
\filldraw[fill=white] (5,-1) circle (3mm) node[above=1mm] {$\scriptstyle 3$};
\foreach \x in {1,2} {
\filldraw[fill=white] (2*\x,-1) circle (3mm) node[above=1mm] {$\scriptstyle\x$}; 
\filldraw[fill=white] (2*5-2*\x,-1) circle (3mm) node[above=1mm] {$\scriptstyle\x$}; }
\foreach \x/\y in {4/4,5/9,5/11,5/13} 
{\node[shape=circle,draw,fill=black,inner sep=.5mm] at (\x,\y) {};}
\end{tikzpicture}\right]
=
\left[\begin{tikzpicture}[scale=.25,yscale=-1]
\useasboundingbox (1.5,-2) -- (8.5,7);
\draw[help lines] (2,0) grid (8,14);
\draw (2,-1) -- (4,-1); \draw (6,-1) -- (8,-1);
\draw[double,->] (4,-1) -- (4.8,-1); \draw[double,->] (6,-1) -- (5.2,-1);
\filldraw[fill=white] (5,-1) circle (3mm) node[above=1mm] {$\scriptstyle 3$};
\foreach \x in {1,2} {
\filldraw[fill=white] (2*\x,-1) circle (3mm) node[above=1mm] {$\scriptstyle\x$}; 
\filldraw[fill=white] (2*5-2*\x,-1) circle (3mm) node[above=1mm] {$\scriptstyle\x$}; }
\foreach \x/\y in {4/0,4/4,5/9,5/11,5/13} 
{\node[shape=circle,draw,fill=black,inner sep=.5mm] at (\x,\y) {};}
\end{tikzpicture}\right]
\left[\begin{tikzpicture}[scale=.25,yscale=-1]
\useasboundingbox (1.5,-2) -- (8.5,7);
\draw[help lines] (2,0) grid (8,14);
\draw (2,-1) -- (4,-1); \draw (6,-1) -- (8,-1);
\draw[double,->] (4,-1) -- (4.8,-1); \draw[double,->] (6,-1) -- (5.2,-1);
\filldraw[fill=white] (5,-1) circle (3mm) node[above=1mm] {$\scriptstyle 3$};
\foreach \x in {1,2} {
\filldraw[fill=white] (2*\x,-1) circle (3mm) node[above=1mm] {$\scriptstyle\x$}; 
\filldraw[fill=white] (2*5-2*\x,-1) circle (3mm) node[above=1mm] {$\scriptstyle\x$}; }
\foreach \x/\y in {4/4,5/9,5/11} 
{\node[shape=circle,draw,fill=black,inner sep=.5mm] at (\x,\y) {};}
\end{tikzpicture}\right]
+
\left[\begin{tikzpicture}[scale=.25,yscale=-1]
\useasboundingbox (1.5,-2) -- (8.5,7);
\draw[help lines] (2,0) grid (8,14);
\draw (2,-1) -- (4,-1); \draw (6,-1) -- (8,-1);
\draw[double,->] (4,-1) -- (4.8,-1); \draw[double,->] (6,-1) -- (5.2,-1);
\filldraw[fill=white] (5,-1) circle (3mm) node[above=1mm] {$\scriptstyle 3$};
\foreach \x in {1,2} {
\filldraw[fill=white] (2*\x,-1) circle (3mm) node[above=1mm] {$\scriptstyle\x$}; 
\filldraw[fill=white] (2*5-2*\x,-1) circle (3mm) node[above=1mm] {$\scriptstyle\x$}; }
\foreach \x/\y in {2/2,2/6,4/12} 
{\node[regular polygon, regular polygon sides=3,draw,fill=white,inner sep=.3mm] at (\x,\y) {};}
\end{tikzpicture}\right]
\left[\begin{tikzpicture}[scale=.25,yscale=-1]
\useasboundingbox (1.5,-2) -- (8.5,7);
\draw[help lines] (2,0) grid (8,14);
\draw (2,-1) -- (4,-1); \draw (6,-1) -- (8,-1);
\draw[double,->] (4,-1) -- (4.8,-1); \draw[double,->] (6,-1) -- (5.2,-1);
\filldraw[fill=white] (5,-1) circle (3mm) node[above=1mm] {$\scriptstyle 3$};
\foreach \x in {1,2} {
\filldraw[fill=white] (2*\x,-1) circle (3mm) node[above=1mm] {$\scriptstyle\x$}; 
\filldraw[fill=white] (2*5-2*\x,-1) circle (3mm) node[above=1mm] {$\scriptstyle\x$}; }
\foreach \x/\y in {5/1,5/3,5/5,6/10} 
{\node[shape=diamond,draw,fill=white,inner sep=.5mm] at (\x,\y) {};}
\end{tikzpicture}\right]
\ee
\end{figure}
 
\subsection{Minimal affinizations on two neighbouring nodes} 
\label{twonodes}
The KR modules are the minimal affinizations of the simple $\uqg$-modules $V(\Ti\omega_i)$, $\Ti\in \Z_{\geq 1}$, $i\in I$. The relations of Theorem \ref{Tsys} also close on the class of minimal affinizations of simple $\uqg$-modules of the form $V(\Ti\omega_i+\Tii\omega_{i+1})$. Indeed, given any $(i,k)\in \Z\times \Z$ and any $\Ti,\Tii\in \Z_{\geq 0}$, let us define
\begin{align*} \W i k {\Ti,\Tii} &:=\Big[\L\left(\st i k {\Ti} \st {i+1} {k+2\Ti r_i-r_i+r_{i+1} + \max(r_i,r_{i+1})} {\Tii}\right)\Big], \\ 
 \Wt {i+1} k {\Ti,\Tii} &:= \Big[\L\left(\st {i+1} k {\Ti} \st {i} {k + 2\Ti r_{i+1} -r_{i+1}+r_i+\max(r_i,r_{i+1})} {\Tii}\right)\Big].\end{align*}

Every minimal affinization of  $V(\Ti\omega_i + \Tii \omega_{i+1})$, $1\leq i \leq N-1$, is of the form $\W i k {\Ti, \Tii}$ or $\Wt {i+1} k {\Tii, \Ti}$ (for some choice of the $a\in \Cx$ in \S\ref{sec:snakes}). 
We have the following relations for all $k\in \Z$ and all $m,n\in \Z_{\geq 1}$.
\subsubsection*{Type A} For all $1\leq i \leq N-1$,
 \begin{align*}  \W i k {\Ti, \Tii-1}  \W i {k+2} {\Ti-1,\Tii}  
   &=  \W i k {\Ti,\Tii} \W i {k+2} {\Ti-1,\Tii-1} 
+  \W {i-1} {k+1} {\Ti,\Tii-1} \W {i+1} {k+1} {\Ti-1,\Tii} , \\
 \Wt i k {\Ti, \Tii-1}  \Wt i {k+2} {\Ti-1,\Tii}  
   &=  \Wt i k {\Ti,\Tii} \Wt i {k+2} {\Ti-1,\Tii-1} 
+  \Wt {i+1} {k+1} {\Ti,\Tii-1} \Wt {i-1} {k+1} {\Ti-1,\Tii}.\end{align*}
\subsubsection*{Type B} For all $1\leq i<N-1$,
\begin{align*} \W i k {\Ti,\Tii-1}  \W i {k+4} {\Ti-1,\Tii}  
   &=  \W i k {\Ti,\Tii}  \W i {k+4} {\Ti-1,\Tii-1}  
        +   \W {i-1} {k+2} {\Ti,\Tii-1}   \W {i+1} {k+2} {\Ti-1,\Tii}.\end{align*}
For all $1<i\leq N-1$, 
\begin{align*} \Wt {i} k {\Ti,\Tii-1}  \Wt {i} {k+4} {\Ti-1,\Tii}  
   &=  \Wt {i} k {\Ti,\Tii}  \Wt {i} {k+4} {\Ti-1,\Tii-1}  
     +   \Wt {i+1} {k+2} {\Ti,\Tii-1}   \Wt {i-1} {k+2} {\Ti-1,\Tii}.\end{align*}
Finally, let us define $s(\lambda)$, $\lambda\in \Z$, to be $0$ if $\lambda$ is even, $1$ if $\lambda$ is odd. Then
\begin{align}
      \W {N-1} k {\Ti,\Tii-1}   \W {N-1} {k+4} {\Ti-1,\Tii}  
   &=  \W {N-1} k {\Ti,\Tii} \W {N-1} {k+4} {\Ti-1,\Tii-1}  
+   \W {N-2} {k+2} {\Ti,\lfloor{\frac{\Tii-1}{2}\rfloor}}  
    \Wt {N} {k+1} {2\Ti-1,\lfloor{\frac{\Tii}{2}\rfloor}}, \nn\\
      \Wt {N} k {\Ti,\Tii-1} \Wt {N} {k+2} {\Ti-1,\Tii}  
   &=  \Wt {N} k {\Ti,\Tii} \Wt {N} {k+2} {\Ti-1,\Tii-1}  
  +   \W {N-1} {k+1+2s(m)} {\lfloor{\frac{\Ti}{2}\rfloor}, 2\Tii-1}  
  \Wt {N-1} {k+1+2s(m-1)} {\lfloor{\frac{\Ti-1}{2}\rfloor},\Tii}. \label{btwonoderel}\end{align}
Observe that the final pair of relations in (\ref{btwonoderel}) mix $\mathbf T$ and $\widetilde{\mathbf T}$ in type B. This is an indication that in type B the relations of Theorem \ref{Tsys} do not close on the class of all minimal affinizations.

Given the classes of KR modules, $\W j \ell p$, the relations above uniquely determine all the $\W i k {m,n}$.

Figures \ref{figA} and \ref{figB} illustrate the fashion in which the 3-term relations of this subsection generalize those of the usual T-system.

\subsection{Minimal affinizations in type A}
For this subsection we work in type $A_N$. 
Given $(a,k)\in \Z\times\Z$, $s\in \Z_{\geq 1}$, and $(n_1,\dots,n_s)\in \Z_{\geq 0}^{s}$, let $k_1:=k$ and for each $1\leq i<s$, $k_{i+1} = k_i + 2n_{i} + 1$. Then we define
\begin{align*} \W a k {n_1,n_2,\dots,n_{s}} 
  &:= \left[\L (\st a {k_1} {n_1} \st {a+1} {k_{2}} {n_{2}} \dots \st {a+s-1} {k_{s}} {n_{s}})\right], \end{align*}
\begin{align*} 
\Wt {a} {k} {n_1,n_2,\dots,n_{s}} 
&:= \left[\L (\st {a} {k_{1}} {n_{1}} \st {a-1} {k_{2}} {n_{2}} \dots \st {a-s+1} {k_s} {n_s})\right]. \end{align*}

Given any $\Lambda = \sum_{i\in I} \lambda_i \omega_i\in P^+\setminus\{0\}$, let $a:=\min\{i\in I:\lambda_i>0\}$ and $b:=\max\{i\in I: \lambda_i>0\}$. All minimal affinizations of $V(\Lambda)$ are of the form $\W a k {\lambda_a,\lambda_{a+1},\dots,\lambda_{b-1}, \lambda_b}$ or $\Wt b k {\lambda_b,\lambda_{b-1},\dots,\lambda_{a+1}, \lambda_a}$, $k \in \Z$. If $a=b$ we have the relations of \S\ref{KR}; otherwise $a<b$ and we have
\begin{align}
\W a k {\lambda_a,\lambda_{a+1},\dots,\lambda_{b-1}, \lambda_b-1} 
\W a {k+2} {\lambda_a-1,\lambda_{a+1},\dots,\lambda_{b-1}, \lambda_b} 
&=
\W a k {\lambda_a,\lambda_{a+1},\dots,\lambda_{b-1}, \lambda_b} 
\W a {k+2} {\lambda_a-1,\lambda_{a+1},\dots,\lambda_{b-1}, \lambda_b-1}\\ 
&{}+
\W {a-1} {k+1} {\lambda_a,\lambda_{a+1},\dots,\lambda_{b-1}, \lambda_b-1} 
\W {a+1} {k+1} {\lambda_a-1,\lambda_{a+1},\dots,\lambda_{b-1}, \lambda_b},\nn \\
\Wt b k {\lambda_b,\lambda_{b-1},\dots,\lambda_{a+1}, \lambda_a-1} 
\Wt b {k+2} {\lambda_b-1,\lambda_{b-1},\dots,\lambda_{a+1}, \lambda_a} 
&=
\Wt b k {\lambda_b,\lambda_{b-1},\dots,\lambda_{a+1}, \lambda_a} 
\Wt b {k+2} {\lambda_b-1,\lambda_{b-1},\dots,\lambda_{a+1}, \lambda_a-1}\label{Aminaffsys}\\ 
&{}+
\Wt {b+1} {k+1} {\lambda_b,\lambda_{b-1},\dots,\lambda_{a+1}, \lambda_a-1} 
\Wt {b-1} {k+1} {\lambda_b-1,\lambda_{b-1},\dots,\lambda_{a+1}, \lambda_a}. 
\nn\end{align}
For all  $s>s'\geq 1$, these relations determine the $\W a k {n_1,\dots,n_s}$ in terms of the $\W b \ell {m_1,\dots,m_{s'}}$; and likewise the $\Wt a k {n_1,\dots,n_s}$ in terms of the $\Wt b \ell {m_1,\dots,m_{s'}}$.
 
In the conventions of \cite{MY1}, the modules in (\ref{Aminaffsys}) are the evaluation representations whose highest $\mf{gl}_{N+1}$-weights are given by the Young diagrams sketched in Figure \ref{YDfig}.

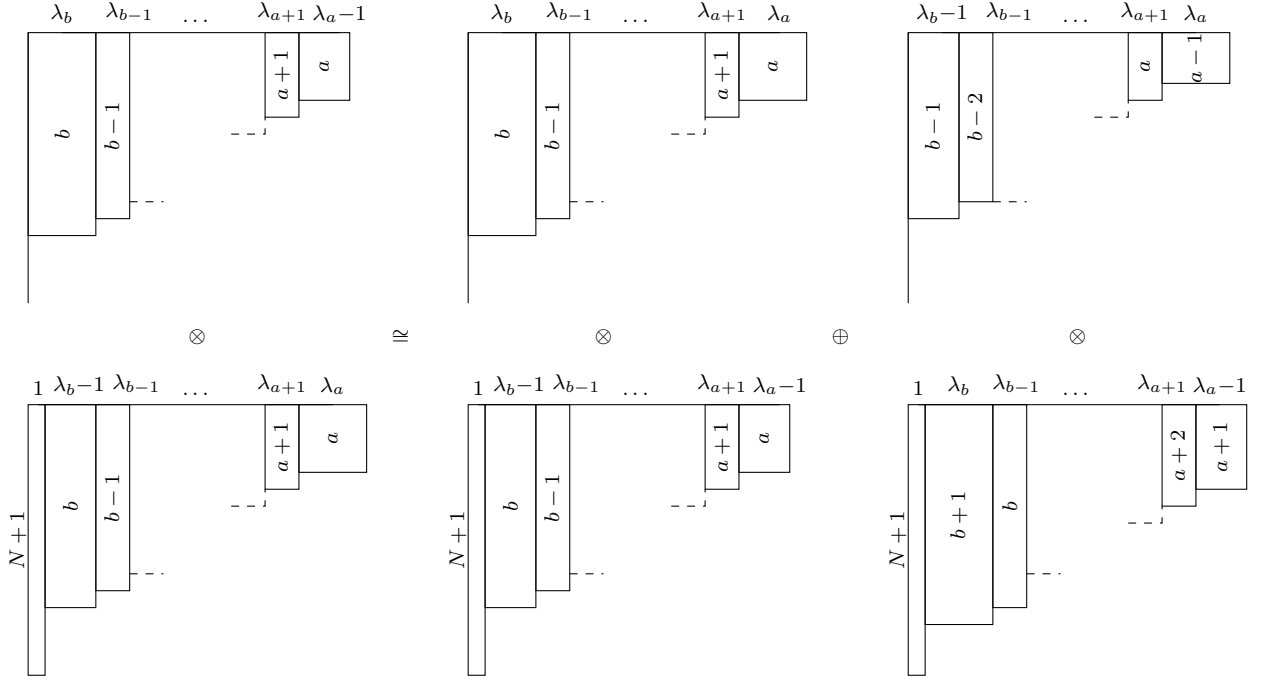
\begin{figure}
 \caption{\label{YDfig} Sketches of the Young diagrams for the evaluation modules appearing in (\ref{Aminaffsys}).}
\begin{tikzpicture}[rotate=0,scale=.45,every node/.style={font=\scriptsize}]
    \node at (9,-9)  {$\otimes$} ; 
     \node at (21,-9)  {$\otimes$} ; 
    \node at (35,-9)  {$\otimes$} ; 
    \node at (15,-9)  {$\cong$} ; 
    \node at (28,-9)  {$\oplus$} ;
\begin{scope}[yscale=-1,xshift=13cm]
\draw (5,0) node[above] {$\lambda_b$} -- (7,0) node[above] {$\lambda_{b-1}$} -- (9,0) node[above] {$\dots$} -- (11.5,0) node[above] {$\lambda_{a+1}$} -- (13.2,0)node[above] {$\lambda_a$}; 
\draw (4,0) rectangle 
  (4,8); \draw (4,0) rectangle node[rotate=90] {$b$}(6,6); \draw
      (6,0) rectangle node[rotate=90] {$b-1$}(7,5.5); \draw[dashed]
      (7,5)-- ++(1,0);\draw[dashed] (10,3) -- ++(1,0) -- ++(0,-.5);
      \draw (11,0) rectangle node[rotate=90] {$a+1$} ++(1,2.5); \draw
      (12,0) rectangle node[rotate=90] {$a$} ++(2,2);
    \end{scope}
    \begin{scope}[yscale=-1,xshift=13cm,yshift=11cm]
      \draw (4.3,0) node[above] {$1$} -- (5.5,0) node[above]
      {$\lambda_b{-}1$} -- (7.2,0) node[above] {$\lambda_{b-1}$} --
      (9,0) node[above] {$\dots$} -- (11.5,0) node[above]
      {$\lambda_{a+1}$} -- (13.2,0)node[above] {$\lambda_a{-}1$};
      \draw (4,0) rectangle node[above,rotate=90] {$N+1$} (4.5,8);
      \draw (4.5,0) rectangle node[rotate=90] {$b$}(6,6); \draw (6,0)
      rectangle node[rotate=90] {$b-1$}(7,5.5); \draw[dashed] (7,5)--
      ++(1,0);\draw[dashed] (10,3) -- ++(1,0) -- ++(0,-.5); \draw
      (11,0) rectangle node[rotate=90] {$a+1$} (12,2.5); \draw (12,0)
      rectangle node[rotate=90] {$a$} (13.5,2);
    \end{scope}
    \begin{scope}[yscale=-1,xshift=0cm]
      \draw (5,0) node[above] {$\lambda_b$} -- (7,0) node[above]
      {$\lambda_{b-1}$} -- (9,0) node[above] {$\dots$} -- (11.5,0)
      node[above] {$\lambda_{a+1}$} -- (13.2,0)node[above]
      {$\lambda_a{-}1$}; \draw (4,0)
      rectangle 
      (4,8); \draw (4,0) rectangle node[rotate=90] {$b$}(6,6); \draw
      (6,0) rectangle node[rotate=90] {$b-1$}(7,5.5); \draw[dashed]
      (7,5)-- ++(1,0);\draw[dashed] (10,3) -- ++(1,0) -- ++(0,-.5);
      \draw (11,0) rectangle node[rotate=90] {$a+1$} (12,2.5); \draw
      (12,0) rectangle node[rotate=90] {$a$} (13.5,2);
    \end{scope}
    \begin{scope}[yscale=-1,xshift=0cm,yshift=11cm]
      \draw (4.3,0) node[above] {$1$} -- (5.5,0) node[above]
      {$\lambda_b{-}1$} -- (7.2,0) node[above] {$\lambda_{b-1}$} --
      (9,0) node[above] {$\dots$} -- (11.5,0) node[above]
      {$\lambda_{a+1}$} -- (13,0)node[above] {$\lambda_a$}; \draw
      (4,0) rectangle node[above,rotate=90] {$N+1$} (4.5,8); \draw
      (4.5,0) rectangle node[rotate=90] {$b$}(6,6); \draw (6,0)
      rectangle node[rotate=90] {$b-1$}(7,5.5); \draw[dashed] (7,5)--
      ++(1,0);\draw[dashed] (10,3) -- ++(1,0) -- ++(0,-.5); \draw
      (11,0) rectangle node[rotate=90] {$a+1$} (12,2.5); \draw (12,0)
      rectangle node[rotate=90] {$a$} (14,2);
    \end{scope}
    \begin{scope}[yscale=-1,xshift=26cm]
      \draw (5,0) node[above] {$\lambda_b{-}1$} -- (7,0) node[above]
      {$\lambda_{b-1}$} -- (9,0) node[above] {$\dots$} -- (11,0)
      node[above] {$\lambda_{a+1}$} -- (12.5,0)node[above]
      {$\lambda_a$}; \draw (4,0)
      rectangle 
      (4,8); \draw (4,0) rectangle node[rotate=90] {$b-1$}(5.5,5.5);
      \draw (5.5,0) rectangle node[rotate=90] {$b-2$}(6.5,5);
      \draw[dashed] (6.5,5)-- ++(1,0);\draw[dashed] (9.5,2.5) --
      ++(1,0) -- ++(0,-.5); \draw (10.5,0) rectangle node[rotate=90]
      {$a$} (11.5,2); \draw (11.5,0) rectangle node[rotate=90] {$a-1$}
      (13.5,1.5);
    \end{scope}
    \begin{scope}[yscale=-1,xshift=26cm,yshift=11cm]
      \draw (4.3,0) node[above] {$1$} -- (5.5,0) node[above]
      {$\lambda_b$} -- (7.2,0) node[above] {$\lambda_{b-1}$} -- (9,0)
      node[above] {$\dots$} -- (11.5,0) node[above] {$\lambda_{a+1}$}
      -- (13.2,0)node[above] {$\lambda_a{-}1$}; \draw (4,0) rectangle
      node[above,rotate=90] {$N+1$} (4.5,8); \draw (4.5,0) rectangle
      node[rotate=90] {$b+1$}(6.5,6.5); \draw (6.5,0) rectangle
      node[rotate=90] {$b$}(7.5,6); \draw[dashed] (7.5,5)--
      ++(1,0);\draw[dashed] (10.5,3.5) -- ++(1,0) -- ++(0,-.5); \draw
      (11.5,0) rectangle node[rotate=90] {$a+2$} (12.5,3); \draw
      (12.5,0) rectangle node[rotate=90] {$a+1$} (14,2.5);
    \end{scope} \end{tikzpicture}
\end{figure}

\subsection{Minimal affinizations and wrapping modules in type B}\label{wrapmods}
In this subsection we work in type $B_N$. Recall (\confer \S\ref{minaffprop}) that if  $(i_t,k_t)\in \It$, $1\leq t \leq \T$, $\T\in \Z_{\geq 0}$, is a minimal snake then $\L(\prod_{t=1}^\T Y_{i_t,k_t})$ is a minimal affinization if and only if $(i_t)_{1\leq t\leq \T}$ is monotonic. In type B, however, monotonicity of the sequence $(i_t)_{1\leq t\leq \T}$ is not in general preserved by the relations in Theorem \ref{Tsys}: if a snake has this property, its neighbours may not. This can be seen by considering Figure \ref{snakefig} and will be manifest in (\ref{notclose}) below. 
This motivates us to make the following definition.
\begin{defn}
We say a minimal snake module $\L(\prod_{t=1}^\T Y_{i_t,k_t})$ is a \emph{wrapping module} if and only if the first coordinates of the points $\big(\iota(i_t,k_t) \big)_{1\leq t\leq \T}$, \confer (\ref{iotadef}), form a monotonic sequence. 
\end{defn}
It is easy to see the following proposition.
\begin{prop}
\begin{enumerate}[(i)]
\item Every minimal affinization is a wrapping module.
\item The relations of Theorem \ref{Tsys} close on the class of wrapping modules.\qed
\end{enumerate}
\end{prop}
Now we write explicitly the relations of Theorem \ref{Tsys} for wrapping modules.

For convenience, we let $r_0:=2$. Given $(a,k)\in \Z\times \Z$ and $s\in \Z_{\geq 1}$ such that $0\leq a \leq a+s-1\leq N$, and $(n_1,\dots,n_s)\in \Z_{\geq 0}^{s}$, let $k_1:=k$ and for each $2\leq i<s$, $k_{i+1} := k_i + 2n_{i}r_{a-1+i} -r_{a-1+i} + r_{a+i} + 2\max(r_{a-1+i},r_{a+i})$, and then define
\begin{align*} \W a k {n_1,n_2,\dots,n_{s}} 
&:= \left[\L (\st a {k_1} {n_1} \st {a+1} {k_{2}} {n_{2}} \dots \st {a+s-1} {k_{s}} {n_{s}})\right]. \end{align*}
Similarly, given $(a,k)\in \Z\times\Z$ and $s\in \Z_{\geq 1}$ such that $0\leq a-s+1 \leq a\leq N$, and $(n_1,\dots,n_s)\in \Z_{\geq 0}^{s}$, let $\tilde k_1:=k$ and for each $2\leq i<s$, $\tilde k_{i+1} := \tilde k_i + 2n_{i}r_{a+1-i} -r_{a+1-i} + r_{a-i} + 2\max(r_{a+1-i},r_{a-i})$, and then define
\begin{align*} 
\Wt {a} {k} {n_1,n_2,\dots,n_{s}} 
&:= \left[\L (\st {a} {\tilde k_{1}} {n_{1}} \st {a-1} {\tilde k_{2}} {n_{2}} \dots \st {a-s+1} {\tilde k_s} {n_s})\right]. \end{align*}

Next, given $a,b,\ell\in \Z$ such that $0\leq a,b\leq N-1$, and given 
\be\nn  (n_1,n_2,\dots,n_{N-a})\in \Z_{\geq 0}^{N-a},\quad (\bar n_1,\bar n_2,\dots,\bar n_{N-b})\in \Z_{\geq 0}^{N-b},\quad\text{ and }\quad n\in \Z_{\geq 0},\ee let: $\ell_1:=\ell$, $\ell_{i+1} := \ell_i + 4n_i +2$ for each $1\leq i<N-1-a$, $\ell_N:=\ell_{N-a} + 4n_{N-a}+1$, $\bar \ell_{N-b} :=  \ell_N + 2(2n+1)+3$, $\bar \ell_{i-1} = \bar \ell_i + 4n_i +2$ for each $N-b\geq i > 1$. 
Then we define
\be \Ww a \ell {n_1,\dots,n_{N-a};2n+1; \bar n_{N-b}, \dots, \bar n_2, \bar n_1} := \Big[ \L ( \prod_{i=1}^{N-1-a} \st {a+i-1} {\ell_i} {n_{i}} \cdot \st {N} {\ell_N} {2n+1} \cdot \prod_{j = 1}^{N-1-b} \st {b+j-1}  {\bar \ell_j} {\bar n_{j}}) \Big]. \nn\ee 


For any $\Lambda = \sum_{i\in I} \lambda_i \omega_i\in P^+\setminus\{0\}$ the minimal affinizations of $V(\Lambda)$ are $\W a k {\lambda_a,\lambda_{a+1},\dots,\lambda_{b-1}, \lambda_b}$ and $\Wt a k {\lambda_b,\lambda_{b-1},\dots,\lambda_{a+1}, \lambda_a}$, where $a=\min\{i\in I:\lambda_i>0\}$ and $b=\max\{i\in I: \lambda_i>0\}$. Suppose $b-a\geq 2$ (if not, we have the relations in \S\ref{KR} or \S\ref{twonodes}). Then
\begin{align}\label{Bminaffsys}
\W a k {\lambda_a,\lambda_{a+1},\dots,\lambda_{b-1}, \lambda_b-1} 
\W a {k+4} {\lambda_a-1,\lambda_{a+1},\dots,\lambda_{b-1}, \lambda_b} 
&=
\W a k {\lambda_a,\lambda_{a+1},\dots,\lambda_{b-1}, \lambda_b} 
\W a {k+4} {\lambda_a-1,\lambda_{a+1},\dots,\lambda_{b-1}, \lambda_b-1} + X Y
\nn\end{align}
where 
\be X = \begin{cases} 
\W {a-1} {k+2} {\lambda_a,\dots, \lambda_{b-1}, \lambda_b-1} & b <N \\
\W {a-1} {k+2} {\lambda_a,\dots, \lambda_{N-1}, \lfloor\frac{\lambda_N-1}{2}\rfloor} & b =N, \end{cases}\qquad
Y = \begin{cases}
\W {a+1} {k+2} {\lambda_{a}-1,\lambda_{a+1},\dots,\lambda_b} & b<N-1\\
\W {a+1} {k+2} {\lambda_{a}-1,\lambda_{a+1},\dots,\lambda_{N-2},2\lambda_{N-1}} & b=N-1\\
\Ww {a+1} {k+2} {\lambda_{a}-1,\lambda_{a+1},\dots,\lambda_{N-2};2\lambda_{N-1}+1;\lfloor\frac{\lambda_N}{2}\rfloor} & b=N;\end{cases} \label{notclose}\ee
and
\begin{align}\label{Bminaffsys}
\Wt b k {\lambda_b,\lambda_{b-1},\dots,\lambda_{a+1}, \lambda_a-1} 
\Wt b {k+2r_b} {\lambda_b-1,\lambda_{b-1},\dots,\lambda_{a+1}, \lambda_a} 
&=
\Wt b k {\lambda_b,\lambda_{b-1},\dots,\lambda_{a+1}, \lambda_a} 
\Wt b {k+2r_b} {\lambda_b-1,\lambda_{b-1},\dots,\lambda_{a+1}, \lambda_a-1} + \widetilde X \widetilde Y,
\nn\end{align}
where, recalling the map $s$ from \S\ref{twonodes}, we have
\be 
\widetilde X = \begin{cases}
\Wt {b+1} {k+2} {\lambda_{b},\lambda_{b-1},\dots,\lambda_{a+1},\lambda_a-1} & b<N-1\\
\Wt {N}  {k+1} {2\lambda_{N-1},\lambda_{N-2},\dots,\lambda_{a+1},\lambda_a-1} & b=N-1\\
\Ww {N-1} {k+1+2s(\lambda_N)}{\lfloor\frac{\lambda_N}{2}\rfloor; 2\lambda_{N-1}+1; \lambda_{N-2},\dots \lambda_{a+1},\lambda_a}  & b=N,\end{cases}\qquad
\widetilde Y = \begin{cases} 
\Wt {b-1} {k+2} {\lambda_b-1,\lambda_{b-1},\dots, \lambda_{a+1}, \lambda_a} & b <N \\
\Wt {N-1} {k+1+2s(\lambda_N-1)} {\lfloor\frac{\lambda_N-1}{2}\rfloor,\lambda_{N-1},\dots,\lambda_{a+1},\lambda_a} & b =N. \end{cases}
 \nn\ee

Finally, given $(\lambda_1,\dots,\lambda_{N-1};\lambda;\bar \lambda_{N-1},\dots,\bar\lambda_2,\bar \lambda_1)\in \Z_{\geq 0}^{2N-1}$, let $a:=\min(\{i\in I:\lambda_i > 0 \}\cup \{N\})$ and $b:=\min(\{i\in I:\bar \lambda_i>0\}\cup\{N\})$. Suppose $a\leq N-1$ and $b\leq N-1$. Then for all $k$ such that $(a,k)\in \It$ we have
\begin{align} 
&{} \Ww a k {\lambda_a,\lambda_{a+1},\dots,\lambda_{N-1};2\lambda+1;\bar \lambda_{N-1},\dots, \bar\lambda_{b+1}, \bar\lambda_b-1}
\Ww a {k+2} {\lambda_a-1,\lambda_{a+1},\dots,\lambda_{N-1};2\lambda+1;\bar \lambda_{N-1},\dots, \bar\lambda_{b+1}, \bar\lambda_b}\nn\\
& =
\Ww a k {\lambda_a,\lambda_{a+1},\dots,\lambda_{N-1};2\lambda+1;\bar \lambda_{N-1},\dots, \bar\lambda_{b+1}, \bar\lambda_b}
\Ww a {k+2} {\lambda_a-1,\lambda_{a+1},\dots,\lambda_{N-1};2\lambda+1;\bar \lambda_{N-1},\dots, \bar\lambda_{b+1}, \bar\lambda_b-1} 
+ \overline X\,\overline Y,\label{wraprel}
\end{align} 
where 
\begin{align*}
    \overline X&= \begin{cases} \Ww {a-1} {k+2} {\lambda_a,\lambda_{a+1},\dots,\lambda_{N-1},\lambda;2\bar \lambda_{N-1}+1;\bar\lambda_{N-2},\dots, \bar\lambda_{b+1}, \bar\lambda_b-1} & b<N-1 \\
      \W {a-1} {k+2} {\lambda_a,\lambda_{a+1},\dots,\lambda_{N-1},\lambda,2\bar \lambda_{N-1}-1} & b = N-1, \end{cases}\\
    \overline Y&= \begin{cases} \Ww {a+1} {k+2}{\lambda_a-1,\lambda_{a+1},\dots,\lambda_{N-2};2\lambda_{N-1}+1;\lambda,\bar\lambda_{N-1},\dots, \bar\lambda_{b+1}, \bar\lambda_b} &a < N-1 \\
      \Wt N {k+1} {2\lambda_{N-1}-1, \lambda,\bar \lambda_{N-1},\dots,\bar\lambda_{b+1},\bar\lambda_b} &      a = N-1. \end{cases}
 \end{align*}

For all  $s>s'>1$, these relations determine the $\W a k {n_1,\dots,n_s}$, $\Wt a k {n_1,\dots, n_s}$ and $\Ww a k {n_1,\dots,n_{r-1};2n_{r}+1;n_{r+1},\dots,n_s}$ ($1<r<s$) in terms of the $\W b \ell {m_1,\dots,m_{s'}}$, $\Wt b \ell {m_1,\dots, m_{s'}}$ and $\Ww b \ell {m_1,\dots,m_{r'-1};2m_{r'}+1;m_{r'+1},\dots,m_{s'}}$ ($1<r'<s'$).

\section{Wrapping modules in type $B_2$.}
As an application of Theorem \ref{Tsys}, let us compute the dimensions and $\uqbt$-module decompositions of the wrapping modules (\S\ref{wrapmods}) in type $B_2$. 
For all $m,k,n\in \Z_{\geq 0}$, let 
\bea  \WQ {m,k,n} &:=& 
   \res \Big[\L(Y_{1,0}Y_{1,4}\dots Y_{1,4m-4} \nn\\
&&     \qquad\quad Y_{2,4m+1} Y_{2,4m+3},\dots Y_{2,4m+2k-1}    \nn\\
 &&    \qquad\quad Y_{1,4m+2k+4} Y_{1,4m+2k+8}\dots Y_{1,4m+2k+4n})\Big]\nn\eea
and $\WQ {m,k}   := \WQ{m,k,0}$. Note that $\WQ {m,k,n} = \WQ {n, k, m}$.

\begin{prop}\label{b2Qsys} For all $m,k,n\in \Z_{\geq 0}$ we have 
\begin{align*} \WQ{m,2k+1,n+1} \WQ{m+1,2k+1,n} &= \WQ{m+1,2k+1,n+1} \WQ{m,2k+1,n} + \WQ{k,2n+1} \WQ{k,2m+1},\\
\WQ{m+1,2k} \WQ{m,2k+1} &= \WQ{m+1,2k+1} \WQ{m,2k} + \WQ{k,0} \WQ{k,2m+1} ,\\
(\WQ{m+1,0})^2&= \WQ{m+2,0} \WQ{m,0} + \WQ{0,2m+2},\\
(\WQ{0,k+1})^2&= \WQ{0,k+2} \WQ{0,k} + \WQ{\lfloor\frac{k+1}{2}\rfloor,0} 
                                       \WQ{\lfloor\frac{k+2}{2}\rfloor,0} .
\end{align*}
\end{prop}
\begin{proof}
For all $\ell\in\Z$ and all $m,k\in \Z_{\geq 0}$ we have $\WQ {m,2k+1,n} = \res \Ww 1 \ell {m,2k+1,n}$ for all $n>0$ and  $\WQ {m,k} = \res \W 1 \ell {m,k}= \res\Wt2 {\ell}{k,m}$. In particular $\WQ{m,0} = \res\W1 \ell m$ and $\WQ{0,k} = \res\W2 \ell k$.  The  first  two equalities therefore follow from (\ref{wraprel}) and (\ref{btwonoderel}) respectively.
The final two equalities, which are the usual $Q$-system in type $B_2$, follow from (\ref{BendTsys}). 
\end{proof}
\begin{prop}\label{b2Qsoln} For all $m,k,n\in \Z_{\geq 0}$,
\begin{align} \WQ{m,2k+1} &\cong \bigoplus_{i=0}^{k} V(m\omega_1 + (2i+1)\omega_2) \qquad
\WQ{m,2k} \cong \bigoplus_{i=0}^{k} V(m\omega_1 + 2i\omega_2) \label{decomp1}\\
\WQ{m,2k+1,n} &\cong \bigoplus_{i=0}^{\min(m,n)} \bigoplus_{j=0}^k V\left(\left(m+n-2i\right)\omega_1+ \left(2i+2k-2j+1\right)\omega_2\right). \label{decomp2} 
\end{align}
\end{prop}
\begin{proof} We have $\WQ{1,0} \cong V(\omega_1)$ and $\WQ{0,1} \cong V(\omega_2)$. The solution to the relations of Proposition \ref{b2Qsys} with these initial conditions is unique, so it is enough to check that (\ref{decomp1}-\ref{decomp2}) is it. Since the $\uqbt$-character homomorphism $\chi$ is injective it is enough to check this at the level of $\uqbt$-characters. To do so one can use the Weyl character formula: let $\rho$ be the Weyl vector, $W$ the Weyl group, $\ell:W\to \Z_{\geq 0}$ the length function, and $\Lambda\in P^+$; then 
\be \chi(V(\Lambda)) = \frac{e^{-\rho}}{\prod_{\alpha>0}(1-e^{-\alpha})}N(\Lambda), \quad \text{where}\quad N(\Lambda) := \sum_{w\in W} (-1)^{\ell(w)} e^{w(\Lambda + \rho)}. \ee
In type $B_2$, $\rho = \omega_1+\omega_2$, 
$W = \{\id, \sigma_1 , \sigma_2,\sigma_1\sigma_2 , \sigma_2\sigma_1, \sigma_1\sigma_2\sigma_1 , \sigma_2\sigma_1\sigma_2, \sigma_1\sigma_2\sigma_1\sigma_2\}$
and the numerator $N(\Lambda)$ is a sum of $|W|=8$ terms: 
\begin{align}\nn N( \lambda_1\omega_1+\lambda_2\omega_2)=\hspace{-2cm} &\\\nn & 
 {y_1}^{\lambda_1+1}\,{y_2}^{\lambda_2+1}
-{y_1}^{-\lambda_1-1}\,{y_2}^{\lambda_2+2\,\lambda_1+3}
-{y_1}^{\lambda_2+\lambda_1+2}\,{y_2}^{-\lambda_2-1}
+{y_1}^{-\lambda_2-\lambda_1-2}\,{y_2}^{\lambda_2+2\,\lambda_1+3}\\\nn &
+{y_1}^{\lambda_2+\lambda_1+2}\,{y_2}^{-\lambda_2-2\,\lambda_1-3}
-{y_1}^{-\lambda_2-\lambda_1-2}\,{y_2}^{\lambda_2+1}
-{y_1}^{\lambda_1+1}\,{y_2}^{-\lambda_2-2\,\lambda_1-3}
+{y_1}^{-\lambda_1-1}\,{y_2}^{-\lambda_2-1},\end{align}
where $y_i:=e^{\omega_i}$, $i=1,2$.  
To obtain the numerator of $\chi(\WQ{m,2k+1,n})$ one can perform the sum from (\ref{decomp2}) in each of these 8 terms, since these are geometric progressions.  The result is a sum of 8 explicit rational functions of the $(y^{\pm 1}_i)_{i\in I}$; for example, the rational function obtained by summing the $w=\id$ terms is \be\nn y_1^{n+m+1}\,y_2^{2k+2}\,\frac{y_2^{-2( k+1) }-1 }{{y_2}^{-2}-1}
\frac{(y_2/y_1)^{2\left( \min(m,n)+1\right)}-1}{ (y_2/y_1)^{2}-1}.\ee
It is then a direct finite calculation, which we have performed with the aid of the computer algebra system Maxima, to verify that the numerators obey the relations of Proposition \ref{b2Qsys}. 
\end{proof}
The $\uqbt$-decompositions of the minimal affinizations in type $B_2$, i.e. the above expressions for $\WQ{m,k}$, can be found in \cite{Cminaffrank2}. 

\begin{cor} For all $m,k,n\in \Z_{\geq 0}$,
\be\dim\left(\WQ{m,2k+1,n}\right) = \frac{1}{3} (k+1)(n+1)(m+1)(n+k+2)(m+k+2)(n+m+k+3).\nn\ee
\end{cor} 
\begin{proof}
This follows from Proposition \ref{b2Qsoln} using the Weyl dimension formula.
\end{proof}

\begin{rem} Similar methods can also be used to obtain, for all $m,k,n\in \Z_{\geq 0}$, the decomposition of $\WQ{m,2k,n}$ into simple $\uqbt$-modules:
\be \WQ{m,2k,n} \cong \bigoplus_{i=0}^{\min(m,n)} \bigoplus_{j=0}^{i+k} V\left(\left(m+n-2i\right)\omega_1+ \left(2i+2k-2j\right)\omega_2\right).\nn\ee 
$\WQ{m,2k,n}$ is not the restriction of a wrapping module. 
We have found with the aid of a computer algebra system that, in contrast to $\WQ{m,2k+1,n}$, its dimension does not factor fully to linear factors with integer coefficients.
\end{rem}
Some 3-term relations among minimal affinizations in type $B_2$, different from the ones in the present paper, can be found in  \cite{MPqcharsminaffs}.

\section{Paths and moves}\label{sec:paths}
In order to prove Theorem \ref{Tsys} we first recall from \cite{MY1} the closed form for the $q$-characters of snake modules in terms of non-overlapping paths. 
\subsection{Paths and corners} 
For each $(i,k)\in \It$ we shall define a set  $\scr P_{i,k}$ of paths. For us, a \emph{path} is a  finite sequence of points in the plane $\R^2$. 
We write $(j,\ell)\in p$ if $(j,\ell)$ is a point of the path $p$.

\subsubsection*{Paths of Type A} For all $(i,k)\in \It$, let
\bea \scr P_{i,k} := \Big\{ \big( (0,y_0), (1,y_1), \dots, (N+1,y_{N+1}) \big) &:& y_0 = i+k,\,\, y_{N+1} = N+1-i+k,\nn\\ &\text{and}& y_{i+1}-y_i \in \{ 1, -1 \} \,\,\forall 0\leq i \leq N \Big\} . \nn\eea
We define the sets $C_{p,\pm}$ of upper and lower \emph{corners} of a path $p=\big( (r,y_r) \big)_{0\leq r\leq N+1} \in \scr P_{i,k}$ to be
\bea C_{p}^{+} &:=& \left\{ (r,y_r) \in p:  r\in I,\,y_{r-1} = y_r+1  = y_{r+1}\right\}, \nn\\
   C_{p}^{-} &:=& \left\{ (r,y_r) \in p:  r\in I,\,y_{r-1} = y_r-1  = y_{r+1}\right\}  \nn.\eea

\subsubsection*{Paths of Type B}  
Pick and fix an $\eps$, $1/2>\eps>0$. We first define $\scr P_{N,\ell}$ for all $\ell\in 2\Z+1$ as follows. 
\begin{itemize}\item
For all $\ell\equiv 3\!\!\mod 4$,
\bea \scr P_{N,\ell} &:=& \Big\{\big( (0,y_0), (2,y_1), \dots, (2N -4,y_{N-2}), (2N-2,y_{N-1}), (2N-1,y_{N})\big) \nn\\ &&\quad: \quad y_0 = \ell+2N-1,  y_{i+1}-y_i \in\{2,-2\} \,\,\forall 0\leq i\leq N-2  \nn \\ &&\qquad\qquad\qquad\quad\text{and}\quad y_{N}-y_{N-1} \in \{1+\eps,-1-\eps \} \Big\}.\nn\eea
\item
For all $\ell\equiv 1\!\!\mod 4$,
\bea \scr P_{N,\ell} &:=& \Big\{\big( (4N-2,y_0), (4N-4,y_1), \dots, (2N +2,y_{N-2}), (2N,y_{N-1}), (2N-1,y_{N})\big) \nn\\ &&\quad: \quad y_0 = \ell+2N-1,  y_{i+1}-y_i \in\{2,-2\} \,\,\forall 0\leq i\leq N-2  \nn \\ &&\qquad\qquad\qquad\quad\text{and}\quad y_{N}-y_{N-1} \in \{1+\eps,-1-\eps \} \Big\}.\nn \eea
\end{itemize}
Next we define $\scr P_{i,k}$ for all $(i,k)\in \It$, $i<N$, as follows.
\bea \scr P_{i,k} := \Big\{ (a_0,a_1, \dots, a_{N}, \bar a_{N} , \dots ,\bar a_1,\bar a_0) &:& (a_0,a_1,\dots,a_N) \in \scr P_{N,k-(2N-2i-1)},\nn\\&& (\bar a_0,\bar a_1,\dots,\bar a_N) \in \scr P_{N,k+(2N-2i-1)},\label{bvectpathdef1}\\
&\text{and}& a_N-\bar a_N = (0,y) \quad\text{where}\quad y>0 \Big\} .\nn\eea
For all $(i,k)\in \It$, we define the sets of upper and lower \emph{corners} $C_{p}^{\pm}$ of a path $p= \big( (j_r,\ell_r) \big)_{0\leq r\leq |p|-1} \in \scr P_{i,k}$, where $|p|$ is the number of points in the path $p$, as follows:
\bea C_{p}^{+} &:= & \iota^{-1}\Big\{ (j_r,\ell_r) \in p:  j_r\notin \{0,2N-1,4N-2\}, \,\ell_{r-1} > \ell_r,\,\,  \ell_{r+1}>\ell_r \Big\}\nn \\ 
           && {}\sqcup \{(N,\ell)\in \It: (2N-1,\ell-\eps)\in p 
                        \text{ and } (2N-1,\ell+\eps) \notin p \}\nn ,\eea
\bea C_{p}^{-} &:= & \iota^{-1}\Big\{ (j_r,\ell_r) \in p:   j_r\notin \{0,2N-1,4N-2\},\,\ell_{r-1} < \ell_r,\,\, \ell_{r+1}< \ell_r\Big\}\nn \\ 
           && {}\sqcup \{(N,\ell)\in \It:  (2N-1,\ell+\eps)\in p 
                          \text{ and } (2N-1,\ell-\eps) \notin p \}\nn .\eea
where $\iota$ is the map defined in (\ref{iotadef}). Note that $C_{p}^{\pm}$ is a subset of $\It$.

We define a map $\mon$ sending paths to monomials, as follows:
\be \mon : \bigsqcup_{(i,k)\in \It} \scr P_{i,k} \longrightarrow \Z\left[Y_{j,\ell}^{\pm 1}\right]_{(j,\ell)\in \It}\,\,;\quad
 p \mapsto \mon(p) := \prod_{(j,\ell)\in C_{p}^{+} } \YY j \ell \prod_{(j,\ell)\in C_{p}^{-} } \MM j \ell.\label{mondef}\ee

\subsection{Lowering and raising moves}\label{lowdef}
Let $(i,k)\in \It$ and $(j,\ell)\in \Iw$. We say a path $p\in \scr P_{i,k}$ can be \emph{lowered}  at $(j,\ell)$ if and only if $(j,\ell-r_j)\in C_{p}^{+}$ and $(j,\ell+r_j)\notin C_{p}^{+}$. 
If so, we define a \emph{lowering move} on $p$ at $(j,\ell)$, resulting in another path in $\scr P_{i,k}$ which we write as $p\scr A_{j,\ell}^{-1}$ and which is defined to be the unique path such that $\mon(p\scr A_{j,\ell}^{-1}) = \mon(p) A_{j,\ell}^{-1}$. A detailed case-by-case description of these moves can be found in \cite{MY1}, Section 5. 

Let $(i,k)\in \It$ and $(j,\ell)\in \Iw$. We say a path $p\in \scr P_{i,k}$ can be \emph{raised at} $(j,\ell)$ if and only if $p=p'\scr A_{j,\ell}^{-1}$ for some $p'\in \scr P_{i,k}$. If $p'$ exists it is unique, and we define $p\scr A_{j,\ell}:=p'$. It is straightforward to verify that $p$ can be raised at $(j,\ell)$ if and only if  $(j,\ell+r_j)\in C_{p}^{-}$ and $(j,\ell-r_j)\notin C_{p}^{-}$. 

\subsection{The highest/lowest path}\label{sec:highestpath} For all $(i,k)\in \It$, define $\phigh_{i,k}$, the \emph{highest path} to be the unique path in $\scr P_{i,k}$ with no lower corners. 
Equivalently, $\phigh_{i,k}$ is the unique path such that:
\begin{align}\text{Type A}& :  &(i,k) &\in \phigh_{i,k}  \nn\\
\text{Type B},&\,\, i<N:  &\iota(i,k) &\in \phigh_{i,k} \nn\\
\text{Type B},&\,\, i=N:  &(2N-1,k) - (0,\eps)&\in \phigh_{N,k} .\nn\end{align} 
Define $\plow_{i,k}$, the \emph{lowest path}, to be the unique path in $\scr P_{i,k}$ with no upper corners. 
Equivalently, $\plow_{i,k}$ is the unique path such that:
\begin{align}\text{Type A}& :  &(N+1-i,k+N+1) &\in \plow_{i,k}  \nn\\
\text{Type B},&\,\, i<N:  &\iota(i,k+4N-2) &\in \plow_{i,k} \nn\\
\text{Type B},&\,\, i=N:  &(2N-1,k+4N-2) + (0,\eps)&\in \plow_{N,k} .\nn\end{align}  

\subsection{The snake-lowered path}\label{snakelowereddef}
Suppose $(i,k)\in \It$ and $(i',k')\in \It$ are such that $(i',k')$ is in prime snake position with respect to $(i,k)$. 
There is a unique path in $\scr P_{i,k}$ that has a lower corner at $(i',k')$ and no other lower corners: we call this path the \emph{snake-lowered path} from $(i,k)$ to $(i',k')$, and denote it $\psnake_{i,k; i',k'}$. 
By construction, the set of upper corners of $\psnake_{i,k;i',k'}$ is precisely the disjoint union $\nbri_{i,k}^{i',k'}\sqcup\nbrii_{i,k}^{i',k'}$ of the neighbouring points, as defined in \S\ref{sec:nbrdef}. 

Given any two paths $p=(x_r,y_r)_{1\leq r\leq n}$ and $p'=(x_r,y'_r)_{1\leq r\leq n}$, $n\in\Z_{>0}$, in $\scr P_{i,k}$ we say $p$ is \emph{weakly above} (resp. \emph{weakly below}) $p'$ if and only if $y_r\leq y'_r$ (resp. $y_r\geq y'_r$) for all $1\leq r\leq n$. We also define 
\be \bott(p,p') :=  (x_r, \max(y_r,y'_r))_{1\leq r\leq n},\nn\ee
which is a path in $\scr P_{i,k}$ which is weakly below both $p$ and $p'$ (\confer Lemma 5.7 of \cite{MY1}).
\begin{lem}\label{bl1}
Let $p,p'\in \scr P_{i,k}$ such that $C^-_{\bott(p,p')}\subset \{(i',k')\}$ and $p\neq \psnake_{i,k;i',k'}$. Then $p=\phigh_{i,k}$ and $p'\in \{ \phigh_{i,k}, \psnake_{i,k;i',k'}\}$.  
\end{lem}
\begin{proof}
If $C^-_{\bott(p,p')}=\emptyset$ then $\bott(p,p') = \phigh_{i,k}$ and hence $p=p'=\phigh_{i,k}$. If  $C^-_{\bott(p,p')} = \{(i',k')\}$ then $\bott(p,p')=\psnake_{i,k;i',k'}$ and $(i',k')\in C^-_{p'}$. By inspection the only path weakly above $\psnake_{i,k;i',k'}$ with a lower corner at $(i',k')$ is $\psnake_{i,k;i',k'}$ itself. Thus $p'=\psnake_{i,k;i',k'}$. By inspection, if $(j,\ell)\in C^+_{\psnake_{i,k;i',k'}}$ then $(j,\ell)\notin C^-_{p}$, for any $p\in \scr P_{i,k}$. So finally $C^-_p=\emptyset$, i.e. $p=\phigh_{i,k}$. 
\end{proof}

\subsection{Points above/below paths}
We say a point $(x,y)$ is \emph{strictly above} (resp. \emph{strictly below}) a path $p$ if and only if $z>y$ (resp. $z<y$) for all $z$ such that $(x,z)\in p$. 

We say a point $(x,y)$ is \emph{weakly above} (resp. \emph{weakly below}) a path $p$ if and only if, for any point $(x,z)\in p$ there is a point $(x,z')\in p$ such that $z'\geq y$ (resp. $z'\leq y$).

\subsection{Non-overlapping paths}\label{overlapdef}
Let $p,p'$ be paths. We say $p$ is \emph{strictly above} $p'$, and $p'$ is \emph{strictly below} $p$, if and only if
\be  (x,y)\in p \text{ and } (x,z)\in p'  \implies  y< z .\nn\ee
We say a $T$-tuple of paths $(p_1,\dots,p_T)$ is \emph{non-overlapping} if and only if $p_s$ is strictly above $p_t$ for all $s<t$. Otherwise, for some $s<t$ there exist $(x,y)\in p_s$ and $(x,z)\in p_t$ such that $y\geq z$, and we say \emph{ $p_s$ overlaps $p_t$ in column $x$}. 

For any snake $(i_t,k_t)\in \It$, $1\leq t \leq T$, $T\in \Z_{\geq 1}$, let us define 
\be \nops := \left\{(p_1,\dots,p_T): p_t\in \scr P_{i_t,k_t}, 1\leq t\leq T\,,  (p_1,\dots,p_T)\text{ is non-overlapping } \right\}. \nn\ee

Generically no two corners of any tuple $\ps\in\nops$ of non-overlapping paths coincide. The only exception is in type B where it can happen that, for some $t$, $1\leq t\leq \T-1$, a point $(N,\ell)$ is an upper corner of $p_{i_t,k_t}$ and a lower corner of $p_{i_{t+1}, k_{t+1}}$. (See Figure 7 in \cite{MY1}.) But in that case $p_{i_t,k_t}$ has a lower corner at some $(j',\ell')\in \It$ with $\ell'>\ell$. Thus we have

\begin{lem}\label{rneglem}
Let $(i_t,k_t)\in \It$, $1\leq t \leq \T$, be a snake of length $\T\in \Z_{\geq 1}$ and $\ps \in \nops$. If 
$(j,\ell)$ is a lower corner of some path $p_t$, $1\leq t\leq \T$ and no point $(j',\ell')\in \It$ such that $\ell'>\ell$ is a lower corner of any path in $\ps$,
then $(j,\ell)$ is not an upper corner of any path in $\ps$.  \qed
\end{lem}


\begin{lem}[\cite{MY1}]\label{movelemmaA} Let $p$ and $p'$ be paths in $\scr P_{i,k}$. Then
$p$ can be obtained from $p'$ by a sequence of moves containing no inverse pair of raising/lowering moves. \qed
\end{lem}
The following is Lemma 5.10 in \cite{MY1}.  
\begin{lem}\label{movelemma}
Let $(i_t,k_t)\in \It$, $1\leq t\leq \T$, be a snake of length $\T\in \Z_{\geq 1}$ and $(j_r,\ell_r)$, $1\leq r\leq R$,  a sequence of $R\in\Z_{\geq 0}$ points in $\Iw$. For all $\ps\in\nops$ and $\pps\in\nops$, the following are equivalent: 
\begin{enumerate}[(i)] 
\item $\prod_{t=1}^\T \mon(p'_t) =  \prod_{t=1}^\T \mon(p_t) \cdot \prod_{r=1}^R A_{j_r,\ell_r}^{-1}$ 
\item there is a permutation $\sigma \in S_R$ such that $\big( (j_{\sigma(1)},\ell_{\sigma(1)}),\dots, (j_{\sigma(R)},\ell_{\sigma(R)}) \big)$ is a sequence of lowering moves that can be performed on $\ps$, without ever introducing overlaps, to yield $\pps$. 
\end{enumerate}\qed
\end{lem}

\subsection{The path formula for $q$-characters of snake modules}
\label{sec:snakechar}
\begin{thm}[\cite{MY1}]
\label{snakechar} Let $(i_t,k_t)\in \It$, $1\leq t\leq \T$, be a snake of length $\T\in \Z_{\geq 1}$. Then
\be\chi_q\left(\L\left(\prod_{t=1}^\T Y_{i_t,k_t}\right)\right) = \sum_{\displaystyle \substack{(p_1,\dots,p_\T)\in \nops}} 
\prod_{t=1}^\T \mon(p_t).\ee
The module $\L(\prod_{t=1}^\T \YY {i_t}{k_t})$ is thin, special and anti-special.
\end{thm}
We can now prove the proposition on prime snake modules stated earlier: 
\begin{proof}[Proof of Proposition \ref{primesnakes}]
Let $(i_t,k_t)\in \It$, $1\leq t\leq \T$, be a snake of length $\T\in \Z_{\geq 0}$. If it is not a prime snake then there is an $s$, $1\leq s<\T$, such that $(i_{s+1},k_{s+1})$ is not in prime snake position to $(i_s,k_s)$. Theorem \ref{snakechar} and the injectivity of $\chi_q$ imply that 
$\L\left(\prod_{t=1}^\T Y_{i_t,k_t} \right) \cong \L\big(\prod_{t=1}^s Y_{i_t,k_t}\big)\otimes \L\big(\prod_{t=s+1}^\T Y_{i_t,k_t}\big)$. By recursion, any snake module is isomorphic to a tensor product of snake modules whose snakes are prime.

Now consider the case that the given snake is prime.  Let $U \sqcup U'=\{ (i_t,k_t) : 1\leq t\leq \T\}$ be any set partition with $U$ and $U'$ non-empty and, without loss of generality, $(i_1,k_1)\in U$. We shall now show that 
\be\label{wrng}\L\left(\prod_{t=1}^\T Y_{i_t,k_t}\right)\not\cong \L\left(\prod_{(i,k)\in U}Y_{i,k}\right)\otimes\L\left(\prod_{(i,k)\in U'}Y_{i,k}\right).\ee 
The points of $U$ may not form a snake but $\chi_q\left(\L\left(\prod_{(i,k)\in U}Y_{i,k}\right)\right)$ always includes the monomial  $\prod_{(i,k)\in U}\mon(\plow_{i,k})$: see e.g. \cite{CHbeyondKR} Theorem 3.
Therefore, the $q$-character of the right-hand side contains the monomial 
\be\label{mfdef} m:=\prod_{(i,k)\in U}\mon(\plow_{i,k}) \prod_{(i,k)\in U'} \mon(\phigh_{i,k}).\ee  
We claim that the $q$-character of the left-hand side of (\ref{wrng}) does not contain $m$. Indeed, suppose for a contradiction that $\ps \in \nops$ is such that $m=\prod_{t=1}^\T \mon(p_t)$. Note that $m$ has at most $\T$ factors $Y^{\pm 1}_{j,\ell}$. 

Consider type A. There is no cancellation between $\mon(p_t)$, $\mon(p_s)$ for any pair $t\neq s$. So each path must have at most one corner, or else $m$ would have too many factors. Therefore each path is either highest or lowest. Then (\ref{mfdef}) can only hold if for all $t\in U$, $p_t=\plow_{i_t,k_t}$ and for all $t\in U'$, $p_t=\phigh_{i_t,k_t}$. But there is an $s$ such that $(i_{s+1},k_{s+1})\in U'$ and $(i_s,k_s)\in U$, so, by definition of prime snake position, $\phigh_{i_{s+1},k_{s+1}}$ overlaps $\plow_{i_s,k_s}$: a contradiction.

Consider type B. Cancellations between $\mon(p_t)$, $\mon(p_s)$, $t\neq s$, can occur only in column $N$.
Every path $p_t$ has at least one non-cancelling corner. So, by counting, every path $p_t$ must have exactly one non-cancelled corner. Hence, if $i_t\neq N$ then $p_t$ must be highest or lowest. So for all $t\in U$ such that $i_t\neq N$, $p_t=\plow_{i_t,k_t}$ and for all $t\in U'$ such that $i_t\neq N$, $p_t=\phigh_{i_t,k_t}$. Dividing (\ref{mfdef}) by these factors, we have
\be\nn \prod_{t: i_t= N} \mon(p_t) = \prod_{t\in U: i_t =N} \mon(\plow_{N,k_t}) \prod_{t\in U': i_t=N} \mon(\phigh_{N,k_t}).\ee
But now if, for some $t$ such that $i_t=N$, $p_t$ is not highest or lowest then it has a corner not in column $N$, which produces some factor $Y_{j,\ell}^{\pm 1}$ that is neither cancelled nor present on the right-hand side of this equation. Therefore for all $t$, $p_t$ must be highest or lowest. So for all $t\in U$, $p_t=\plow_{i_t,k_t}$ and for all $t\in U'$, $p_t=\phigh_{i_t,k_t}$. But, as in type A, these paths are overlapping: a contradiction.
\end{proof}

\section{Proof of Theorem \ref{Tsys}}\label{sec:proof}
Our strategy of proof is influenced by \cite{HernandezKR}. We show that the dominant monomials on the left- and right-hand sides of (\ref{grl}) coincide. We shall see in particular that (\ref{nnsimp}) is special, and therefore simple. Finally, to show that (\ref{tbsimp}) is simple we show that for each of its dominant monomials $m$  other than the highest, the $q$-character of $\L(m)$ contains a monomial which is not present in the $q$-character of (\ref{tbsimp}).

\subsection{Classification of dominant monomials}
\begin{lem}\label{sll} Let $(i_t,k_t)\in \It$, $1\leq t \leq \T$, be a snake of length $\T\in \Z_{\geq 2}$.
Let $m' =\prod_{t=1}^{\T-1}\mon(p'_t)$ be a monomial of $\chi_q(\L( \prod_{t=1}^{\T-1} \YY{i_t}{k_t}))$, where $(p'_1,\dots,p'_{\T-1})\in \overline{\scr P}_{(i_t,k_t)_{1\leq t\leq \T-1}}$ are a tuple of non-overlapping paths as in Theorem \ref{snakechar}. Likewise, let $m=\prod_{t=2}^\T \mon(p_t)$ be a monomial of $\chi_q(\L (\prod_{t=2}^\T \YY{i_t}{k_t}))$, where $(p_2,\dots,p_{\T})\in \overline{\scr P}_{(i_t,k_t)_{2\leq t\leq \T}}$. 
Suppose $m'm$ is dominant. Then, $p_{t}=\phigh_{i_t,k_t}$ for all $2\leq t\leq \T$ and there exists and $R$, $1\leq R\leq \T$, such that $p'_{t} = \phigh_{i_t,k_t}$ for all $1\leq t<R$ and $p'_{t}= \psnake_{i_t,k_t;i_{t+1},k_{t+1}}$ for all $R\leq t< \T$.
\end{lem}
\begin{proof}
We shall show by induction on $t$ that, for all $t\in \{0,1,\dots,\T-1\}$:
$p_{\T-t} = \phigh_{i_{\T-t},k_{\T-t}}$ if $t< \T-1$;
$p'_{\T-t}\in \{\phigh_{i_{\T-t},k_{\T-t}}, \psnake_{i_{\T-t},k_{\T-t};i_{\T-t+1},k_{\T-t+1}}\}$ if $t> 0$;
and moreover if $p'_{\T-t}=\phigh_{i_{\T-t},k_{\T-t}}$ then $p'_{\T-t-1} = \phigh_{i_{\T-t-1},k_{\T-t-1}}$, if $\T-1>t>0$.
For the case $t=0$, if $p_\T\neq \phigh_{i_\T,k_\T}$ then $m'm$ is right-negative and hence not dominant. Now assume the statement is true for $t-1$. Consider the monomial $n := \mon(p_{\T-t})Y_{i_{\T-t+1},k_{\T-t+1}}\mon( p'_{\T-t})$. By virtue of Lemma \ref{bl1}, it is sufficient to show that $n$ is dominant.

Suppose for a contradiction that $n$ is not dominant. Let $(j,\ell)\in \It$ be such that $u_{j,\ell}(n)<0$ and for all $\ell'>\ell$ and all $j\in I$, $u_{j',\ell'}(n) \geq 0$. That is, $Y_{j,\ell}^{-1}$ is the right-most uncancelled factor $Y^{-1}$ in $n$. 
If we are in type B and $j=N$ and either $p_{\T-t-1}$ or $p'_{\T-t-1}$ has an upper corner at $(N,\ell)$, we call this the exceptional case. 

Suppose we are not in the exceptional case. Then $(j,\ell)$ cannot be an upper corner of any path $p_{s}$ or $p'_{s}$ with $s< \T-t$. So $Y^{-1}_{j,\ell}$ is not cancelled by any of these paths. 
Next, $Y_{j,\ell}^{-1}$ is not cancelled in $\mon(p_s)$, $s>\T-t+1$, because if $p'_{s-1}=\psnake_{i_{s-1},k_{s-1};i_{s},k_{s}}$ then $\mon(p_s)=Y_{i_s,k_s}$ is already cancelled in $\mon(p'_{s-1})$, and if not then $p_s'=\phigh_{i_s,k_s} = p_{s}$ and $Y_{j,\ell}^{-1}$ still cannot be cancelled since $(j,\ell)$ is a corner of $p_{\T-t}$ or $p_{\T-t}'$, both strictly above $\phigh_{i_s,k_s}$. Finally $Y_{j,\ell}^{-1}$ cannot be cancelled in $\mon(p'_s)$, $s>\T-t$: if $(j,\ell)$ is a lower corner of $p'_{\T-t}$ then this is immediate from the non-overlapping property, and if $(j,\ell)$ is a lower corner of $p_{\T-t}$ then we use the fact that $p_{\T-t}$ is strictly above $p_{\T-t+1}=\phigh_{i_{\T-t+1},k_{\T-t+1}}$ and therefore also strictly above $p'_{\T-t+1}$. Thus $m'm$ has a factor $Y_{j,\ell}^{-1}$: a contradiction since $m'm$ is dominant. 

Now suppose we are in the exceptional case. Necessarily, $i_{\T-t}=i_{\T-t-1}=N$ and $k_{\T-t}-k_{\T-t-1}\equiv 2\!\! \mod 4$. 
Whichever path has an upper corner at $(j,\ell)$ (either $p_{\T-t-1}$ or $p'_{\T-t-1}$, or both) also has a lower corner at some $(j',\ell')\in \It$, $\ell'>\ell$. If there is more than one such lower corner, we choose the one for which $j'$ is maximal, i.e. the one closest to the spinor node. Neither  $p_{\T-t-1}$ nor $p'_{\T-t-1}$ can have an upper corner at $(j',\ell')$. Nor can $p_{\T-t}$ or $p'_{\T-t}$, by the condition on $k_{\T-t}-k_{\T-t-1}$. By the non-overlapping property, no path $p_s$ or $p_s'$, $s<\T-t-1$, has an upper corner at $(j',\ell')$. 
No path $p'_s$, $s>\T-t$, has an upper corner at $(j',\ell')$: if $(j',\ell')$ is a corner of $p'_{\T-t-1}$ then this is immediate from the non-overlapping property, and if $(j',\ell')$ is a corner of $p_{\T-t-1}$ then we use the fact that $p_{\T-t-1}$ is strictly above $p_{\T-t+1}=\phigh_{i_{\T-t+1},k_{\T-t+1}}$ and therefore also strictly above $p'_{\T-t+1}$. Next, the factor $Y_{j',\ell'}^{-1}$ is not cancelled in $\mon(p_s)$, $s>\T-t+1$, because if $p'_{s-1}=\psnake_{i_{s-1},k_{s-1};i_s,k_s}$ then $\mon(p_s) = Y_{i_s,k_s}$ is already cancelled and if not then $p'_{s-1}=\phigh_{i_{s-1},k_{s-1}}= p_{s-1}$, both strictly below $p_{\T-t-1}$ and $p'_{\T-t-1}$.  
It remains to check that $Y_{j',\ell'}$ cannot be cancelled by $\mon(p_{\T-t+1})=\mon(\phigh_{i_{\T-t+1},k_{\T-t+1}}) =Y_{i_{\T-t+1},k_{\T-t+1}}$, as follows. First note $\ell'=k_{\T-t-1}+2(N-j')-1$. But $(i_{\T-t+1},k_{\T-t+1})$ is in snake position to $(i_{\T-t},k_{\T-t})=(N,k_{\T-t})$, so $k_{\T-t+1} \equiv k_{\T-t} + 2(N-i_{\T-t+1})-1 \!\!\mod 4$ and therefore $k_{\T-t+1} \equiv k_{\T-t-1} + 2(N-i_{\T-t+1})+1\!\! \mod 4$ since $k_{\T-t}-k_{\T-t-1}\equiv 2\!\!\mod 4$. Hence $(j',\ell')\neq (i_{\T-t+1}, k_{\T-t+1})$. Thus $m'm$ has a factor $Y_{j',\ell'}^{-1}$: again, a contradiction since $m'm$ is dominant. 

Finally, that the fact that if $p'_{\T-t}=\phigh_{i_{\T-t},k_{\T-t}}$ then $p'_{\T-t-1} = \phigh_{i_{\T-t-1},k_{\T-t-1}}$ is clear from the non-overlapping property.
\end{proof}

\begin{prop}\label{dommons} 
Let $(i_t,k_t)\in \It$, $1\leq t \leq \T$, be a snake of length $\T\in \Z_{\geq 2}$. Let \be\scr U_{S} = \left\{ \prod_{t=1}^{R-1} m(\phigh_{i_t,k_t}) \prod_{t=R}^{\T-1} m(\psnake_{i_t,k_t;i_{t+1},k_{t+1}}): 1\leq R \leq \T \right\}.\ee
Then \begin{enumerate}[(i)]
\item\label{LxR} $\scr U_S \prod_{t=2}^\T\YY{i_t}{k_t}$ is the set of dominant monomials of $\chi_q(\L (\prod_{t=1}^{\T-1} \YY{i_t}{k_t} )\otimes \L(\prod_{t=2}^\T\YY{i_t}{k_t}))$.
\item\label{TxB} $\scr U_S \prod_{t=2}^\T \YY{i_t}{k_t} \setminus \left\{\prod_{t=1}^{\T-1} \mon(\psnake_{i_t,k_t;i_{t+1},k_{t+1}})\prod_{t=2}^\T \YY{i_t}{k_t} \right\}  $ is the set of dominant monomials of $\chi_q(\L ( \prod_{t=2}^{\T-1} \YY{i_t}{k_t})\otimes \L(\prod_{t=1}^\T \YY{i_t}{k_t} ) )$.
\end{enumerate} 
All of these dominant monomials occur with multiplicity one.
\end{prop}
\begin{proof}
The first part follows immediately from Lemma \ref{sll}, and the second by similar reasoning.
\end{proof}

\begin{prop}\label{MNsimpspec}
The module $\L(\prod_{(i,k)\in\nbri}Y_{i,k}) \otimes \L(\prod_{(i,k)\in\nbrii}Y_{i,k})$ is simple.  Moreover, it is special. 
\end{prop}
\begin{proof}
Special implies simple, so it is enough to show that $\L(\prod_{(i,k)\in\nbri}Y_{i,k}) \otimes \L(\prod_{(i,k)\in\nbrii}Y_{i,k})$ is special. For brevity let us write $\nbri_t := \nbri_{i_t,k_t}^{i_{t+1},k_{t+1}}$ and $\nbrii_t := \nbrii_{i_t,k_t}^{i_{t+1},k_{t+1}}$.
Proposition \ref{nbrsprop} states that $\nbri:=\nbri_1\concat\nbri_2\concat\dots\concat\nbri_{\T-1}$ and $\nbrii:=\nbrii_1\concat\nbrii_2\concat\dots\concat\nbrii_{\T-1}$ are snakes with no elements in common. 

Let $n$ be a monomial in  $\chi_q(\L(\prod_{(i,k)\in\nbri}Y_{i,k}))$ and $m$ a monomial in $\chi_q(\L(\prod_{(i,k)\in\nbrii}Y_{i,k}))$. By Theorem \ref{snakechar} we have 
\be n = \prod_{t=1}^{\T-1} \prod_{(i,k)\in\nbri_t} \mon(p_{(t;i,k)}),\quad\quad
    m = \prod_{t=1}^{\T-1} \prod_{(i,k)\in\nbrii_t} \mon(p_{(t;i,k)})\ee
for some paths $p_{(t;i,k)}$, with $p_{(t;i,k)}\in \scr P_{i,k}$ for each $1\leq t<\T$ and each $(i,k)\in \nbri_t \sqcup \nbrii_t$.

We would like to show that if $nm$ is not highest then it is not dominant. So suppose $nm$ is not highest, i.e. that at least one of the paths $p_{(t;i,k)}$, $1\leq t\leq \T-1$, $(i,k)\in \nbri_t\sqcup \nbrii_t$ is not highest. Then there exists a $(j,\ell)\in \It$ such that $(j,\ell)$ is a lower corner of at least one of the paths and none of the paths has a lower corner at any point $(j',\ell')\in \It$ with $\ell'>\ell$. 

Let us suppose for definiteness that $(j,\ell)$ is a lower corner of a path $p_{(t;i,k)}$ with $(i,k)\in \nbri_t$. If instead  $(i,k)\in \nbrii_t$ the argument is similar. 

By choice of $(j,\ell)$ and Lemma \ref{rneglem}, no path $p_{(t';i',k')}$, $(i',k')\in \nbri_{t'}$, has an upper corner at $(j,\ell)$. We shall now argue that no path $p_{(t';i',k')}$, $(i',k')\in \nbrii_{t'}$ has an upper corner at $(j,\ell)$ either. By choice of $(j,\ell)$, if $p_{(t';i',k')}$ has an upper corner at $(j,\ell)$ then $p_{(t';i',k')}=\phigh_{i',k'}$. Thus it is enough to show that $(j,\ell)\notin \nbrii$. 

Suppose $(j',\ell')\in \nbrii$ is weakly above $\phigh_{i,k}$. Then no path in $\scr P_{i,k}$ has a lower corner at $(j',\ell')$, so $(j',\ell')\neq (j,\ell)$.

Suppose $(j',\ell')\in \nbrii$ is strictly below $\plow_{i,k}$. Then clearly $(j',\ell')\neq (j,\ell)$.

Suppose $(j',\ell') \in \nbrii$ is weakly below $\phigh_{i',k'}$ for some point $(i',k')\in \nbri_{t'}$ that succeeds $(i,k)$ in the snake $\nbri$. Let us say such a point $(j',\ell')$ is \emph{shielded}. In that case, if $p_{(t;i,k)}$ has a lower corner at $(j',\ell')$ then $p_{(t';i',k')}$ must have a lower corner at some $(j'',\ell'')$ with $\ell''>\ell'$. So $(j',\ell')\neq (j,\ell)$.

In type A, this exhausts all the elements of $\nbrii$. In type B there is one further possibility: there can be at most one point $(N,\ell')\in \nbrii_{t+1}$ that is strictly below $\phigh_{i,k}$ and yet is neither shielded nor strictly below $\plow_{i,k}$. But then $\ell'\mod 4$ is always such that no path in $\scr P_{i,k}$ has a lower corner at $(N,\ell')$. 

Thus indeed none of the paths has an upper corner at $(j,\ell)\in \It$. Hence $nm$ has a factor $Y_{j,\ell}^{-1}$ and is not dominant, as required.
\end{proof}

\subsection{Exclusion argument}
\begin{prop}\label{TBsimple}
The module $\L (\prod_{t=2}^{\T-1} \YY{i_t}{k_t})\otimes \L(\prod_{t=1}^\T \YY{i_t}{k_t} ) $ is simple.
\end{prop}
\begin{proof}
We shall show that if $n$ is any non-highest dominant monomial in $\chi_q(\L  (\prod_{t=2}^{\T-1} \YY{i_t}{k_t})\otimes \L(\prod_{t=1}^\T \YY{i_t}{k_t}))$  then $\chi_q (\L(n))$ contains a monomial not present in $\chi_q(\L  (\prod_{t=2}^{\T-1} \YY{i_t}{k_t})\otimes \L(\prod_{t=1}^\T \YY{i_t}{k_t}))$. This is sufficient, for then $\L(n)$ cannot appear in the composition series of $\L (\prod_{t=2}^{\T-1} \YY{i_t}{k_t})\otimes \L(\prod_{t=1}^\T \YY{i_t}{k_t})$, which means that the only entry of this series is  $\L \left(\prod_{t=2}^{\T-1} \YY{i_t}{k_t} \prod_{t=1}^\T \YY{i_t}{k_t}\right)$, i.e. that indeed the module is simple. 

The dominant monomials are catalogued in Proposition \ref{dommons} part (\ref{TxB}). Consider 
\be n:=\prod_{t=1}^\T \m(\phigh_{i_t,k_t}) \prod_{t=2}^{R-1} \mon(\phigh_{i_t,k_t}) 
                                    \prod_{t=R}^{\T-1} \mon(\psnake_{i_t,k_t;i_{t+1},k_{t+1}}) \label{ndef}\ee
for any given $R$ with $2\leq R\leq \T-1$. We have
\be n = Y_{i_1,k_1} \,\,\cdot\,\,\prod_{t=2}^{R-1} Y_{i_t,k_t}^2 \,\,\cdot\,\, Y_{i_R,k_R} \,\, \cdot\prod_{t=R}^{\T-1} \prod_{\substack{(i,k)\in\\ C_{\psnake_{i_t,k_t;i_{t+1},k_{t+1}}}^{+}}} Y_{i,k}.\ee
This monomial can be written as a product of highest paths, $\phigh_{i,k}$ for each factor $Y_{i,k}$ that appears in this product. In contrast to snake modules, some of these paths overlap. Nonetheless, we can use Theorem \ref{thmA} to compute enough of the $q$-character for our purposes. 

To do so, let $U\subset \Iy$ be the set of points such that performing lowering moves at these points, in some order, on $\phigh_{i_r,k_R}$ yields $\psnake_{i_R,k_R;i_{R+1},k_{R+1}}$. Lemma \ref{movelemmaA} states that such a set of points exists. See for example Figure \ref{Udeffig}. Define
\be \mc M := \left\{ n \mon(\phigh_{i_R,k_R})^{-1} \mon(p): p\in\scr P\right\}\nn\ee
where
\be\nn\scr P := \left\{ p\in \scr P_{i_R,k_R}: p \text{ is weakly above } \psnake_{i_R,k_R;i_{R+1},k_{R+1}}\right\}. \nn\ee
We now argue that $(n,\mc M, U)$ satisfy the conditions of Theorem \ref{thmA}.
Recall that we are writing $(i,k)$ for $(i,aq^k)$. Property (\ref{incone}) is by definition of $\mc M$. Property (\ref{monlydom}) is true because every path $p\in \scr P$ has a lower corner at some point in $\{(j,\ell)\in \It: (j,\ell-r_j)\in U\}$ and $u_{j,\ell}(n)=0$ for all such points. Note in particular that $u_{i_{R+1},k_{R+1}}(n)=0$. Property (\ref{onewayback}) is by inspection. For Property (\ref{inthinsimple}) we have to show that for every $m\in \mc M$ and every $i\in I$ there exists a unique $i$-dominant monomial in $m\Z[A_{i,k}]_{k\in\Z}\cap \mc M$, say $M_i$, such that 
\be \trunc_{\beta_i(M_i \Z[A^{-1}_{j,\ell}]_{(j,\ell)\in U})}\chi_q(\L(\beta_i(M_i))) = \sum_{m'\in m\Z[A_{i,k}]_{k\in\Z} \cap \mc M} \beta_i(m') .\label{req}\ee
It follows from Lemma \ref{movelemma} (case $\T=1$) that there is a unique path $p\in\scr P$ such that $m=n \mon(\phigh_{i_R,k_R})^{-1} \mon(p)$ and further that the only way to produce a monomial in $m\Z[A_{i,k}]_{k\in\Z}$ is to perform the corresponding raising moves at points $(i,k-r_i)$, $k\in \Z$, on this path. Thus, if $p$ has no lower corner of the form $(i,k)$, $k\in \Z$, then $M_i=m$ and this is the only term on either side of (\ref{req}). If $p$ has one lower corner of the form $(i,k)$, $k\in \Z$, then $M_i=mA_{i,k-r_i}$, and both sides of  (\ref{req}) are equal to $M_i(1+A^{-1}_{i,k-r_i})$. Finally, in type $B$ when $i=N$ and $i_R<N$, it is possible that $p$ has lower corners at $(N,k)$ and $(N,k+2)$ for some $k\in \Z$. In that case  $M_i=mA_{N,k+1}A_{N,k-1}$ and both sides of (\ref{req}) are equal to $M_i(1+A^{-1}_{N,k+1} + A^{-1}_{N,k+1}A^{-1}_{N,k-1})$. 

Therefore Theorem \ref{thmA} applies, and we have 
\be\nn \trunc_{n\Z[A_{j,\ell}^{-1}]_{(j,\ell)\in U}} \chi_q(\L (n)) 
   = n \mon(\phigh_{i_R,k_R})^{-1} \sum_{p\in\scr P} \mon(p).\ee
In particular $\chi_q(\L(n))$ includes the monomial $m':= n \mon(\phigh_{i_R,k_R})^{-1} \mon(\psnake_{i_R,k_R;i_{R+1},k_{R+1}})$. 

On the other hand, $m'$ is not a monomial of $\chi_q(\L ( \prod_{t=2}^{\T-1} \YY{i_t}{k_t})\otimes \L(\prod_{t=1}^\T \YY{i_t}{k_t}))$. Indeed, suppose for a contradiction that $m'$ \emph{is} a monomial of $\chi_q(\L ( \prod_{t=2}^{\T-1} \YY{i_t}{k_t})\otimes \L(\prod_{t=1}^\T \YY{i_t}{k_t}))$. Then, by Theorem \ref{snakechar}, $m'=\prod_{t=2}^{\T-1} \mon(p_t)\prod_{t=1}^\T \mon(p'_t)$ for some unique $(p_2,\dots, p_{\T-1}) \in \overline{\scr P}_{(i_t,k_t)_{2\leq t\leq \T-1}}$ and $\pps\in \nops$. Call these the ``inner'' and ``outer'' tuples, respectively. By Lemma \ref{movelemma} the inner and outer tuples are obtained from the inner and outer tuples of $n$, namely $(\phigh_{i_2,k_2},\dots, \phigh_{i_{R-1},k_{R-1}},\psnake_{i_R,k_R;i_{R+1},k_{R+1}}, \dots, \psnake_{i_{\T-1},k_{\T-1};i_\T,k_\T})$ and $\ptop$, \confer (\ref{ndef}), by performing some sequence of moves, such that no inverse pair of moves is performed on either tuple. In particular, only lowering moves can be performed on the outer tuple, since we start with $\ptop$. 
In order to reach $m'$ from $n$, any lowering move at a point not in $U$ must be cancelled by a raising move on the other tuple of paths, and at every point in $U$ we must perform exactly one more lowering move than raising move. Consequently, on the inner tuple we cannot perform any lowering move not in $U$. But, by inspection, it is not possible to lower the inner tuple at any point in $U$ either. So we must perform, on the outer tuple, at least one lowering move at every point in $U$. Let $K\geq 1$ be the number of lowering moves performed on the outer tuple at the point $(i_{R+1},k_{R+1}-r_{i_{R+1}})\in U$. To avoid overlaps, we must also lower the outer tuple $K'\geq K$ times at the point $(i_{R+1},k_{R+1}+r_{i_{R+1}})\notin U$. Each of the latter moves must be cancelled by a raising move at the same point on the inner tuple. But, again, this creates an overlap unless we also perform at least $K'$ raising moves at the point $(i_{R+1},k_{R+1}-r_{i_{R+1}})\in U$ on the inner tuple. So the net number of lowering moves at this point is non-positive -- a contradiction, since it must be exactly 1. 
\end{proof}

\begin{figure}
\caption{\label{Udeffig} See the proof of Proposition \ref{TBsimple}. In type A, with $\T=6$, suppose for example that the highest monomial $\prod_{t=2}^{\T-1} \YY{i_t}{k_t}\prod_{t=1}^\T \YY{i_t}{k_t} $ is as shown on the left below.  Single (resp. double) black circles indicate factors $Y$ (resp $Y^2$). The dominant monomial $n$ corresponding to $R=4$ is shown on the right. Lowering $n$ at the points in $U$, as marked, produces a monomial in $\chi_q(\L(n))$ not present in $\chi_q(\prod_{t=2}^{\T-1} \YY{i_t}{k_t})\chi_q(\prod_{t=1}^\T \YY{i_t}{k_t})$. }
\be \begin{tikzpicture}[scale=.35,yscale=-1]
\draw[help lines] (0,0) grid (10,17);
\draw (1,-1) -- (9,-1);
\draw[dashed] (-1,-1) --(1,-1); \draw[dashed] (9,-1) -- (11,-1);
\foreach \x in {1,2,3,4,5,6,7,8,9} {\filldraw[fill=white] (\x,-1) circle (2mm); }
\foreach \x/\y in {    4/2,5/5,5/9,4/14     } {\filldraw[thick,double] (\x,\y) circle (2mm);}
\foreach \x/\y in {4/0,                 3/17} {\filldraw (\x,\y) circle (2mm);}
\end{tikzpicture}
\begin{tikzpicture}[scale=.35,yscale=-1]
\draw[help lines] (0,0) grid (10,17);
\draw (1,-1) -- (9,-1);
\draw[dashed] (-1,-1) --(1,-1); \draw[dashed] (9,-1) -- (11,-1);
\foreach \x in {1,2,3,4,5,6,7,8,9} {\filldraw[fill=white] (\x,-1) circle (2mm); }
\foreach \x/\y in {    4/2,5/5     } {\filldraw[thick,double] (\x,\y) circle (2mm);}
\foreach \x/\y in {4/0,        5/9 } {\filldraw (\x,\y) circle (2mm);}
\foreach \x/\y in {2/12, 2/16     } 
{\filldraw (\x,\y) circle (2mm);}
\foreach \x/\y in {7/11, 5/15     } 
{\filldraw (\x,\y) circle (2mm);}
\foreach \x/\y in {5/10,6/11,4/11,3/12,5/12,4/13} {\draw (\x,\y) circle (2mm);}
\fill[opacity=.1] (2,12) -- (5,9) -- (7,11) -- (4,14);
\node at (4.5,11.5) {$U$};
\end{tikzpicture}\nn\ee

\end{figure}
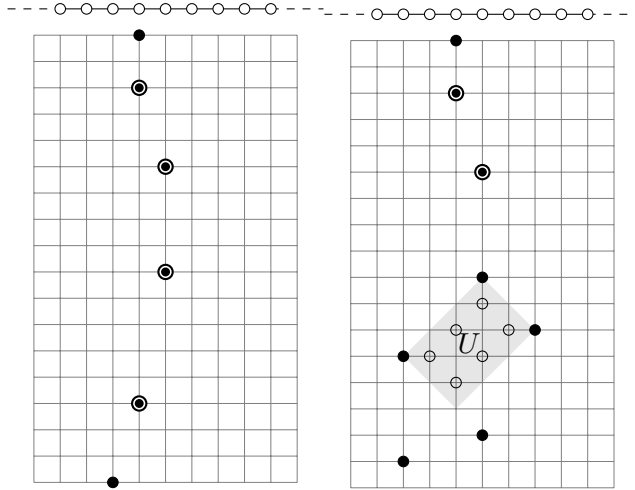

Finally, we can prove Theorem \ref{Tsys}.
\proof[Proof of Theorem \ref{Tsys}] 
An element of $\chi_q(\groth\uqgh)$ is determined uniquely its dominant monomials. Therefore
the equality (\ref{grl}) follows from Propositions \ref{dommons} and \ref{MNsimpspec}.
Finally, $\L (\prod_{t=2}^{\T-1} Y_{i_t,k_t}) \otimes \L(\prod_{t=1}^\T Y_{i_t,k_t})$ and $\L(\prod_{(i,k)\in\nbri}Y_{i,k}) \otimes \L(\prod_{(i,k)\in\nbrii}Y_{i,k})$ are simple, by Propositions \ref{TBsimple} and \ref{MNsimpspec} respectively. 
 \qed

\begin{appendix}
\section{Examples in Types C and D}\label{typeCD}
As noted in the introduction, we believe that extended T-systems exist in all Dynkin types. Such recursions should allow the classes of (at least) all minimal affinizations to be expressed in terms of the classes of Kirillov-Reshetikhin modules (and hence, by means of the usual T-system, in terms of the classes of fundamental modules). 
Here we discuss some examples which illustrate  new features that arise beyond types A and B. 

It should be stressed that all relations in this section are conjectural. 
We have checked, with the aid of a computer, that they are correct \emph{if} one assumes that the algorithm of \cite{FM} gives the correct $q$-character for all of the irreducible representations that appear.

Since we deal with specific examples only, it is convenient in this appendix to use the shorthand notation $1_0=Y_{1,0}$, $2_3= Y_{2,3}$ and so on. Also, we write $[m]$ for the class of $\L(m)$ in $\groth\uqgh$.

\subsection*{Type $C_3$} We use the node numbering \be\nn
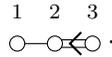
\tikz[baseline =0,scale=.5]{\draw (0,0) -- ++(1,0);  \draw (2,0.1) -- (1,.1); \draw (2,-.1) -- (1,-.1); \draw[thick] (1.7,.3) -- (1.4,0) -- (1.7,-.3);
\filldraw[fill=white] (0,0) circle (.2) node[above=2mm] {$\scriptstyle 1$};
\filldraw[fill=white] (1,0) circle (.2) node[above=2mm] {$\scriptstyle 2$};
\filldraw[fill=white] (2,0) circle (.2) node[above=2mm] {$\scriptstyle 3$};
}.\ee
Thus $r_1=r_2=1$, $r_3=2$. 
Let us consider choosing as the top module the minimal affinization $\L(3_0 2_5 2_7)$. It is clear that the leftmost and rightmost factors should be $3_{0}$ and $2_{7}$ respectively. Computations then suggest that
\be\nn [3_0 2_5] [ 2_5 2_7]
= [3_0 2_5 2_7][2_5]
 + [2_1 1_4 1_6 3_6].\ee
A single neighbour, $\L(2_1 1_4 1_6 3_6)$, is generated. It is not a minimal affinization, so, as in type B, the recursion does not close among minimal affinizations. But in type B it was always possible to interpret the neighbours as ``snakes'': that is, to read the factors of the dominant monomial in order of increasing shift (the lower index), such that, in particular, no non-neighbouring pair defines a minimal affinization.  Here, both $\L(2_11_4)$ and $\L(2_13_6)$ are minimal affinizations, so the structure is rather to be thought of as pictured in Figure \ref{sjfig}. 
\begin{figure}\caption{\label{sjfig}Illustrative examples of highest monomials of the new types of module occuring in types $C_3$ (left pair) and $D_5$ (right pair).}
\be\begin{tikzpicture}[scale=.7] \node {$2_1$} child {node {$1_4$} child {node {$1_6$}}}  child[level distance =30mm]  {node {$3_6$}}; 
\begin{scope}[xshift=30mm,yshift=30mm] \node {$2_{-3}$} child{ node {$2_{-1}$} child{ node {$2_1$} child {node {$1_4$} child {node {$1_6$} child{node {$1_8$} child {node {$1_{10}$}}}}} child[level distance = 30mm] {node {$3_6$} child {node {$3_{10}$}}}}}; \end{scope}
\begin{scope}[xshift=60mm]  \node {$3_1$} child{node {$4_4$}} child{node {$5_4$}}; 
\end{scope}
\begin{scope}[xshift=90mm,yshift=15mm]  \node {$3_{-1}$} child{ node {$3_1$} child{node {$4_4$}} child{node {$5_4$}}}; 
\end{scope}
\end{tikzpicture}
\nn\ee  

\end{figure}
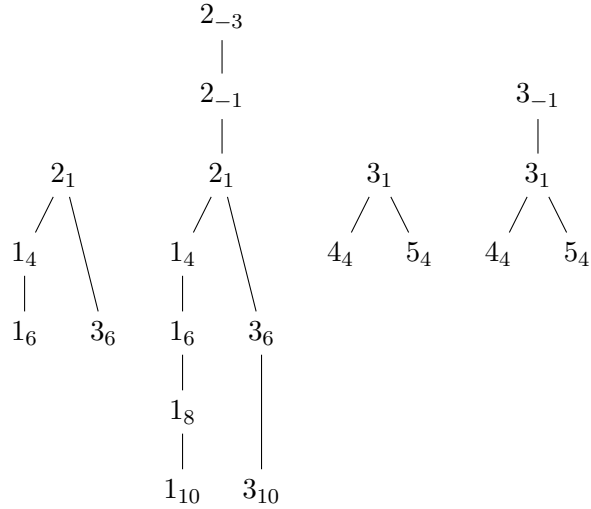

Thus, when we come to treat $\L(2_1 1_4 1_6 3_6)$ in turn as the top module, we should certainly take $2_{1}$ as the leftmost factor, but it is unclear whether $1_{6}$ or $3_{6}$ is rightmost. However, it appears that either choice works. For example, if we pick $3_6$ we find the relation
\be\nn [2_1 1_4 1_6] [1_41_6 3_6] = [2_1 1_4 1_6 3_6] [1_4 1_6] + [2_5] [1_2 1_4 1_6] [1_4 1_6]. \ee
Note that there are 3 neighbours, not 2. This is also a feature of the usual T-system in type C. To complete the recursion, for $[2_1 1_4 1_6]$ we have the relation 
\be\nn [2_1 1_4] [1_4 1_6] = [2_1 1_4 1_6] [1_4] + [2_5][3_2], \ee
while $[1_41_6 3_6]$ actually factors, $[1_4 1_6 3_6] = [1_4 1_6] [3_6]$. Thus
we have succeeded in expressing $[2_1 1_4 1_6 3_6]$, and hence the original module $[3_0 2_5 2_7]$, in terms of the classes of Kirillov-Reshetikhin modules. 

It appears that larger examples work in the same way. If we start with the minimal affinization $\L(3_{-4} 3_0 2_5 2_7 2_9 2_{11})$, we find 
\be\nn [3_{-4} 3_0 2_5 2_7 2_9] [ 3_0 2_5 2_7 2_9 2_{11}]
= [3_{-4} 3_0 2_5 2_7 2_9 2_{11}][ 3_0 2_5 2_7 2_9]
 + [3_8] [2_{-3} 2_{-1} 2_1 1_4 1_6 1_8 1_{10} 3_6 3_{10}].\ee
In the neighbouring module $\L(2_{-3} 2_{-1} 2_1 1_4 1_6 1_8 1_{10} 3_6 3_{10})$, the factor $2_1$ is ``trivalent'', c.f. the picture in Figure \ref{sjfig}.
Choosing $3_{10}$ to be the rightmost factor,
\bea [2_{-3} 2_{-1} 2_1 1_4 1_6 1_8 1_{10} 3_6 ]
    [       2_{-1} 2_1 1_4 1_6 1_8 1_{10} 3_6 3_{10}]
&=&[2_{-3} 2_{-1} 2_1 1_4 1_6 1_8 1_{10} 3_6 3_{10}]
[  2_{-1} 2_1 1_4 1_6 1_8 1_{10} 3_6] \nn\\ &&+ [3_0 2_5 2_7 2_9] [3_{-2} 1_4 1_6 1_8 1_{10}] [1_{-2} 1_0 1_2 1_4 1_6 1_8 1_{10}]\nn \eea
and so on. Observe that the neighbours generated here are all, once more, minimal affinizations; in particular they do not have any ``trivalent'' factors.

\subsection*{Type $D_5$}
We use the node numbering \be\nn
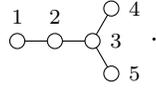
\begin{tikzpicture}[baseline=0pt,scale=.5] \draw (1,0) -- ++(2,0) -- ++(60:1) ++(60:-1) -- ++(-60:1) ;    
\filldraw[fill=white] (1,0) circle (.2cm) node [above=1mm] {$\scriptstyle 1$};    
\filldraw[fill=white] (3,0) circle (.2cm) node [right=1mm] {$\scriptstyle 3$};    
\filldraw[fill=white] (2,0) circle (.2cm)  node [above=1mm] {$\scriptstyle 2$};    
\filldraw[fill=white] (3,0)++(60:1) circle (.2cm)  node [right=1mm] {$\scriptstyle 4$};    
\filldraw[fill=white] (3,0)++(-60:1) circle (.2cm) node [right=1mm] {$\scriptstyle 5$};    
\end{tikzpicture}.\ee    
It appears that type D works similarly to type C. Consider the minimal affinization $\L(2_0 2_2 3_5)$. We find
\be\nn [2_0 2_2 ] [2_2 3_5]  =[2_0 2_2 3_5] [2_2] + [1_11_3] [3_14_45_4] \ee
and then if we choose $5_4$ as the rightmost factor of $3_14_45_4$, we find
\be\nn [3_1 4_4][4_4 5_4]= [3_1 4_4 5_4][4_4 ] + [2_2] [4_24_4] [4_4]. \ee
Here $[4_4,5_5]$ factors, $[4_4 5_4] = [4_4] [5_4]$, so the recursion is complete. 
A more generic starting point is the minimal affinization $\L(2_{-2} 2_0 2_2 3_5)$: we find
\be\nn [2_{-2} 2_0 2_2 ] [2_0 2_2 3_5]  =[2_{-2} 2_0 2_2 3_5] [2_0 2_2] + [1_{-1} 1_11_3] [3_{-1} 3_14_45_4] \ee
and then 
\be\nn [3_{-1} 3_1 4_4][3_1 4_4 5_4]= [3_{-1} 3_1 4_4 5_4][3_1 4_4 ] + [2_02_2] [4_04_24_4] [5_04_4]. \ee
Here there are no common factors. Note that $\L(5_0 4_4)$ is a minimal affinization.
Treating it as the top module, we find
\be\nn [5_0] [4_4] = [5_04_4] + [2_2]. \ee
Noting finally that
\be\nn [3_{-1} 3_1] [3_1 4_4] = [3_{-1} 3_1 4_4] [3_1] + [2_02_2] [4_0] [5_0 5_2],\ee
we have all the relations needed to express $[2_{-2} 2_0 2_2 3_5]$ in terms of the classes of Kirillov-Reshetikhin modules. 

\section{On Thin Special Truncated $q$-characters}\label{proofA}
In this Appendix we prove Theorem \ref{thmA}. We shall need the following two results.
\begin{prop}[\cite{MY1}]\label{Aprop}
Let $V$ be a finite-dimensional $\uqgh$-module and $m,m'\in\mchiq(V)$. Let $\ket m\in \ker\left(\phi^\pm_i(u)-\gamma(m)^\pm_i(u)\right)\subset V_m$ for all $i\in I$. Then, for all $j\in I$, at least one of the following holds: 
\begin{enumerate}
\item there is an $a\in\Cx$ such that $m'=mA_{j,a}$ (resp. $m'=mA_{j,a}^{-1}$), or
\item for all $r\in\Z$, when $x^+_{j,r}\on \ket m$ (resp. $x^-_{j,r}\on \ket m$) is decomposed into $l$-weight spaces, \confer (\ref{lwdecomp}), its component in  $V_{\bs\gamma(m')}\equiv V_{m'}$ is zero.\qed
\end{enumerate}
\end{prop}
\begin{lem}\label{thinoverXsl2lem}
Let $\g=\mf{sl}_2$. Let $m\in \P$. Let $U\subset \{1\}\times \Cx$. Let 
 $M\in \P^+$ be such that $\L(M)$ is thin in $U$ and $m$ is a monomial in $\trunc_{M\Q^-_U}\chi_q(\L(M))$.
Then exactly one of $(i),(ii),(iii)$ holds for all $(1,a)\in U$: 
\begin{align*}
&(i)  &0&< u_{1,aq^{-1}}(m) = u_{1,aq}(m)+1 ,  &mA_{1,a}^{-1}&\in\mchiq(\L(M)) ;\\
&(ii) &0&< u_{1,aq^{-1}}(m) \leq  u_{1,aq}(m), &mA_{1,a}^{-1}&\in\mchiq(W(M)) \setminus \mchiq(\L(M)) ; \\
&(iii) &0&\geq u_{1,aq^{-1}}(m), &mA_{1,a}^{-1}&\notin \mchiq(W(M)).  
\end{align*} 
\end{lem}
\begin{proof}
The result follows from the known closed forms of all Weyl modules \cite{CPweyl} and simple modules, \cite{CPsl2}, in type $A_1$. 
\end{proof}

\begin{proof}[Proof of Theorem \ref{thmA}]
For convenience, let us define
\be \chi_n := \trunc_{m_+ \Q^-_{U,(=n)}}\chi_q(\L (m_+))
 \qquad \mc M_n := \trunc_{m_+ \Q^-_{U,(=n)}} (\mc M),\ee
and similarly $\chi_{\leq n}$ and $\mc M_{\leq n}$. 

Given property (\ref{incone}), to prove the equality (\ref{headch}) it is sufficient to establish the following claim for all $n\in\Z_{\geq 0}$.

\emph{Claim:}
\be \chi_n = \sum_{m\in \mc M_n}m.\label{claim}\ee
We proceed by induction on $n$. The claim is true for $n=0$, given Property (\ref{monlydom}).  
So for the inductive step, let $n\in \Z_{>0}$ and suppose the claim is true for $n-1$. 

It follows from Proposition \ref{Aprop} that all monomials $m'$ of $\chi_n$ are of the form $m A_{i,a}^{-1}$ for some monomial $m$ of $\chi_{n-1}$ and some $(i,a)\in I\times \Cx$. For suppose not: then by Proposition \ref{Aprop}, $\L(m_+)_{m'}$ contains a highest $l$-weight vector and so generates a proper submodule -- a contradiction. Therefore by the inductive hypothesis all monomials $m'$ in $\chi_n$ are of the form $m A_{i,a}^{-1}$ for some $m\in\mc M_{n-1}$ and some $(i,a)\in I\times \Cx$. By definition of $\chi_n$, $(i,a)\in U$.

Property (\ref{monlydom}) implies that if $m'\in \mc M_n$ then there exists an $i\in I$ such that $m'$ is not $i$-dominant. Hence Properties (\ref{incone}) and (\ref{inthinsimple}) together imply that every $m'\in\mc M_n$ is also of the form $mA_{i,a}^{-1}$ for some $m\in \mc M_{n-1}$ and some $(i,a)\in U$. 

Consequently, it is enough to consider monomials of the form  $mA_{i,a}^{-1}$ for some $m\in \mc M_{n-1}$ and some $(i,a)\in U$.

Let $M$ be the unique $i$-dominant monomial in $m \Q^+_{U\cap (\{i\}\times \Cx),(>0)}\cap \mc M$. Property (\ref{inthinsimple}) asserts that such an $M$ exists (possibly $m=M$) and  moreover that the simple $\uqsl i$-module $\L(\beta_i (M))$ is thin in $U$.  Lemma \ref{thinoverXsl2lem} therefore implies that exactly one of the following three cases applies.
\begin{enumerate}[(I)]
\item $0< u_{i,aq^{-r_i}}(m) = u_{i,aq^{r_i}}(m) +1$, and $\beta_i(mA_{i,a}^{-1})\in \mchiq(\L(\beta_i(M)))$.
\item $0< u_{i,aq^{-r_i}}(m) \leq u_{i,aq^{r_i}}(m)$, and $\beta_i(mA_{i,a}^{-1})\in\mchiq(W(\beta_i(M))) \setminus \mchiq(\L(\beta_i(M)))$.
\item $0\geq u_{i,aq^{-r_i}}(m)$ and $\beta_i(mA_{i,a}^{-1})\notin\mchiq(W(\beta_i(M)))$.
\end{enumerate}
We shall now complete the inductive step by showing that in case (I), $mA_{i,a}^{-1}$ appears, with coefficient exactly 1, on both sides of (\ref{claim}), while in cases (II) and (III), $mA_{i,a}^{-1}$ does not appear on either side of (\ref{claim}). 

First observe that by the inductive assumption, 
\be\nn \trunc_{mA_{i,a}^{-1} \Q^+_{U\cap (\{i\}\times \Cx) ,(>0)}}  \mchiq(\L(m_+))
     = \trunc_{mA_{i,a}^{-1} \Q^+_{U\cap (\{i\}\times \Cx) ,(>0)}}  \mc M \ee 
and by Properties (\ref{incone}) and (\ref{inthinsimple}), $M$ is the unique $i$-dominant monomial in this set.

Now consider case (I). By Property (\ref{inthinsimple}), $mA_{i,a}^{-1}\in\mc M$ (note that $\beta_i$ is injective when restricted to $m\Q_{\{i\}\times\Cx}$). By the injectivity and property (\ref{tauj})  of the homomorphism $\tau_i$ we deduce that $m A_{i,a}^{-1}$ must be a monomial of $\chi_q(\L (m_+))$. The coefficient of $\beta_i(mA_{i,a}^{-1})$ in $\chi_q(\L(\beta_i(M)))$ is 1, so therefore $mA_{i,k}^{-1}$ must have coefficient exactly 1 in $\chi_q(\L(m_+))$, for if it appeared with a larger coefficient it would not be part of a consistent $\uqsl i$-character.

Now consider case (II). By Property (\ref{inthinsimple}), $mA_{i,a}^{-1}\notin \mc M$. Hence by Property (\ref{onewayback}), $mA_{i,a}^{-1} A_{j,b}\notin \mc M$ unless $(j,b)=(i,a)$. If $mA_{i,a}^{-1} A_{j,b}$ is not in $m_+ \Q^-$ then it is not in $\mchiq(\L(m_+))$ by (\ref{imchiq}). If $mA_{i,a}^{-1} A_{j,b}$ is in $m_+\Q^-$ then $v(mA_{i,a}^{-1} A_{j,b})=n-1$ and we can use the inductive assumption. Therefore 
\be\label{nin}mA_{i,a}^{-1} A_{j,b}\notin \mchiq(\L(m_+))\quad\text{ unless }\quad  (j,b) = (i,a).\ee  
Now suppose, for a contradiction, that $m':=mA_{i,a}^{-1}\in\mchiq(\L(m_+))$.  Then we can pick a non-zero $\ket{m'}$ such that for all $i\in I$, $\ket{m'}\in \ker(\phi_i^\pm(u)-\gamma_i^\pm(m')(u))\subseteq \L (m_+)_{m'}$. By Proposition \ref{Aprop}, for all $r\in \Z$, $x^+_{j,r} \on \ket{m'} = 0$ for all $j\neq i$, and $x^+_{i,r}\on\ket{m'}\in (L m_+)_m$. If $x^+_{i,r}\on \ket {m'} =0$ for all $r\in \Z$ then $\ket {m'}$ generates a proper submodule in $\L (m_+)$: a contradiction since $\L (m_+)$ is simple. So for some $r\in \Z$, $x^+_{i,r} \on \ket{m'}$ is non-zero and spans the one-dimensional (by the inductive assumption) $l$-weight space $\L (m_+)_m$. 
Therefore $\ket{m'}\notin \Span_{r\in \Z} x_{i,r}^- (\L (m_+)_m)$, because $\L(\beta_i(M))$ is by definition irreducible, $\beta_i(m)$ appears in its $q$-character, and $\beta_i(m')$ does not. Now, if $m''\in \mchiq(\L(m_+))$ and $j\in I$ are such that $\ket{m'}\in \Span_{r\in \Z} x_{j,r}^{-}\on \L(m_+)_{m''}$ then $\wt(m'') = \wt(m')+e^{\alpha_j}$ and hence $v(m'' m_+^{-1})=n-1$. So, by the inductive assumption, $\dim(\L(m_+)_{m''})=1$, and thus Proposition \ref{Aprop} applies to any $\ket{m''}\in \L(m_+)_{m''}$. Hence, by (\ref{nin}),  $x_{j,r}^-\ket{m''}$ has zero component in $\L(m_+)_{m'}$ for all $m''\neq m$. So $\ket{m'}\notin \Span_{i\in I,r\in \Z} x_{i,r}^- (\L (m_+))$: a contradiction. Therefore in fact $m'$ is not a monomial in $\chi_q(\L (m_+))$.

Finally consider case (III). By Property (\ref{inthinsimple}), $mA_{i,a}^{-1}\notin\mc M$. Now $mA_{i,a}^{-1}$ is not $i$-dominant since $u_{i,aq^{-r_i}}(mA_{i,a}^{-1}) =-1$. But $\beta_i(mA_{i,a}^{-1})$ is not in the $q$-character of the Weyl module $W(\beta_i(M))$, and, recall, $M$ is the unique $i$-dominant monomial in $mA_{i,a}^{-1} \Q^+_{U\cap (\{i\}\times \Cx),(>0)}$. Therefore $mA_{i,a}^{-1}$ cannot appear in $\chi_q(\L(m_+))$ because it is not part of a consistent $\uqsl i$-character.

This completes the inductive step, and we have therefore established the claim above, for all $n\in\Z_{\geq 0}$, and hence the equality (\ref{headch}). Finally, it follows that $\L(m_+)$ is manifestly thin in $U$, and it is special in $U$ by Property (\ref{monlydom}).
\end{proof}

\end{appendix}


\def\cprime{$'$}
\providecommand{\bysame}{\leavevmode\hbox to3em{\hrulefill}\thinspace}
\providecommand{\MR}{\relax\ifhmode\unskip\space\fi MR }
\providecommand{\MRhref}[2]{%
  \href{http://www.ams.org/mathscinet-getitem?mr=#1}{#2}
}
\providecommand{\href}[2]{#2}

\end{document}